\documentclass[12pt]{article}
\usepackage[utf8]{inputenc}
\usepackage[title]{appendix}
\usepackage{comment}

\title{Singularity of Cannon--Thurston maps}
\author{Vaibhav Gadre, Joseph Maher, Catherine Pfaff, Caglar Uyanik }

\date{\today}

\usepackage[top=1.25in,bottom=1.25in,left=1in,right=1in]{geometry}
\usepackage[dvipsnames]{xcolor}
\usepackage{amsfonts, amsmath, amssymb, amscd, graphicx}
\usepackage{amsthm}
\usepackage{thmtools}
\usepackage[final, 
colorlinks=true, linkcolor=purple, citecolor=blue, unicode]{hyperref}
\usepackage{cleveref}
\usepackage{enumitem}

\usepackage[nobysame,alphabetic]{amsrefs}
\usepackage{xspace}

\usepackage{tikz}
\usetikzlibrary{matrix,arrows,decorations.markings,decorations.pathmorphing,patterns}
\usetikzlibrary{calc}

\theoremstyle{plain}
\newtheorem{theorem}{Theorem}
\newtheorem{lemma}[theorem]{Lemma}
\newtheorem{proposition}[theorem]{Proposition}
\newtheorem{corollary}[theorem]{Corollary}

\newtheorem*{claim*}{Claim}

\theoremstyle{definition}
\newtheorem{definition}[theorem]{Definition}

\newtheorem{remark-convention}[theorem]{Remark-Convention}

\newcommand{\fakeenv}{} 

{ 
 \renewcommand{\fakeenv}{#2} 
 \theoremstyle{plain} 
 \newtheorem*{\fakeenv}{#1~\ref{#2}} 
 \begin{\fakeenv}
}
{
 \end{\fakeenv}
}

\newcommand{\EE}{\mathbb{E}}
\newcommand{\HH}{\mathbb{H}} 
\newcommand{\NN}{\mathbb{N}} 
\newcommand{\RR}{\mathbb{R}}
\newcommand{\PP}{\mathbb{P}}
\newcommand{\ZZ}{\mathbb{Z}}

\newcommand{\calP}{\mathcal{P}}

\newcommand{\oalpha}{{\overline{\alpha}}}
\newcommand{\ogamma}{{\overline{\gamma}}}
\newcommand{\LL}{\Lambda}
\newcommand{\oLambda}{{\overline{\Lambda}}}

\renewcommand{\ge}{\geqslant}
\renewcommand{\le}{\leqslant}

\newcommand{\e}{\epsilon}
\newcommand{\ws}{{\widetilde{S}}}
\newcommand{\wsr}{{\widetilde{S}_h \times \RR}}

\newcommand{\GP}[2]{\left( #1 \, . \, #2 \right)}

\DeclareFontFamily{U}{mathx}{}
\DeclareFontShape{U}{mathx}{m}{n}{<-> mathx10}{}
\DeclareSymbolFont{mathx}{U}{mathx}{m}{n}
\DeclareMathAccent{\widecheck}{0}{mathx}{"71}

\newcommand{\rnu}{\widecheck{\nu}}

\newcommand{\PSL}{\text{PSL}}

\newcommand{\pa}{f\xspace}

\newcommand{\diam}{\textup{diam}}

\newcommand{\param}%
	{{\mathchoice{\mkern1mu\mbox{\raise2.2pt\hbox{$\centerdot$}}\mkern1mu}%
	{\mkern1mu\mbox{\raise2.2pt\hbox{$\centerdot$}}\mkern1mu}%
	{\mkern1.5mu\centerdot\mkern1.5mu}{\mkern1.5mu\centerdot\mkern1.5mu}}}

\newtheorem*{proposition:ft}{Proposition \ref{prop:fellow travel}}
\newtheorem*{proposition:pi}{Proposition \ref{prop:projection interval}}
\newtheorem*{theorem:hitting}{Theorem \ref{theorem:hitting}}
\newtheorem*{theorem:lebesgue}{Theorem \ref{theorem:lebesgue}}

\newtheorem*{lemma:bounded distance}{Lemma \ref{lemma:bounded distance}}
\newtheorem*{lemma:intersection interval qg}{Lemma \ref{lemma:intersection interval qg}}
\newtheorem*{prop:union qg}{Proposition \ref{prop:union qg}}
\newtheorem*{lemma:straight qg}{Lemma \ref{lemma:straight qg}}
\newtheorem*{cor:corner distance}{Lemma \ref{cor:corner distance}}
\newtheorem*{prop:vertical flow distance decreasing}{Proposition \ref{prop:vertical flow distance decreasing}}

\newlist{thmenum}{enumerate}{1} 
\setlist[thmenum]{leftmargin=3\parindent, label=\textup{(\arabic*)},
                  ref=\textup{\thetheorem.\arabic*}}
\crefname{thmenumi}{proposition}{propositions}

\begin{document}

\maketitle

\begin{abstract}
    
In a closed fibered hyperbolic 3-manifold $M$, the inclusion of a fiber $S$, with $S$ and $M$ lifted to the universal covers $\widetilde{S}$ and $\widetilde{M}$, gives an exponentially distorted embedding of the hyperbolic plane into hyperbolic 3-space. Nevertheless, Cannon and Thurston showed that there is a map from the circle at infinity of the hyperbolic plane to the 2-sphere at infinity of hyperbolic 3-space. The Cannon--Thurston map is surjective, finite-to-one, and gives a space-filling curve.

\medskip
Here we use properties of geodesics to prove that many natural
measures on the circle when pushed forward by the Cannon--Thurston map
become singular with respect to many natural measures on the
2-sphere. The circle measures we consider are the Lebesgue measure and
stationary measures that arise from fully supported random walks on
the surface group.  The measures on the sphere we consider are the Lebesgue measure and stationary measures that arise from geometric random walks on the 3-manifold group.

\medskip
We obtain the singularity of measures from the following properties
of typical geodesics. We prove that a hyperbolic geodesic sampled with respect to a pushforward measure asymptotically spends a definite proportion of its time close to a fiber. On the other hand, we show that a hyperbolic geodesic sampled with respect to a natural measure on the sphere spends an asymptotically negligible proportion of its time close to a fiber. For a more restricted class of circle measures, namely the Lebesgue measure and stationary measures from geometric random walks on the surface group, we also prove an effective result for the proportion of time spent close to a fiber. 
\end{abstract}

\tableofcontents

\section{Introduction}

Let $M$ be a closed 3-manifold that fibers over the circle.  Suppose
that the fiber $S$ is a closed orientable surface.  Fixing a fiber,
$M$ is homeomorphic to a mapping torus $S \times [0,1]/ \sim$, where
$\sim$ is the identification of $S \times \{0 \}$ with
$S \times \{1\}$ by a diffeomorphism $f \colon S \to S$.  The topology
of $M$ depends only on the isotopy class of $f$, that is, only on $f$
as a mapping class.

\medskip Thurston's hyperbolization theorem states that such a
manifold $M$ admits a complete hyperbolic metric if and only if $S$
has negative Euler characteristic and the monodromy $f: S \to S$ is
pseudo-Anosov.  In this case, the universal covers of $S$ and $M$ can
be identified with $\mathbb{H}^2$ and $\mathbb{H}^3$ respectively.
The natural inclusion of a fiber as $S \times \{0\}$ induces a map
$\iota: \widetilde{S} = \mathbb{H}^2 \hookrightarrow \mathbb{H}^3 =
\widetilde{M}$ between the universal covers.  The inclusion $\iota$ is
exponentially distorted for the hyperbolic metrics on the source and
the target.  Cannon--Thurston proved that despite the distortion, the
inclusion extends to a $\pi_1(S)$-equivariant, continuous map at
infinity, known as the Cannon-Thurston map, which we shall also call
$\iota$, i.e.
$\iota \colon \partial \mathbb{H}^2 = S^1_\infty \to S^2_\infty =
\partial \mathbb{H}^3$ \cite{cannon-thurston}. We maintain this notation of $M$, $S$, $f$, and 
$\iota$
throughout this paper. In particular, we assume throughout that $\chi (S) < 0$ and $f$ is pseudo-Anosov.

\medskip We may compare measures on $S^1_\infty$ with measures on
$S^2_\infty$, by taking the pushforward of the measures on
$S^1_\infty$ under the Cannon-Thurston map. We now describe the
collections of measures that we will consider.  The first collection, which we refer to as \emph{surface measures}, are measures on the
boundary of the universal cover of the surface $S$, namely $S^1_\infty$.
We write $S_h$ for $S$ endowed with a particular choice of hyperbolic
metric. That is, $S_h$ is a hyperbolic surface.

\medskip Suppose that $G$ is a group. 
A probability measure $\mu$ on $G$ generates a random
walk on $G$.  The steps of the random walk are independent identically
$\mu$-distributed random variables $(g_n) \in (G, \mu)^{\NN}$, and
the location of the random walk at time $n$ is given by
$w_n = g_1 \ldots g_n$.  If $G$ acts on a metric space $(X, d)$ with
a basepoint $x_0$, we say that $\mu$ has \emph{finite exponential moment} if there is a constant $c > 1$ such that $\sum_{g \in G} \mu(g) c^{d(x_0, g x_0)}$
is finite.  We say a probability measure $\mu$ on a group $G$ is
\emph{geometric} if $\mu$ has finite exponential moment with respect
to a word metric on $G$ and if the support of $\mu$ generates $G$ as a
semigroup. We denote the set of
geometric probability measures on $G$ by $\calP(G)$. In this paper, the group $G$ will be either $\pi_1(S)$ or $\pi_1(M)$. In these cases, for any
basepoint $x_0$, almost every sample path $w_n x_0$ of the random walk
converges to the boundary, and the resulting boundary measure is known as the \emph{hitting measure} determined by $\mu$. For any two
basepoints $x_0$ and $x_1$, the distance between $w_n x_0$ and
$w_n x_1$ is constant, so the hitting measure does not depend on the
choice of basepoint.

\begin{definition}\label{def:surface}
Let $S_h$ be a closed hyperbolic surface of genus at least $2$.  Let
$S^1_\infty$ be the Gromov boundary of the universal cover, with
basepoint $x_0$. We will refer to the following measures on $S^1_\infty$ as
\emph{geometric surface measures}.

\begin{itemize}

\item\label{def:leb} The Lebesgue measure on $S^1_\infty$ determined by the visual
measure at $x_0$.

\item The hitting measure on $S^1_\infty$ for a random walk determined
by a geometric probability distribution $\mu \in \calP(\pi_1(S))$.

\end{itemize}

\end{definition}

The Lebesgue measures that arise as the visual measures from different
marked hyperbolic metrics are mutually singular, see for example Agard
\cite{agard}.  However, for a fixed hyperbolic metric, different
choices of basepoints give absolutely continuous measures.  Similarly,
different choices of the probability measure $\mu$ on $\pi_1(S)$ are
expected to usually give mutually singular hitting measures, whereas
for a fixed $\mu$, the hitting measures arising from different
basepoint choices are absolutely continuous with respect to each
other.

\medskip
\Cref{def:surface} covers a wide class of measures. In fact, it is a long standing conjecture of Guivarc'h--Kaimanovich--Ledrappier
(see \cite{dkn2009}*{Conjecture 1.21}) that a hitting measure $\nu$
that arises from any finitely supported random walk on $\pi_1(S)$ is
singular with respect to the Lebesgue measure on $S^1_\infty$. 

\medskip
It will also be convenient to consider random walks on surface groups
generated by more general probability measures.  We say a probability
measure $\mu$ on $\pi_1(S)$ is \emph{non-elementary} if the semigroup
generated by the support of $\mu$ contains a pair of non-trivial
elements with disjoint fixed points on the boundary $S^1_\infty$.  We
say a probability measure $\mu$ on $\pi_1(S)$ is \emph{full} if every
open set in $S^1_\infty$ has positive hitting measure.

\begin{definition}\label{def:surface general}
Let $S_h$ be a closed hyperbolic surface of genus at least $2$.  Let
$S^1_\infty$ be the Gromov boundary of the universal cover. We refer to the larger collection of measures on $S^1_\infty$ which contains in addition to geometric surface measures,

\begin{itemize}

\item the hitting measure on $S^1_\infty$ for a random walk determined
by a non-elementary, full probability measure $\mu$ on $\pi_1(S)$.

\end{itemize}
as the \emph{full surface measures}.
\end{definition}

The final collection of measures which we shall refer to as
\emph{$3$-manifold measures}, are measures on the boundary of
the universal cover of the $3$-manifold $M$, namely $S^2_\infty$.

\begin{definition}\label{def:3-manifold}
Let $M$ be a closed hyperbolic $3$-manifold.  Let $S^2_\infty$ be the
Gromov boundary of the universal cover, with basepoint $x_0$.  We
shall call the following measures \emph{$3$-manifold measures}.

\begin{itemize}
    
\item The Lebesgue measure on $S^2_\infty$ determined by the visual
measure at $x_0$.
\item The hitting measure on $S^2_\infty$ determined by a geometric
random walk generated by $\mu \in \calP(\pi_1(M))$.

\end{itemize}

\end{definition}

As with the Lebesgue measures on $S^1_\infty$, Lebesgue measures on $S^2_\infty$ are expected to be mutually
singular with respect to the hitting measures arising from finitely
supported random walks on $\pi_1(M)$.  The Lebesgue measures arising
from different basepoints are absolutely continuous with
respect to each other.  
The hitting measures
arising from different probability measures $\mu$ are typically expected to be mutually singular.

\medskip
Let $M$ be a closed hyperbolic $3$-manifold which fibers over the
circle.  As mentioned before, this determines a collection of $3$-manifold measures on
$S^2_\infty$ given by \Cref{def:3-manifold}.  Let $S$ be a
fiber and let $S_h$ be a hyperbolic structure on $S$.  These choices of $M$ and $S_h$ determine a collection of surface measures on $S^1_\infty$, given by \Cref{def:surface general}.  The pushforwards of these surface
measures by the Cannon-Thurston map $\iota$ give measures on
$S^2_\infty$, so that we may now compare the surface measures to the
$3$-manifold measures:

\begin{theorem}\label{theorem:singularity}
Suppose that $M$ is a closed hyperbolic $3$-manifold that fibers over the circle. Then pushforwards to $S^2_\infty$ under the Cannon-Thurston map of any full surface measure from \Cref{def:surface general} (using any hyperbolic structure $S_h$ on the fiber $S$) is mutually singular with respect to any $3$-manifold measure from \Cref{def:3-manifold}.
\end{theorem}

Suppose now that $M$ is a compact hyperbolic 3-manifold and $S$ is an
incompressible surface in $M$.  As an immediate corollary of
\Cref{theorem:singularity}, we obtain:

\begin{corollary}
Suppose that $\mu$ is a geometric probability measure on the fundamental group
of any incompressible surface $S$ of a closed hyperbolic 3-manifold
$M$.  Then the hitting measure on $S^2_\infty$ arising from $\mu$ is
mutually singular with respect to both the Lebesgue measure on
$S^2_\infty$ and any hitting measure arising from a geometric random
walk on $\pi_1(M)$.
\end{corollary}

\begin{proof}
An incompressible surface $S$ in $M$ is either quasi-Fuchsian or a
virtual fiber, by the Tameness Theorem \cite{calegari-gabai} and the Covering
Theorem \cite{canary}.  If it is quasi-Fuchsian then the hitting
measure is supported on the limit set of the surface subgroup.  This limit
set has measure zero for both the Lebesgue measure on $S^2_\infty$ and for
any hitting measure arising from a geometric random walk on
$\pi_1(M)$.

\medskip On the other hand, if $S$ is a virtual fiber then $M$ admits
a fibered finite cover in which the fiber $F$ is a finite cover of
$S$.  As $\pi_1(F)$ is a finite index subgroup of $\pi_1(S)$, a
geometric random walk on $\pi_1(S)$ is recurrent on $\pi_1(F)$, and
the restriction of the random walk on $\pi_1(S)$ to $\pi_1(F)$ is the
random walk on $\pi_1(F)$ generated by the first hitting measure
$\mu'$ of the random walk of $\pi_1(S)$ on $\pi_1(F)$.  The
probability measure $\mu'$ is not finitely generated, but is
non-elementary, and has the same hitting measure as $\mu$, so is full.
Therefore, \Cref{theorem:singularity} implies that the hitting measure
is mutually singular with respect to either of the $3$-manifold
measures on $S^2_\infty$.
\end{proof}

Singularity of the pushforward of the Lebesgue measure on $S^1_\infty$
with respect to the Lebesgue measure on $S^2_\infty$ was previously
shown by Tukia (see \cite{tukia}*{page 430}) using conformal
techniques, such as cross-ratios.  This result also follows from work
of Kim and Oh \cite{KimOh}, showing singularity of conformal measures
for representations of divergence type and Anosov representations, see
their discussion of Cannon-Thurston maps in \cite{KimOh}*{Section
  1}. Kim and Zimmer give further such rigidity results
\cite{kim-zimmer}, which also imply this result.  The singularity of
the push forward of the surface random walk hitting measure and the Lebesgue measure on $S^2_\infty$ follows from Blach\`ere, Ha\"issinsky
and Mathieu, \cite{bhm}*{Proposition 5.5}, which implies that the hitting
measure has strictly smaller Hausdorff dimension, as the fundamental
group of the surface is not quasi-isometrically embedded in $\HH^3$.
In light of the Guivarc'h--Kaimanovich--Ledrappier
conjecture, we do not expect hitting measures to be conformal.

\medskip
As we describe in the next section, \Cref{subsection:geodesics}, our approach uses the behavior of typical geodesics chosen according to the measures on the boundary which enables us to provide a unified (geometric) context in which the different types of surface measures can be compared to the $3$-manifold measures.

\subsection{Statistics for typical geodesics}\label{subsection:geodesics}

An oriented geodesic in $\HH^n$ is determined by its (ordered)
endpoints in
$( S^{n-1}_\infty \times S^{n-1}_\infty ) \setminus \Delta$, where
$\Delta$ is the diagonal.  A probability measure $\nu$ on
$S^{n-1}_\infty$ gives rise to a probability measure
$\nu \times \nu$ on the space of oriented geodesics in $\HH^n$, as long as the diagonal has measure zero.  

\medskip
For random walks, if $\mu$ is not symmetric, we will
use the measure $\nu \times \widecheck \nu$ instead of the product
measure, where $\widecheck \nu$ is the hitting measure corresponding
to the limits $\lim_{n \to -\infty} w_n x_0$ of the sample paths. To prove Theorem \ref{theorem:singularity}, we show the desired measures are mutually singular by showing that they give rise to typical geodesics possessing different behavior.

\medskip
The fibration of $M$ by closed surfaces lifts to a fibration of the
universal cover by the universal covers of the fibers.  We will call
the image of $\widetilde S$ under the Cannon-Thurston map the \emph{base
fiber} $S_0 := \iota(\widetilde S)$.

\begin{itemize}

\item With respect to the pushforwards of the surface measures to $S^2_\infty$,
almost all geodesics in $\HH^3$ spend a positive proportion of time
close to the base fiber $S_0$.

\item With respect to the $3$-manifold measures on $S^2_\infty$, for almost all
geodesics in $\HH^3$, the proportion of time the geodesic spends close to the
base fiber $S_0$ tends to zero.

\end{itemize}

We now give precise versions of these statements.  Let $\gamma(t)$ be
a geodesic with unit speed parametrization.  We will write
$\gamma([0, T])$ for the segment of $\gamma$ between $\gamma(0)$ and
$\gamma(T)$.  For any subset $A$ of $\gamma$, we will write $|A|$ for
the standard Lebesgue measure of $A$. 
As $M$ fibers over the circle,
the universal cover also fibers as $\widetilde{S} \times \RR$.  As we explain in detail in \Cref{section:pseudometric}, Cannon--Thurston define a
pseudometric on $\wsr$ called the \emph{Cannon--Thurston} metric. 
The Cannon--Thurston metric depends on a choice of hyperbolic metric $S_h$ on $S$ and on the pseudo-Anosov $\pa$ in terms of its pair of invariant measured laminations $\Lambda_+$ and $\Lambda_-$.
We will write $(S_h, \Lambda)$ to denote a hyperbolic metric on $S$ and the corresponding pair $\Lambda = (\Lambda_+, \Lambda_-)$ of invariant measured
laminations on $S_h$.  For full surface measures we show:

\begin{theorem}\label{theorem:surface measures}
Suppose that $\pa \colon S \to S$ is a pseudo-Anosov map and $(S_h, \Lambda)$ is a hyperbolic structure on $S$ together with a pair of invariant measured laminations.  Let $\widetilde S_h \times \RR$ be the universal cover of the corresponding mapping torus with the Cannon--Thurston metric, and let $\iota$ be the Cannon--Thurston map.  Let $\nu$ be a full surface measure from \Cref{def:surface general}.

\medskip
Then there are constants $R \ge 0$ and $\epsilon > 0$ (that depend on $f$ and $\nu$) such that for $\iota_* \nu$-almost all geodesics
$\gamma$ in $\widetilde S_h \times \RR$, for any unit speed parametrization $\gamma(t)$,
\[ \lim_{T \to \infty} \frac{1}{T} | \gamma([0, T]) \cap N_R(S_0) | \ge \epsilon. \]
\end{theorem}

\medskip In fact, the constant $\epsilon$ tends to one as $R$ tends to
infinity. By restricting to geometric surface measures from \Cref{def:surface},  we prove the following effective bound on the rate of convergence.

\begin{theorem}\label{theorem:effective}
Suppose that $\pa \colon S \to S$ is a pseudo-Anosov map with stretch
factor $k > 1$, and $(S_h, \Lambda)$ is a hyperbolic structure on $S$
together with a pair of invariant measured laminations.  Let
$\widetilde S_h \times \RR$ be the universal cover of the
corresponding mapping torus with the Cannon--Thurston metric, and let $\iota$ be the Cannon--Thurston
map.  Let $\nu$ be a geometric surface measure from \Cref{def:surface}.

\medskip
Then there are constants $K > 0$ and $\alpha > 0$ such that for
$\iota_* \nu$-almost all geodesics $\gamma$ in
$\widetilde S_h \times \RR$ and for any unit speed parametrization
$\gamma(t)$, for any $R \ge 0$,
\[ \lim_{T \to \infty} \frac{1}{T} | \gamma([0, T]) \cap N_R(S_0) | \ge 1 - K e^{- \alpha k^{R}}. \]
\end{theorem}

\medskip
For $3$-manifold measures we show:

\begin{theorem}\label{theorem:3-manifold measures}
Suppose that $\pa \colon S \to S$ is a pseudo-Anosov map and
$(S_h, \Lambda)$ is a hyperbolic structure on $S$ together with a pair of
invariant measured laminations.  Let $\widetilde S_h \times \RR$ be
the universal cover of the corresponding mapping torus with the Cannon--Thurston metric and let
$\iota$ be the Cannon--Thurston map.  Let $\nu$ be a $3$-manifold measure from \Cref{def:3-manifold}.

\medskip
Then for $\nu$-almost all geodesics $\gamma$ in
$\widetilde S_h \times \RR$, for any unit speed parametrization
$\gamma(t)$ and any $R > 0$,
\[ \lim_{T \to \infty} \frac{1}{T} | \gamma([0, T]) \cap N_R( S_0 ) | = 0.  \]
\end{theorem}

The mutual singularity of the surface measures and the $3$-manifold
measures, given in \Cref{theorem:singularity}, is an immediate
consequence of \Cref{theorem:surface measures} and
\Cref{theorem:3-manifold measures}.  Sets exhibiting the mutual
singularity may be explicitly described as sets of geodesics which
spend a positive proportion of time close to the base fiber $S_0$, and
sets of geodesics whose proportion of time close to $S_0$ tends to
zero.

\medskip Compared to \Cref{theorem:surface measures}, the effective
version, namely \Cref{theorem:effective}, relies on an explicit construction
of certain quasi-geodesics which we now briefly describe, see \cite{gmpu}*{Section 3}
for more details.  We discuss how this is related to previous constructions of quasigoedesics in \Cref{section:solv}.

\medskip
For a compact hyperbolic $3$-manifold $M$ which fibers over the
circle, the universal cover $\widetilde M$ has two natural structures.
First, the hyperbolic metric on $M$ lifts to a hyperbolic metric on
$\widetilde M$, which is isometric to $\HH^3$.  Any
$\pi_1(M)$-invariant metric on $\widetilde M$ will be quasi-isometric to this hyperbolic metric. Second, as $M$ fibers over the circle, the fibration lifts to a product structure $\widetilde S \times \RR$ on $\widetilde M$.  This is only a topological product structure, as no product metric on $\widetilde S \times \RR$ can be quasi-isometric to the hyperbolic metric.

\medskip
Given a hyperbolic structure $S_h$, 
the pseudo-Anosov mapping class $\pa$ determines a pair of invariant measured laminations.  A choice of a hyperbolic structure $S_h$ lifts to a hyperbolic structure $\widetilde S_h$ on the universal cover
$\widetilde S$, and the invariant measured laminations can be realized
as measured geodesic laminations in this metric.  Cannon--Thurston
\cite{cannon-thurston} used the invariant measured geodesic
laminations to construct a $\pi_1(M)$-invariant pseudo-metric on
$\widetilde S_h \times \RR$, which is therefore quasi-isometric to the
hyperbolic metric on $\widetilde M$, see \Cref{section:pseudometric}
for further details. We call this pseudo-metric the
\emph{Cannon--Thurston metric} on $\widetilde S_h \times \RR$.

\medskip 
Lifting an invariant lamination to $\widetilde{S}_h$, any leaf $\ell$ of the lift is geodesic with respect to the hyperbolic metric on $\widetilde{S}_h$.  However, its image $\iota(\ell)$ in $\widetilde S_h \times \RR$ is embedded metrically as a horocycle.  In
particular, $\iota(\ell)$ is not quasigeodesic, and a segment of
$\iota(\gamma)$ of length $L$ has endpoints distance roughly $\log L$
apart in the Cannon--Thurston metric on $\widetilde S_h \times \RR$.
For any surface measure, almost all geodesics $\gamma$ in
$\widetilde S_h$ have arbitrarily long subsegments that fellow travel
leaves of the invariant laminations.  The image $\iota(\gamma)$ in
$\widetilde S_h \times \RR$ is thus not quasigeodesic, even up to
reparametrization.  Our basic idea is to straighten ``horocyclic''
$\iota(\gamma)$ segments, that is, replace $\gamma$-subsegments that
fellow-travel leaves of the invariant laminations by shortcuts.  To do
so, we define a \emph{height function} $h_\gamma(t)$ along the
geodesic, which is roughly $\log \log$ of the distance (in the unit
tangent bundle) from the geodesic to the invariant laminations.  We
then show that the paths $( \gamma(t), h_\gamma(t) )$ in
$\widetilde S_h \times \RR$ are uniformly quasigeodesic, i.e. their
quasigeodesic constants do not depend on the choice of the geodesic
$\gamma$ in $\widetilde S_h$.

\medskip
In \Cref{section:background}, we review some previous results
we use, and define some notation.  In \Cref{section:singularity}, we prove Theorems \ref{theorem:surface
  measures} and \ref{theorem:3-manifold measures}, which immediately
imply the singularity of measures result, \Cref{theorem:singularity}. In \Cref{section:effective}, we prove the effective
bounds in \Cref{theorem:effective}, by assuming the main result of \cite{gmpu}*{Section 3}.

\subsection{Quasigeodesics in Cannon--Thurston metrics}\label{section:solv}

In this section, we discuss earlier constructions of quasigeodesics in
Cannon--Thurston metrics, and some technical issues that led us to the
variant that we use in this paper.

\medskip

Given a subset $A$ of the fiber, the union of all flow lines
intersecting $A$ is called the \emph{ladder} over $A$.  Cannon--Thurston \cite{cannon-thurston} showed that ladders over leaves of
the stable and unstable laminations are quasiconvex. Mitra
\cite{mitra} extended this to show that ladders over geodesics in the
fiber are quasiconvex, even in the more general context of hyperbolic
group extensions.  This implies that given a geodesic $\gamma$ in the
universal cover of the fiber surface, the geodesic in the universal cover of the $3$-manifold connecting the endpoints of 
$\iota(\gamma)$ will be coarsely contained in the ladder over
$\iota(\gamma)$, and has a parameterization in $\wsr$ as
$(\iota(\gamma(t)), h_\gamma(t))$, for some function
$h_\gamma \colon \RR \to \RR$.  Note that \emph{a priori} this
function may depend on $\gamma$, and not necessarily arise from a
function on the unit tangent bundle.

\medskip
McMullen \cite{mcmullen} constructed quasigeodesics in the classical
Cannon--Thurston case, using saddle connections in the flat metric.
Hamenstaedt \cite{hamenstaedt} considered the more general case of
word hyperbolic extensions of surface groups, and Mj and Sardar
\cite{mj-sardar} and Kapovich and Sardar \cite{kapovich-sardar}
consider the general case of hyperbolic groups arising as hyperbolic-by-hyperbolic groups.

\medskip
In all the above cases, the authors construct
quasigeodesics by connecting segments of vertical flow lines by horizontal paths\footnote{We warn
  the reader that other authors, for example Kapovich and Sardar
\cite{kapovich-sardar}, use the opposite convention for which
directions are horizontal or vertical.}.  However, they do not prove an appropriate length upper bound for the vertical segments and the projection from the fiber geodesic to the quasi-geodesic by vertical flow to be coarsely distance non-increasing. 

\medskip
In proving the effective statements in \Cref{section:effective}, we need these properties, particularly the coarse distance non-increasing property. See \Cref{prop:vertical flow distance decreasing}.  As we note in the example below, this property does not hold for all
quasigeodesics, and depends on the details of the construction.  
For the example, consider a geodesic $\gamma$ in the hyperbolic plane intersecting a horocycle $h$. 
See the illustration in in \Cref{fig:qg} in the upper half space model of $\HH^2$.

\begin{figure}[h]
\begin{center}
\begin{tikzpicture}[scale=0.85]

\tikzstyle{point}=[circle, draw, fill=black, inner sep=1pt]

\draw (0, 0) -- (8, 0);

\draw (0, 1) -- (8, 1) node [label=above:$h$] {} ;

\draw (1, 0) arc (180:0:3) node [label=above right:$\gamma$] {};

\draw (1, 0) -- (1, 3) node [midway, label=left:$\alpha$] {} -- (7, 3) -- (7, 0);

\draw (1, 0) -- (4, 3) node [midway, label=left:$\beta$] {} -- (7, 0);

\end{tikzpicture}
\end{center}
\caption{Two quasigeodesics in the upper half space model of $\HH^2$.} \label{fig:qg}
\end{figure}
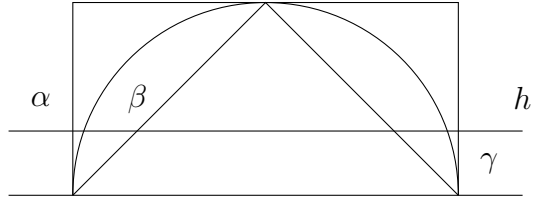

\medskip
The two paths $\alpha$ and $\beta$ in \Cref{fig:qg} are quasigeodesics. For all horocycles $h$ intersecting $\gamma$, the
vertical projection from $h$ onto $\beta$ is uniformly
coarsely distance non-increasing, that is, there are constants $D, K, C > 0$ such that for any horocycle $h$ and any pair of points distance at least $D$ apart in the path metric on $h$, their projections to $\beta$ are distance at most $K$ times their distance along $h$ plus $C$. However, the
vertical projection from $h$ to $\alpha$ does not have this property since a uniform $D$ for $h$ as its height (imaginary part) goes to zero. 

\medskip
In order to apply the ergodicity of the geodesic flow, we need to
relate properties of quasigeodesics to a function on the unit tangent
bundle of the surface.  We expect it to be possible to modify the
previous constructions of quasigeodesics to have the desired
properties, though we can also show directly that the function we give
on the unit tangent bundle gives quasigeodesic paths. In this paper we give detailed proofs of the singularity results and their effective versions, assuming the existence of quasigeodesics with the desired properties, stated precisely in \Cref{theorem:quasigeodesic-h2} and \Cref{prop:vertical flow distance decreasing}.
As verifying these properties
is somewhat technical, we give the details in a separate paper \cite{gmpu}.

\subsection{Remarks on Related Work}

It is interesting to compare Cannon--Thurston maps with the more classical space filling curves, such as the Peano curve.
In contrast to our results here, the Peano curve is absolutely continuous.

\medskip
The Peano curve is also H\"older with exponent $1/2$.
In contrast, Miyachi proved that Cannon--Thurston maps are not H\"older, see \cite{Miy}*{Theorem 1.1}.
Since non-H\"older maps can be absolutely continuous, the absence of regularity does not imply singularity for the pushforward of the Lebesgue measure on the circle.

\medskip
There are many Cannon--Thurston type phenomena generalizing the setup from fibered hyperbolic 3-manifolds. 
A natural generalization is given by Gromov hyperbolic groups that are extensions of surface groups, a fundamental group of a fibered hyperbolic 3-manifold being an extension by $\ZZ$.
The existence of a Cannon--Thurston map is proved in \cite{mitra}, and for more about the structure of the map see \cite{Mj-Rafi}. 
In this case, we expect pushforwards of stationary measures on the circle to be singular with respect to geometric/fully supported stationary measures on the Gromov boundary of the extension.
On the one hand, we expect geodesics in the extension sampled by the pushforward measure to spend a definite proportion of their time in a neighborhood of the surface subgroup.
On the other hand, we expect geodesics sampled by stationary measures resulting from geometric/fully supported random walks to spend asymptotically negligible time in a neighborhood of the surface subgroup.
While the reasons for our expectations are similar, our methods here are specific for hyperbolic 3-manifolds and these questions are left for now to future work.
For a survey of other Cannon--Thurston type examples, see \cite{Gad-Hen}.

\medskip
For Kleinian surface groups, one obtains Cannon-Thurston maps more generally from doubly degenerate surface group representations in $\PSL(2,\mathbb{C})$ \cite{MjCT}. Without the $\ZZ$-periodicity, it makes good sense only to compare the pushforward of the Lebesgue measure on $S^1_\infty$ with the Lebesgue measure on $S^2_\infty$, and here the singularity of the pushforward is again covered by Tukia's work. We indicate which of our techniques underlying  \Cref{theorem:surface measures} and \Cref{theorem:3-manifold measures} hold when the $S \times \mathbb{R}$ has bounded geometry; we leave the full discussion of our perspective in the bounded geometry case to future work. We also leave more general Cannon--Thurston situations to future work.

\subsection{Acknowledgments}

This work would not have been possible without the wonderful working
conditions of the American Institute of Mathematics. We would like to
thank them for their hospitality during the April 2022 workshop
“Random walks beyond hyperbolic groups” where this work started.  We
thank Dongryul Kim for pointing out Tukia's work and his interest in
this work. We thank Fran\c{c}ois Ledrappier, Chi Cheuk Tsang, and
Chris Leininger for helpful discussions. We thank Thomas Haettel for
useful discussions at the early stages of this work.  We thank Mahan
Mj and Hee Oh for their comments on the first draft, and the referees for useful advice.

\medskip
The first author acknowledges the support of the Institut Henri
Poincaré (UAR 839 CNRS-Sorbonne Université) and LabEx CARMIN
(ANR-10-LABX-59-01). The first author thanks Richard Canary and Yair
Minsky for discussions related to this work during the programme
\texttt{"}Higher rank geometric structures\texttt{"} at the IHP.  The
second author thanks PSC-CUNY and the Simons Foundation for their
support.  The third author is grateful to the Institute for Advanced
Study for their hospitality and the Bob Moses fund for funding her
2024--2025 membership, she is funded by an NSERC Discovery Grant. The
last author gratefully acknowledges support from NSF DMS--2439076.


\section{Background}\label{section:background}

In this section, we review some background and fix notation.

\medskip Let $S$ be a closed orientable surface of genus at least two and suppose that $\pa \colon S \to S$ is an orientation preserving
diffeomorphism.  Then $\pa$ determines a mapping
torus $M_\pa = S \times [0,1] / \sim$, which is the quotient of
$S \times [0,1]$ by the relation $(p, 0) \sim (f(p), 1)$.  The map
$\pa$ is often referred to as the \emph{monodromy map} for the mapping
torus.  We say a closed orientable $3$-manifold $M$ \emph{fibers over
  the circle} if $M$ is homeomorphic to a mapping torus $M_\pa$ of
some orientation preserving surface diffeomorphism
$\pa \colon S \to S$.  Thurston \cite{thurston} showed that the
$3$-manifold $M_\pa$ is hyperbolic if and only if $\pa$ is
pseudo-Anosov.  We will always assume that $\pa$ is pseudo-Anosov and
so $M_\pa$ is hyperbolic.  For simplicity of notation, we will just
write $M$ for $M_\pa$.

\medskip
The universal cover $\widetilde M$ may be thought of as $\RR^2 \times \RR$, where 
\begin{itemize}
\item the universal
cover $\widetilde{S}$ of the fiber $S$ is identified with $\RR^2$ and
\item the monodromy map $\pa$ acts by unit
translation in the $\RR$-factor direction.
\end{itemize}
Since $\pa$ is pseudo-Anosov, the action of $f$ on the Teichmüller
space (of marked complete hyperbolic metrics on $S$) has an invariant
axis on which $f$ acts by translation.  By identifying the $\RR$
factor in $\widetilde{S} \times \RR$ with the Teichmüller axis, we
obtain a marked hyperbolic metric on $S$ corresponding to
$\widetilde{S} \times \{0 \}$.  We shall denote this metric by $S_h$.
The monodromy map $\pa$ acts by a change of marking of $S_h$.  The
universal cover $\widetilde S_h$ is isometric to the hyperbolic plane
$\HH^2$.  We will write $d_{\HH^2}$ for the hyperbolic metric on
$\widetilde S_h$.  The inclusion map $\iota$ sends $\widetilde S_h$ to
$\widetilde S_h \times \{ 0 \}$ in the universal cover
$\widetilde S_h \times \RR$ of $M$.  We denote $\iota(\widetilde S_h)$
by $S_0$.

\medskip As in Section \ref{subsection:geodesics}, the action of $\pa$ on the fiber $S_h$ has a pair of
invariant measured geodesics laminations.  The transverse measures are
uniquely ergodic and the laminations are transverse to each other.
The monodromy map $\pa$ acts by stretching the leaves of one
lamination (called the \emph{unstable lamination}) and by contracting
the leaves of the transverse lamination (called the \emph{stable
  lamination}). The laminations can be lifted to
$\pi_1(S)$-equivariant laminations of $\widetilde S_h$. These laminations reappear with further details in \Cref{sec:measured laminations}.

\medskip Using the lifted laminations, Cannon and Thurston
\cite{cannon-thurston} constructed a $\pi_1(M)$-invariant pseudometric
on the universal cover $\widetilde S_h \times \RR$ of $M$.  See \Cref{section:pseudometric} for further details about the
\emph{Cannon--Thurston metric}.

\medskip
The fiber lying on the invariant Teichmüller axis also carries a singular flat metric given by the associated quadratic differential on the underlying marked conformal surface.
The monodromy map $\pa$ acts affinely on the singular flat metric stretching the horizontal foliation and contracting (by the reciprocal of the stretch factor) the vertical foliation.
We call the singular flat metric the \emph{invariant flat metric} for $\pa$.

\medskip
In an analogous way to Cannon--Thurston, we can use the
invariant flat metric to construct a $\pi_1(M)$-equivariant metric on
the universal cover $\widetilde S_q \times \RR$ of $M$.  This metric
is known as the \emph{singular solv metric}.

\medskip
By the Švarc--Milnor lemma, both metrics, and also the hyperbolic metric, on $\HH^3$ are quasi-isometric to $\pi_1(M)$ and so quasi-isometric to each other.

\medskip
We do not use the singular solv metric directly, but we do discuss the relation between the Cannon-Thurston metric and the singular solv metric in \Cref{section:flat}.  
We shall always write $\widetilde S_h \times \RR$ for the universal cover of $M$ to emphasize that we are using the Cannon-Thurston metric and not the others.

\subsection{Measured laminations}
\label{sec:measured laminations}

The properties of measured laminations that we present below are standard, see e.g.  \cite{Casson-Bleiler}. 
We present them in detail to keep our discussion self-contained.

\medskip
A (possibly bi-infinite) geodesic on $S_h$ is \emph{simple} if it has no
self-intersections.  
A \emph{geodesic lamination} on $S_h$ is a closed
union of simple pairwise disjoint geodesics.  
A \emph{transverse measure} on a geodesic
lamination $\Lambda$ is a positive measure $dm$ defined on local transverse arcs to the leaves of $\Lambda$ that 
\begin{itemize}
    \item is invariant under any isotopy preserving the transverse
    intersections with the leaves of $\Lambda$, and
\item is positive and finite on any
nontrivial compact transversals.
\end{itemize}
Such a measure lifts to a $\pi_1(S)$-invariant transverse measure on
the pre-image of $\Lambda$ in $\HH^2$.  By abuse of notation, we will
denote the pre-image also by $\Lambda$ and the lifted measure also by
$dm$.  A geodesic lamination equipped with a transverse measure is
called a \emph{measured lamination}.  We will only consider measured
laminations, and so we will often just write lamination to mean
measured lamination.

\medskip

We say a measured lamination is \emph{filling} if there are no
essential simple closed curves disjoint from the lamination.  The
complement of a filling lamination is a union of ideal polygons with
finitely many sides. We say a leaf of the lamination is a
\emph{boundary leaf} if it is the boundary of an ideal polygon.  There
are finitely many ideal complementary regions in the compact
surface $S_h$, and so there are only finitely many boundary leaves in
$S_h$. This implies that there are countably many boundary leaves in
the universal cover $\widetilde S_h$.

\medskip

We say a measured lamination is
\emph{minimal} if every leaf is dense in the lamination.  A minimal
filling lamination has the following properties:
\begin{itemize}
    \item there are uncountably many leaves,
    \item no leaf is isolated, and
    \item the transverse measure is non-atomic.
\end{itemize}

We say a pair of measured laminations $\Lambda_+$ and $\Lambda_-$
\emph{bind} $S_h$ if each geodesic ray on $S_h$ crosses a leaf of
$\Lambda_+ \cup \Lambda_-$.

\begin{definition}
We shall write $(S_h, \Lambda)$ for a triple consisting of a
hyperbolic metric on a compact surface $S$, together with a pair of
minimal filling measured laminations $\Lambda_+$ and $\Lambda_-$ which
bind the surface.  We refer to such a triple $(S_h, \Lambda)$ as a
\emph{hyperbolic surface and a full pair of laminations}.
\end{definition}

A pseudo-Anosov map $\pa$ determines a pair of invariant measured
laminations called stable and unstable
laminations, which we shall denote
$(\Lambda_+, dx)$ and $(\Lambda_-, dy)$ respectively, where $dx$ and $dy$ are the transverse measures.  In particular, 
\begin{itemize}
    \item $f (\Lambda_+) = \Lambda_+$ and $f_\ast dx = k dx$, and
    \item $f(\Lambda_-) = \Lambda_-$ and $f_\ast dy = k^{-1} dy$,
\end{itemize}
where $k = k_\pa > 1$ is known as the \emph{stretch factor} of the pseudo-Anosov map.
We shall often write $\Lambda_+$ or $\Lambda_-$ to refer to the measured laminations if we
do not need to refer to the respective measures.  
In the coordinate system described in \Cref{section:pseudometric}, leaves of $\Lambda_-$ correspond to lines parallel to the $x$-axis, and leaves of $\Lambda_+$ correspond to lines parallel to the $y$-axis.

\medskip

Cannon and Thurston \cite{cannon-thurston}*{Theorem 10.1} showed that
the invariant measured laminations $\Lambda_+$ and $\Lambda_-$ are
each minimal and filling, and together they \emph{bind} the surface
$S_h$.  In fact, pseudo-Anosov invariant laminations are uniquely
ergodic, that is, the transverse measures are unique up to scale.  In
particular, the invariant measured laminations $\LL_+$ and $\LL_-$
form a \emph{full pair} of laminations for $S_h$.

\medskip We now record some useful properties of full pairs of laminations.

\begin{proposition}\label{prop:non common leaves}
Let $(S, \LL)$ be a hyperbolic surface and a full pair of laminations.
Then $\Lambda_+$ and $\Lambda_-$ have no leaf in common.
\end{proposition}

\begin{proof}
Suppose that $\Lambda_+$ and $\Lambda_-$ share a leaf $\ell$.
Since each lamination is minimal, the common leaf $\ell$ is dense in
both laminations.  This implies that $\Lambda_+ = \Lambda_-$.
But then any ideal complementary region contains a geodesic ray
disjoint from both laminations, contradicting the fact that
$\Lambda_+$ and $\Lambda_-$ bind the surface.
\end{proof}

In \Cref{prop:multi}, we combine the compactness of $S$ with the absence of common leaves to deduce the following properties: first, there is a lower bound on the angle at any point of intersection of a leaf of $\Lambda_+$ with a leaf of $\Lambda_-$, second, if a leaf of $\Lambda_+$ in $\widetilde{S}_h$ is disjoint from a lift of leaf of $\Lambda_-$ in $\widetilde{S}_h$, then there is a lower bound on
the distance between them.  The second property implies
that the ideal complementary regions of one lamination do not share
(ideal) vertices with the ideal complementary regions of the
other lamination.  Finally, there is an upper bound on the length of
a segment of a leaf which does not intersect the other lamination.

\medskip Suppose $\ell$ and $\ell'$ are leaves of 
$\Lambda_+$ and $\Lambda_-$ (in $\widetilde{S}_h$) that intersect, creating two
pairs of complementary angles.  We define the \emph{angle of intersection} to be the smallest of the two angles at the point of intersection.

\begin{proposition}\label{prop:multi}
Suppose that $(S_h, \LL)$ is a hyperbolic surface together with a full
pair of measured laminations.  Then there exist constants
$\alpha_\LL, \epsilon_\LL, L_\LL > 0$ such that each of the following
properties holds for any pair of leaves $(\ell_+,\ell_-)$ with
$\ell_+$ in $\Lambda_+$ and $\ell_-$ in $\Lambda_-$:

\begin{thmenum}[label={(\ref{prop:multi}.\arabic*)}]

\item \label{prop:angle bound} If $\ell_+$ and $\ell_-$ intersect, then their angle of
intersection is at least $\alpha_\LL$.

\item \label{prop:disjoint leaves} 
If $\ell_+$ and $\ell_-$ are disjoint,
the distance between any two lifts of $\ell_+$ and $\ell_-$ in
$\widetilde S_h$ is at least $\epsilon_\LL$.

\item \label{prop:cobounded intersections}\label{c:L_pa} Any segment of a leaf of one of the
laminations of length at least $L_\LL$ intersects a leaf of the other
lamination.

\item \label{prop:complementary regions don't share} No ideal
complementary region in $\widetilde S_h \setminus \Lambda_+$ has an
ideal vertex in common with an ideal complementary region of
$\widetilde S_h \setminus \Lambda_-$.

\end{thmenum}

\end{proposition}

\Cref{prop:complementary regions don't share} follows directly from 
\Cref{prop:disjoint leaves}; we prove the remaining
statements.

\begin{proof}[Proof of \Cref{prop:angle bound}]
Suppose that there is a sequence $(\ell^-_n, \ell^+_n)$ of pairs of intersecting leaves whose angles of intersection tend to zero.  
By compactness of the unit tangent bundle $T^1(S_h)$, we may pass to a subsequence of pairs to assume that 
\begin{itemize}
    \item the points of intersection converge in $S_h$, and
    \item the angles of intersections at these points go to zero.
\end{itemize}
Since laminations are closed subsets, we deduce that the laminations contain a common leaf, a contradiction to \Cref{prop:non common leaves}.
\end{proof}

\begin{proof}[Proof of \Cref{prop:disjoint leaves}]
Suppose there is a sequence of pairs $(\ell^+_n, \ell^-_n)$ of disjoint leaves in $\Lambda_+$ and $\Lambda_-$ such that the distance between them tends to
zero.  
By using cocompactness of the $\pi_1(S)$ action on $\widetilde{S}_h$, we may assume that the closest points between the leaves lie in a compact region in $\widetilde{S}_h$. 
Hence, we may pass to a convergent subsequence of pairs and further assume that the tangent vectors (to the leaves) at these pairs also converge.  
Since laminations are closed subsets, the sequence of leaves limit to a common leaf of $\Lambda_+$ and $\Lambda_-$, a contradiction by \Cref{prop:non common leaves}.
\end{proof}

\begin{proof}[Proof of \Cref{prop:cobounded intersections}]
Suppose there is a sequence of leaves $\ell_n$ in $\Lambda_+$,
containing subsegments $\sigma_n \subset \ell_n$ of length
$|\sigma_n| = L_n \to \infty$ of $\Lambda_+$, such that the segments
$\sigma_n$ do not intersect $\Lambda_-$.  By cocompactness, and the
fact that laminations are closed, the subsegments $\sigma_n$ limit to
a leaf $\ell$ of $\Lambda_+$ which does not intersect $\Lambda_-$.
This implies that there is a geodesic ray asymptotic to $\ell$, which
is disjoint from both laminations, contradicting the fact that the
laminations bind the surface $S_h$. The exact same argument works
with $\Lambda_+$ and $\Lambda_-$ interchanged.
\end{proof}

\subsection{Flow sets and ladders}
\label{sec:flow sets}

The mapping torus construction determines a flow on the universal
cover $\widetilde S_h \times \RR$, i.e. a continuous $1$-parameter
family of homeomorphisms given by $F_t(p, s) = (p, s+t)$, which we
will call the \emph{suspension flow}.  We shall consider the $\RR$
component of $\widetilde S_h \times \RR$ as ``vertical''.

\medskip
Suppose that $A$ is a subset of $\widetilde S_h$.
We define the \emph{suspension flow set $F(A)$} to be
\[
F(A) = \bigcup_{z \in \RR} F_z(A \times \{ 0\} ).
\]
If $A = \{p \}$ is a point, then $F(p)$ is just the suspension flow
line through $p$.  If $A$ is a hyperbolic geodesic $\gamma$ in
$\widetilde S_h$, then the suspension flow set $F(\gamma)$ is called
the \emph{ladder} of $\gamma$.  We will refer to $\gamma$ as the
\emph{base} of the ladder.

\subsection{The Cannon-Thurston metric}\label{section:pseudometric}

Let $(S_h, \LL)$ be a hyperbolic surface together with a full pair of
laminations, and let $\ws_h$ be the universal cover of $S_h$.
Following \cite{cannon-thurston}, we define an infinitesimal
pseudo-metric on $\ws_h \times \RR$ by
\begin{equation}\label{eq:ct metric}
ds^2 = k^{2z} dx^2 + k^{-2z} dy^2 + (\log k)^2
dz^2. 
\end{equation} 
Here $\log$ will mean the natural log base $e$ and we will write
$\log_k$ for $\log$ base $k$.  Throughout this paper, we will use
$k > 1$ exclusively to refer to the constant in the definition of the
Cannon-Thurston metric.  If the pair of laminations is the pair of
invariant laminations determined by a pseudo-Anosov map $\pa$, we will
choose $k > 1$ to be the stretch factor of $\pa$.  With this choice,
the monodromy map acts on the universal cover by vertical translation
by one unit.  For an arbitrary full pair of laminations, we may choose
$k = e$.

\medskip The infinitesimal pseudometric gives rise to a
pseudometric $d_{\widetilde S_h \times \RR}$ on
$\widetilde{S}_h \times \RR$ in the standard manner:
\begin{itemize}
    \item integrating the pseudometric along rectifiable paths gives a pseudo-distance; and then
    \item defining the distance between two points as the infimum of the length over all rectifiable paths connecting the two points.
\end{itemize}
The resulting pseudo-metric on $\widetilde{S}_h \times \RR$ is called
the \emph{Cannon--Thurston metric}.  In the mapping torus case, the pseudo-metric is
$\pi_1(M)$-invariant by construction.  It is genuine pseudo-metric since if $p$ and $q$ are points in the same compact
complementary region of
$\widetilde S_h \setminus (\Lambda_+ \cup \Lambda_-)$, then for any
$z_0 \in \RR$, the corresponding points $(p,z_0)$ and $(q, z_0)$, with the
same $z$-coordinate $z_0$, are distance zero apart in
$\widetilde S_h \times \RR$.

\begin{theorem}\cite{cannon-thurston}*{Theorem 5.1}
Let $f \colon S \to S$ be a pseudo-Anosov map, and let $(S_h, \LL)$ be
a hyperbolic metric on $S$ together with a pair of invariant measured
laminations for $\pa$, and let $M$ be the corresponding fibered $3$-manifold. 
Then the hyperbolic metric $d_{\HH^3}$ and the $\pi_1(M)$–invariant global
pseudometric
$d_{\widetilde S_h \times \RR} = \inf_\gamma \int_\gamma ds$ are
quasi-isometric.
\end{theorem}

If $(S_h, \LL)$ is a full pair of laminations, then Cannon-Thurston
metric on $\wsr$ is Gromov hyperbolic, as the vertical flow lines
satisfy the flaring condition from the Bestvina-Feighn Combination
Theorem \cite{bestvina-feighn}.

\begin{theorem}\cite{bestvina-feighn}*{page 88}
Let $(S, \LL)$ be a hyperbolic metric on $S$ together with a full pair
of laminations.  Then the Cannon-Thurston metric on $\wsr$ is Gromov
hyperbolic.
\end{theorem}

The Cannon-Thurston metric on $\wsr$ is quasi-isometric to $\HH^3$ if
and only if the pair of laminations have bounded geometry, by work of
Rafi \cite{rafi}.  In the bounded geometry case, the Cannon-Thurston
metric is $\pi_1 S$-equivariantly quasi-isometric to $\HH^3$, by the
proof of the Ending Lamination Conjecture due to Minsky \cites{MinskyJAMS, minsky}
and Brock, Canary and Minsky \cite{brock-canary-minsky}.

\medskip
The restriction of the infinitesimal pseudo-metric to the base fiber $S_0$ defines a pseudo-metric on $\widetilde S_h$.
This metric is quasi-isometric to the hyperbolic metric $d_{\HH^2}$, see \Cref{section:flat} for further details.

\medskip
Moving up in the $z$-direction expands distances in the
$x$-direction and contracts them (by the reciprocal) in the $y$-direction.  
This is illustrated in \Cref{fig:solv rectangles}, where the horizontal lines are leaves of $\Lambda_-$ and vertical lines are leaves of
$\Lambda_+$.  

\medskip
The constant $\log k$ is chosen so that the map
given by $(u, v) \mapsto (u, k^{-v})$ is an isometry from $\RR^2$ with the metric $ds^2 = k^{2v} du^2 + (\log k)^2 dv^2$ to the upper half space model of $\HH^2$ with the hyperbolic metric.  

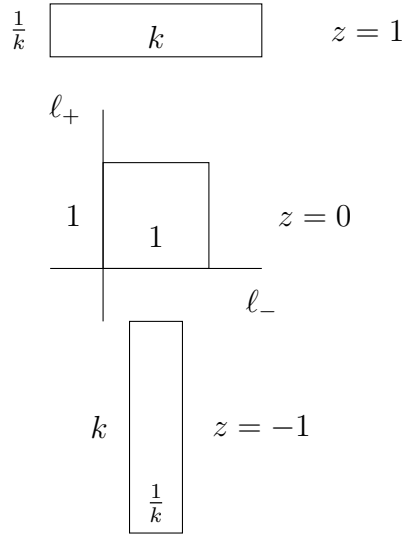
\begin{figure}[h]
\begin{center}
\begin{tikzpicture}[scale=0.7]

\tikzstyle{point}=[circle, draw, fill=black, inner sep=1pt]

\draw (0, 0) rectangle (2, 2);
\draw (4, 1) node {$z = 0$};

\draw (1, 0) node [label=above:$1$] {};
\draw (0, 1) node [label=left:$1$] {};

\draw (-1, 0) -- (3, 0) node [label=below:$\ell_-$] {};
\draw (0, -1) -- (0, 3) node [label=left:$\ell_+$] {};

\begin{scope}[xshift=-1cm, yshift=+4cm, yscale=0.5, xscale=2]
\draw (0, 0) rectangle (2, 2);
\draw (3, 1) node {$z = 1$};

\draw (1, -0.5) node [label=above:$k$] {};
\draw (0, 1) node [label=left:$\frac{1}{k}$] {};
\end{scope}

\begin{scope}[xshift=+0.5cm, yshift=-5cm, yscale=2, xscale=0.5]
\draw (0, 0) rectangle (2, 2);
\draw (5, 1) node {$z = -1$};

\draw (1, -0.1) node [label=above:$\frac{1}{k}$] {};
\draw (0, 1) node [label=left:$k$] {};
\end{scope}

\end{tikzpicture}
\end{center}
\caption{Rescaling arising from the vertical flow in the Cannon-Thurston metric.} \label{fig:solv rectangles}
\end{figure}

\medskip Suppose that $\ell_+$ is a leaf of $\Lambda^+$.  The ladder
$F(\ell_+)$ is then parametrized by the coordinates $(y, z)$ in
$\widetilde S_h \times \RR$.  By the definition of the pseudometric,
$F(\ell_+)$ is a convex subset of $\widetilde S_h \times \RR$.
Moreover, $F(\ell_+)$ is quasi-isometric to $\HH^2$, where in the
upper half space model the quasi-isometry is given by
$(y, z) \mapsto (y, k^{-z})$.  The leaf $\ell_+$ is a coarse
horocycle, as is each image $F_z(\ell_+)$.  The suspension flow lines
are geodesics and the distance between two suspension flow lines
decreases exponentially as the $z$-coordinate increases.  Thus, as
$z \to +\infty$, all suspension flow lines converge to the same limit
point at infinity, as illustrated in \Cref{fig:ladders}.

\medskip Similarly, if $\ell_-$ is a leaf of $\Lambda_-$, then the
ladder $F(\ell_-)$ is parametrized by coordinates $(x, z)$ in
$\widetilde S_h \times \RR$.  The ladder $F(\ell_-)$ is again a convex
subset of $\widetilde S_h \times \RR$ quasi-isometric to $\HH^2$,
though in this case the quasi-isometry to the upper half space is
given by $(x, z) \mapsto (x, k^{z})$.  The images of the leaf
under the suspension flow, namely the $F_z(\ell_-)$, are coarse horocycles.  The
suspension flow lines are geodesics, and the distance between two
suspension flow lines decreases exponentially as the $z$-coordinate
decreases.  Thus, as $z \to - \infty$, all suspension flow lines
converge to the same limit point at infinity.

\begin{figure}[h]
\begin{center}
\begin{tikzpicture}[scale=0.85]

\tikzstyle{point}=[circle, draw, fill=black, inner sep=1pt]

\draw (0, 0) -- (0, 4);

\draw (-1, 1) -- (3, -1) node [label=left:${\ell_+ \in (\Lambda_+, dx)}$] {};

\draw (0, 3) -- (2, 2) -- (2, -0.5) node [midway, label=right:$F(\ell_+)$] {};

\begin{scope}[decoration={
    markings,
    mark=at position 0.5 with {\arrow{>}}}]
\draw [color=ForestGreen, postaction={decorate}] (0.5, 0.25) -- (0.5, 2.75);
\draw [color=ForestGreen, postaction={decorate}] (1, 0) -- (1, 2.5);
\draw [color=ForestGreen, postaction={decorate}] (1.5, -0.25) -- (1.5, 2.25);

\begin{scope}[xscale=-1]
\draw [color=red, postaction={decorate}] (0.5, 0.25) -- (0.5, 2.75);
\draw [color=red, postaction={decorate}] (1, 0) -- (1, 2.5);
\draw [color=red, postaction={decorate}] (1.5, -0.25) -- (1.5, 2.25);
\end{scope}

\draw (7, 1) circle (2cm);
\draw [color=ForestGreen, postaction={decorate}] ([shift=(270:2cm)]5,3) arc (270:360:2cm);
\draw [color=ForestGreen, postaction={decorate}] ([shift=(270:2cm)]9,3) arc (270:180:2cm);
\draw [color=ForestGreen, postaction={decorate}] (7, -1) -- (7, 3);

\draw (1, 1) -- (-3, -1) node [label=right:${\ell_- \in (\Lambda_-, dy)}$] {};

\draw (0, 3) -- (-2, 2) -- (-2, -0.5) node [midway, label=left:$F(\ell_-)$] {};

\begin{scope}[xshift=-14cm, yscale=-1, yshift=-2cm]

\draw (7, 1) circle (2cm);
\draw [color=red, postaction={decorate}] (7, 3) -- (7, -1);
\draw [color=red, postaction={decorate}] ([shift=(360:2cm)]5,3) arc (360:270:2cm);
\draw [color=red, postaction={decorate}] ([shift=(180:2cm)]9,3) arc (180:270:2cm);

\end{scope}

\end{scope}

\draw (9, -1) node {$F(\ell_+)$};

\draw (-5, -1) node {$F(\ell_-)$};

\end{tikzpicture}
\end{center}
\caption{Ladders over leaves are quasi-isometric to $\HH^2$.} \label{fig:ladders}
\end{figure}
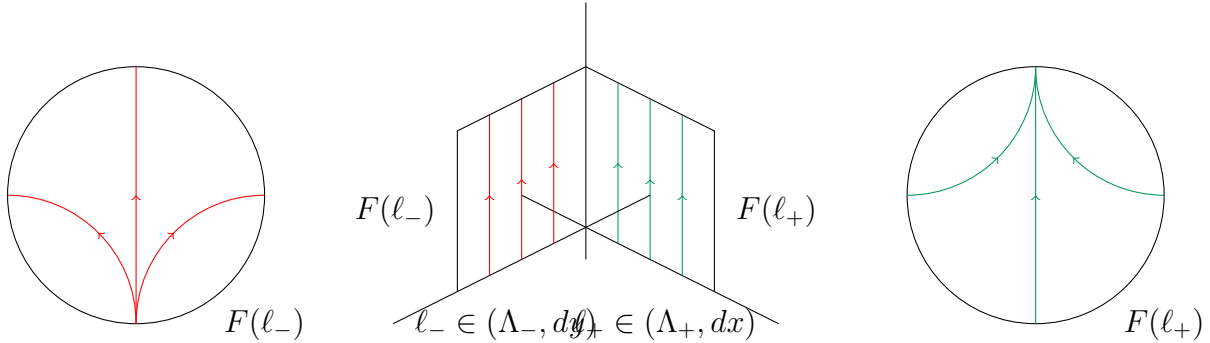

\medskip
In fact, ladders over arbitrary geodesics in $\widetilde S_h$ are quasiconvex, a special case of a more general result of
Mitra \cite{mitra}. 
We will state Mitra's result using the notation of this paper.

\begin{theorem}\cite{mitra}*{Lemma 4.1}
Suppose that $\pa$ is a pseudo-Anosov map and $\widetilde S_h$ is a hyperbolic metric. Then, there
is a constant $K$, such that for any geodesic $\gamma$ in
$\widetilde S_h$, the ladder $F(\gamma)$ is $K$-quasiconvex in
$\widetilde S_h \times \RR$.
\end{theorem}

Ladders over arbitrary geodesics are also quasi-isometric to $\HH^2$, though we do not use this fact directly.

\subsection{Separation for ladders}\label{section:separation}

Given two leaves $\ell$ and $\ell'$ of an invariant lamination, we define the distance
between their ladders to be
\[ d_{\widetilde S_h \times \RR} ( F(\ell), F(\ell') ) = \inf
\left\{ d_{\widetilde S_h \times \RR}(p, p') \mid p \in F(\ell), p'
\in F(\ell') \right\} .  \]
In this section, we show that there is
an $\epsilon = \epsilon_\LL > 0$ such that for any pair of ladders, the Cannon--Thurston distance
between them is either zero, or at
least $\epsilon$.  Furthermore, we prove that the limit sets in
$S^2_\infty = \partial \HH^3$ of any two ladders $F(\ell)$ and
$F(\ell')$ intersect if and only if they are distance zero apart.

\medskip
As the first step, we show that the distance is zero with the infimum
realized if and only if the base of the ladders is a pair of leaves
that are boundary leaves of a common ideal complementary region.

\begin{proposition}\label{prop:boundary leaves}
Suppose that $(S_h, \Lambda)$ is a hyperbolic structure on $S$
together with a full pair of measured laminations.  Suppose that
$\ell$ and $\ell'$ are leaves of the same lamination.  Then
there are points $p \in F(\ell)$ and $p' \in F(\ell')$ with
$d_{\widetilde S_h \times \RR}(p, p') = 0$ if and only if $\ell$ and
$\ell'$ are boundary leaves of an ideal complementary region.
\end{proposition}

\begin{proof}
Suppose that $\ell$ and $\ell'$ are boundary leaves of an ideal
complementary region.  Then we can find a subarc in $\widetilde S_h$
that connects $\ell$ and $\ell'$ such that its interior is disjoint from both laminations.
The subarc then has length zero in the pseudo-metric, and so the distance between the
ladders $F(\ell)$ and $F(\ell')$ is zero.

\medskip Conversely, suppose there is a path in $\widetilde S_h \times \RR$
between $p \in F(\ell)$ and $p' \in F(\ell')$ that has zero length
in the pseudo-metric.  By the definition of the
pseudo-metric, this path lies in a fiber
$\widetilde S_h \times \{ z \}$ and its interior is disjoint from the
images (by the vertical flow $F_z$) of the laminations.
This implies that $\ell$ and $\ell'$ are boundary leaves of a common ideal complementary region.
\end{proof}

For the remainder of this section, we set up some terminology. 
Suppose that $\ell$ and $\ell'$ are leaves of an invariant lamination. 
We denote the open strip in $\widetilde S_h$ with two boundary
components $\ell$ and $\ell'$ by $R$.
We denote by $A(R)$ the collection of arcs in $\widetilde S_h$ that have one endpoint on $\ell$, the other endpoint on $\ell'$, and interior in $R$.
The limit set of $R$ is a disjoint union of
two intervals, possibly with one of them a single point. 
We call these intervals $I$ and $I'$. 
Let $\textrm{Sep}_+ (\ell, \ell')$ be the set of leaves in $\Lambda_+$ that separate $\ell$ from $\ell'$, that is, each leaf in $\textrm{Sep}_+ (\ell, \ell')$ is a leaf of $\Lambda_+$ that has one limit point in $I$, the other limit point in $I'$.
Similarly, let $\textrm{Sep}_- (\ell, \ell')$ of leaves in $\Lambda_-$ that separate $\ell$ from $\ell'$.

\begin{lemma}\label{lem:no common ideal region}
    Suppose that $\ell$ and $\ell'$ are distinct leaves in $\Lambda_+$. Then $\textrm{Sep}_+ (\ell, \ell')$ is non-empty if and only if $\ell$ and $\ell'$ are not boundary leaves of a common ideal complementary region of $\Lambda_+$.
\end{lemma}

\begin{proof}
    If $\textrm{Sep}_+ (\ell, \ell')$ is non-empty then $\ell$ and $\ell'$ cannot be boundary leaves of a common ideal complementary region of $\Lambda_+$. 

\medskip
Conversely, suppose that $\ell$ and $\ell'$ are not boundary leaves of
a common ideal complementary region.  Recall that $I$ and $I'$ are the
two limit sets with one endpoint in $\ell$, and the other endpoint in
$\ell'$.  As the two geodesics $\ell$ and $\ell'$ are distinct, at
least one of $I$ and $I'$ has non-empty interior.

\medskip If the interval $I$ consists of a single point, its convex hull $C_I$
is equal to $I$.  If the interval $I$ has non-empty interior, then the
convex hull $C_I \subset R$ consists of the limit set $I$, together with all bi-infinite geodesics with both endpoints in $I$. In particular, the boundary of $C_I$ is the bi-infinite geodesic $\alpha$ connecting the endpoints of $I$.
Similarly, we denote by $C_{I'}$ the convex hull of $I'$.  Again, if $I'$
has non-empty interior, we denote by $\alpha'$ the bi-infinite geodesic
connecting the endpoints of $I'$.

\medskip By convexity,
any ideal complementary region of $\Lambda_+$ with all its ideal
vertices in $I$ is contained in $C_I$. Similarly, any
ideal complementary region of $\Lambda_+$ with all its ideal vertices
in $I'$ is contained in $C_{I'}$.

\medskip The geodesics $\ell$ and $\ell'$, together with the geodesics $\alpha$ and $\alpha'$, form an ideal quadrilateral $T \subseteq R$ with at least
three limit points, and so $T$ has non-empty interior. Therefore, $T$
must intersect an ideal complementary region U of $\Lambda_+$ contained
in $R$. By construction, this complementary region has ideal vertices in both $I$ and $I'$.
We then find exactly two boundary leaves $\ell_1$
and $\ell_2$ of $U$ connecting $I$ to $I'$. If, as unordered pairs
$(\ell_1, \ell_2) = (\ell, \ell')$, then $\ell$ and $\ell'$ are
boundary leaves of a single ideal complementary region of $\Lambda_+$,
a contradiction. We deduce that at least one of $\ell_1$ or $\ell_2$
is contained in $\textrm{Sep}_+ (\ell, \ell')$, and thus
$\textrm{Sep}_+ (\ell, \ell')$ is non-empty, as required.
\end{proof}

By switching the invariant laminations, \Cref{lem:no common ideal region} also holds for $\Lambda_-$ with $\textrm{Sep}_- (\ell, \ell')$ non-empty if and only if $\ell$ and $\ell'$ are not boundary leaves of a common ideal complementary region of $\Lambda_-$.

\begin{lemma}\label{lem:separating leaves non-empty}
    Suppose that $\ell$ and $\ell'$ are leaves of an invariant lamination. Then the subset $\textrm{Sep}_+ (\ell, \ell')$ is non-empty if and only if the $dx$-measure of $\textrm{Sep}_+ (\ell, \ell')$ is positive. Similarly, $\textrm{Sep}_- (\ell, \ell')$ is non-empty if and only if the $dy$-measure of $\textrm{Sep}_- (\ell, \ell')$ is positive.
\end{lemma}

\begin{proof}
If the $dx$-measure of $\textrm{Sep}_+ (\ell, \ell')>0$ then $\textrm{Sep}_+ (\ell, \ell')$ is non-empty by the definition of transverse measure.
Conversely, suppose that $\textrm{Sep}_+ (\ell, \ell')$ is non-empty and so there is a transverse arc contained in $A(R)$ that intersects in an interior point a leaf in $\textrm{Sep}_+ (\ell, \ell')$. 
Since an invariant lamination has no isolated leaves, it follows that such an arc has positive $dx$-measure. 
Since the $dx$-measure of the arc is a lower bound on the $dx$-measure of $\textrm{Sep}_+ (\ell, \ell')$, the lemma follows.
\end{proof}

\begin{lemma}\label{lem:separating leaves and ladder distance}
Suppose that $\ell$ and $\ell'$ are leaves of an invariant lamination. Then $d_{\widetilde S_h \times \RR}( F(\ell), F(\ell') ) > 0$ if and only if the $dx$-measure of $\textrm{Sep}_+ (\ell, \ell')$ and the $dy$-measure of $\textrm{Sep}_- (\ell, \ell')$ is positive. 
\end{lemma}

\begin{proof}
    We let $a$ be the $dx$-measure of $\textrm{Sep}_+ (\ell, \ell')$ and $b$ the $dy$-measure of $\textrm{Sep}_- (\ell, \ell')$. 
    Suppose that $\gamma$ is an arc in $A(R)$. Then $\gamma$ must intersect in its interior every leaf in $\textrm{Sep}_+ (\ell, \ell')$ and every leaf in $\textrm{Sep}_- (\ell, \ell')$. It follows that $dx (\gamma) \geqslant a$ and $dy (\gamma) \geqslant b$. 

    \medskip
    Suppose $p$ and $p'$ are points on $\ell$ and $\ell'$ respectively. By definition of the Cannon-Thurston metric, one of the two distances $d_{\widetilde S_h \times \RR} ((p, z), (p',z))$ or $d_{\widetilde S_h \times \RR} ((p,z'), (p', z'))$ is at most the distance $d_{\widetilde S_h \times \RR}((p,z), (p', z'))$.
    Thus, we may assume that $z = z'$, that is, the pair of points are at the same height. Suppose that $\gamma$ is an arc in $A(R)$ between $p$ and $p'$.
    Then the length of $F_z(\gamma)$ equals $k^z dx(\gamma) + k^{-z} dy (\gamma) \geqslant k^z a + k^{-z}b$. The lemma now follows.
\end{proof}

We now show that the distance is zero but the infimum is not attained
if and only if the ladders are over leaves that intersect a
complementary region of the other lamination.

\begin{proposition}\label{prop:common intersection leaf}
Suppose that $(S_h, \Lambda)$ is a hyperbolic metric on $S$ together
with a full pair of measured laminations.  Suppose that $\ell$ and
$\ell'$ are leaves of the same lamination and suppose that
$d_{\widetilde S_h \times \RR} (p, p') > 0$ for any pair of points
$p \in F(\ell)$ and $p' \in F(\ell')$. Then the distance
$d_{\widetilde S_h \times \RR}( F(\ell), F(\ell') ) = 0$ if and
only if $\ell$ and $\ell'$ are not boundary leaves of a common
ideal complementary region but intersect a common ideal complementary
region of the other lamination.
\end{proposition}

\begin{proof}
Breaking symmetry, we may assume that $\ell$ and $\ell'$ are leaves of $\Lambda_+$. The same argument holds for $\Lambda_-$ by switching the laminations.

\medskip Suppose that $d_{\widetilde S_h \times \RR}( F(\ell), F(\ell') ) = 0$ but $d_{\widetilde S_h \times \RR} (p, p') > 0$ for any pair of points
$p \in F(\ell)$ and $p' \in F(\ell')$. By \Cref{prop:boundary leaves}, $\ell$ and $\ell'$ are not boundary leaves of a common ideal complementary region of $\Lambda_+$. By \Cref{lem:no common ideal region}, $\textrm{Sep}_+ (\ell, \ell')$ is non-empty. If $\textrm{Sep}_- (\ell, \ell')$ is also non-empty, then by \Cref{lem:separating leaves and ladder distance}, $d_{\widetilde S_h \times \RR}( F(\ell), F(\ell') ) > 0$, a contradiction. Thus $\textrm{Sep}_- (\ell, \ell')$ is empty. This means that there is an arc $\gamma$ in $A(R)$ with endpoints $p$ on $\ell$ and $p'$ on $\ell'$ such that the interior of $\gamma$ intersects only $\Lambda_+$. But then $\gamma$ is contained in a single ideal complementary region of $\Lambda_-$, as required. 

\medskip Conversely, suppose that $\ell$ and $\ell'$ are not boundary leaves of a single ideal complementary region of $\Lambda_+$ but intersect a common ideal complementary region of $\Lambda_-$. By \Cref{lem:no common ideal region}, $\textrm{Sep}_+ (\ell, \ell')$ is non-empty. Since any arc $\gamma$ in $A(R)$ intersects $\textrm{Sep}_+ (\ell, \ell')$, the length of $F_z(\gamma)$ is at least $k^z$ times the $dx$-measure of $\textrm{Sep}_+ (\ell, \ell')$. In particular, this implies that $d_{\widetilde S_h \times \RR} (p, p') > 0$ for any pair of points $p \in F(\ell)$ and $p' \in F(\ell')$.
On the other hand, since $\ell$ and $\ell'$ intersect a common ideal complementary region of $\Lambda_-$, there is an arc $\gamma$ in $A(R)$ with endpoints $p$ on $\ell$ and $p'$ on $\ell'$ such that the interior of $\gamma$ intersects only $\Lambda_+$. Then the length of $F_z(\gamma)$ equals $k^z dx (\gamma)$ which goes to zero as $z \to - \infty$. Thus, $d_{\widetilde S_h \times \RR}( F(\ell), F(\ell') ) = 0$, as required. 
\end{proof}

Finally, we show the distance gap for ladders.

\begin{proposition}\label{prop:separated ladders}
Suppose that $(S_h, \Lambda)$ is a hyperbolic structure on $S$
together with a full pair of measured laminations.  Then there is
a constant $\e_3 > 0$ such that for any two leaves $\ell_1$ and
$\ell_2$ of a lamination, either
$d_{\widetilde S_h \times \RR}(F(\ell_1), F(\ell_2)) \ge \e_3$, or
else $d_{\widetilde S_h \times \RR}(F(\ell_1), F(\ell_2)) = 0 $.
\end{proposition}

\begin{proof}
Suppose that there is a sequence of pairs of leaves $\ell_n$ and
$\ell'_n$ of a lamination such that
$d_{\widetilde{S}_h \times \RR} (F(\ell_n), F(\ell'_n)) >0 $ and tends
to zero as $n \to \infty$.  Breaking symmetry, we assume that $\ell_n$
and $\ell'_n$ are leaves of $\Lambda_+$.  Since $\Lambda_+$ is a
closed set, we may, by passing to a subsequence, assume that $\ell_n$
and $\ell'_n$ converge to leaves $\ell$ and $\ell'$ respectively. It
follows that $d_{\widetilde{S}_h \times \RR} (F(\ell), F(\ell')) = 0$.

\medskip By \Cref{prop:boundary leaves} and \Cref{prop:common intersection leaf}, we can find an arc $\alpha$ in $A(R)$ with endpoints $p$ on $\ell$ and $p'$ on $\ell'$ such that either the interior of $\alpha$ is disjoint from both laminations, or the interior of $\alpha$ intersects only $\Lambda_+$.

\medskip Suppose that $q_n$ on $\ell_n$ and $q'_n$ on $\ell'_n$ are sequences of points that converge to $p$ and $p'$. It follows that by choosing $q_n$ and $q'_n$ sufficiently close to $p$ and $p'$ we can find an arc $\alpha_n$ with endpoints $q_n$ and $q'_n$ such that the interior of $\alpha_n$ intersects only $\Lambda_+$. But then by \Cref{prop:common intersection leaf}, $d_{\widetilde{S}_h \times \RR}  (F(\ell_n), F(\ell'_n)) = 0$, a contradiction. 

\end{proof}

Finally, we show that the limit sets of two ladders intersect if and
only if they are distance zero apart.

\begin{proposition}\label{prop:disjoint limits}
Suppose that $(S_h, \Lambda)$ is a hyperbolic metric on $S$ together
with a full pair of measured laminations.  For any leaves $\ell$
and $\ell'$ in a lamination,
$\overline{F(\ell)} \cap \overline{F(\ell')} \not = \varnothing$ if
and only if $d_{\widetilde S_h \times \RR}( F(\ell), F(\ell')) = 0$.
\end{proposition}

\begin{proof}
Suppose that $\ell$ and $\ell'$ are in $\Lambda_+$ and $d_{\widetilde S_h \times \RR}( F(\ell), F(\ell')) = 0$.
By \Cref{prop:boundary leaves} and \Cref{prop:common intersection
  leaf}, $\ell$ and $\ell'$ are either boundary leaves of an ideal
complementary region of $\Lambda_+$, or else intersect an ideal complementary region of $\Lambda_-$. 
It follows that there are points $p \in \ell$ and $p' \in \ell'$ and an arc $\alpha$ in $A(R)$ with endpoints $p$ and $p'$ such that the interior of $\alpha$ is either disjoint from both laminations or intersects only $\Lambda_+$. 

\medskip If the interior is disjoint from both laminations then $F_z(p)$ and $F_z(p')$ are pseudo-metric distance zero for all $z$. Thus, the flow lines determine the same limit points at infinity.

\medskip Now suppose that the interior of $\alpha$ intersects $\Lambda_+$. Then $dx(\alpha) > 0$ and $dy(\alpha) = 0$. Then the distance between $F_z(p)$ and $F_z (p')$ is $k^z dx(\alpha)$ and hence the flow lines determine the same point at infinity as $z \to - \infty$.

\medskip Conversely, suppose
$\overline{F(\ell)} \cap \overline{F(\ell')}$ is non-empty and let $z_\infty$ be a point of the intersection. Every
limit point in $\overline{F(\ell)}$ (similarly
$\overline{F(\ell')}$) is a limit point of a suspension flow line, and hence there are points $p \in \ell$ and $p' \in \ell'$
such that their suspension flow lines $F(p)$ and $F(p')$ converge to $z_\infty$ in one direction. The Cannon-Thurston metric is
$\delta$-hyperbolic and suspension flow lines are geodesics.  We
deduce that $F(p)$ and $F(p')$ are bounded distance in the direction
of the common limit point.  It follows that there is an arc $\beta$ in $A(R)$ with endpoints $p$ and $p'$ such that the interior of $\beta$ intersects only one of the laminations. Thus, the distance between $F(\ell)$ and
$F(\ell')$ is zero, as desired. 
\end{proof}

\subsection{Quasigeodesics}

We recall some basic facts about quasigeodesics, see for example
Bridson and Haefliger \cite{bh}*{III.H}.

\begin{definition}
Let $(X, d)$ be a geodesic metric space and let
$\gamma \colon I \to X$ be a path, where $I$ is a (possibly infinite)
connected subset of $\RR$.  Let $Q \ge 1$ and $c \ge 0$ be constants.
The path $\gamma$ is a \emph{(Q, c)-quasigeodesic} if for all $t_1$
and $t_2$ in $I$,
\[ \frac{1}{Q} | t_2 - t_1 | - c \le d( \gamma(t_1), \gamma(t_2) )
\le Q |t_2 - t_1| + c.  \]
\end{definition}

By \cite{bh}*{III.H Lemma 1.11}, given a $(Q, c)$-quasigeodesic, there
is a continuous $(Q, c')$-quasigeodesic with the same endpoints, so
for our purposes we may assume that all quasigeodesics are continuous,
and we will do so from now on.

\medskip
A \emph{reparametrization} of a path $\gamma \colon \RR \to X$ is the
path $\gamma \circ \rho$, where $\rho \colon \RR \to \RR$ is a proper
non-decreasing function.  We say a path $\gamma \colon \RR \to X$ is
an \emph{unparametrized} $(Q, c)$-quasigeodesic if there is a
reparametrization of $\gamma$ which is a $(Q, c)$-quasigeodesic.

\medskip
We will use the following stability property for quasigeodesics in
hyperbolic spaces, known as the Morse Lemma.

\begin{lemma}\cite{bh}*{Theorem 1.7}\label{lemma:morse}
Let $X$ be a $\delta$-hyperbolic space.  Then for any $Q$ and $c$
there is a constant $L$ such that any $(Q, c)$-quasigeodesic is
contained in an $L$-neighborhood of the geodesic connecting its
endpoints.
\end{lemma}

\subsection{Nearest point projections and fellow traveling}

Suppose that $\alpha$ is a subset of a Gromov hyperbolic space
$(X, d)$.  The \emph{nearest point projection}
$p_\alpha \colon X \to \alpha$ sends each point $x \in X$ to a closest point to $x$ in $\alpha$.  If $\alpha$ is $Q$-quasiconvex,
then the nearest point projection is $K$-coarsely well defined, where
$K$ depends only on $Q$ and the constant $\delta$ of hyperbolicity.  

\medskip Suppose that $\alpha$ and $\beta$ are two geodesics in $X$. 
We define the \emph{$K$-fellow traveling set} for $\alpha$ with respect to $\beta$ to be the subset
of $\alpha$ contained in a $K$-neighborhood of $\beta$, i.e.
$\alpha \cap N_K(\beta)$.  If the diameter of the projection image
$p_\alpha(\beta)$ is sufficiently large, then $p_\alpha(\beta)$ is
contained in a bounded neighborhood of the geodesic $\beta$, and so is
contained in a fellow traveling set.

\medskip
In the special case that $X$ is $\HH^n$ and $\alpha$
is a geodesic, the closest point on $\alpha$ to any point $x \in X$ is
unique, and so $p_\alpha$ is well-defined.  Furthermore, for any geodesic
$\beta$, the image $p_\alpha(\beta)$ is a subinterval of $\alpha$,
which we will refer to as the \emph{nearest point projection
  interval}, or just the \emph{projection interval}.  Similarly, any
$K$-fellow traveling set $\alpha \cap N_K(\beta)$ is also an
interval, which we shall call the \emph{$K$-fellow traveling
  interval}.

\medskip

For $\HH^n$, the hyperbolicity constant is $\delta = 2 \log 3$.  We shall write $\delta_2$ for the hyperbolicity constant for the pseudometric
$d_{\widetilde S_h}$ on $\widetilde S_h$, and $\delta_3$ for the
hyperbolicity constant for the pseudometric
$d_{\widetilde S_h \times \RR}$ on $\widetilde S_h \times \RR$.  Both
constants depend on the pseudo-Anosov $\pa$.

\medskip
A standard consequence of $\delta$-hyperbolicity is the useful property stated below that if the projection image of a geodesic $\beta$
onto another geodesic $\alpha$ is large, then the two geodesics fellow
travel and the projection image $p_\alpha(\beta)$ is contained in a
bounded neighborhood of $\beta$.

\begin{proposition} \cite{kapovich-sardar}*{Lemma 1.120}
\label{prop:neighbourhood}
Suppose that $X$ is a $\delta$-hyperbolic space and suppose that
$\alpha$ and $\beta$ are geodesics in $X$.  If the diameter of the
projection image $p_\alpha(\beta)$ is greater than $8 \delta$, then
the projection image is contained in a $6 \delta$-neighborhood of
$\beta$, i.e. $p_\alpha(\beta) \subseteq N_{6 \delta}(\beta)$, and so
$p_\alpha(\beta)$ is contained in the fellow traveling set
$\alpha \cap N_{6 \delta}(\beta)$.
\end{proposition}

In $\HH^2$, if two geodesics intersect at angle $\theta$, then the
size of the projection interval is roughly $\log (1/\theta)$.  In
fact,the same result holds for two geodesics that do not intersect,
but are distance $\theta$ apart.  

\begin{proposition}\label{prop:projection interval}\label{c:T_0}
There is a constant $T_0 \ge 0$, such that for any unit speed geodesic
$\gamma_1$ in $\HH^2$ and any geodesic $\gamma_2$ such that

\begin{itemize}

\item $\gamma_2$ intersects $\gamma_1$ at the point $\gamma_1(0)$ at
an angle $0 < \theta \le \pi/2$, or

\item the distance from $\gamma_1$ to $\gamma_2$ is $\theta > 0$, and
the closest point occurs at $\gamma_1(0)$,

\end{itemize}

then the nearest point projection interval $p_{\gamma_1}(\gamma_2)$
is equal to $\gamma_1( [-T, T] )$, where
\[ \log \frac{1}{\theta} \le T \le \log \frac{1}{\theta} +
T_0, \]
and furthermore, for all $|t| \le \log \tfrac{1}{\theta}$, the
distance from $\gamma_1(t)$ to $\gamma_2$ is at most $3/2$.
\end{proposition}
This is well known, we provide the details in \cite{gmpu}*{Appendix A}.
In fact,
for small $|t|$, the two geodesics are exponentially close, see
\cite{gmpu}*{Appendix A}
for further details.

\medskip We now record the useful fact that if two geodesics $\alpha$ and
$\beta$ intersect at angle $\theta$, then the size of their projection
intervals onto each other is roughly $\log \tfrac{1}{\theta}$, and
furthermore, for any other geodesic $\gamma$, the overlap between the
projection intervals for $\alpha$ and $\beta$ on $\gamma$ is bounded
in terms of $\theta$.

\begin{proposition}\label{prop:overlap}
For any constant $\alpha_\LL > 0$ there is a constant $\rho_\LL > 0$
such that for any two geodesics in $\HH^2$ which intersect at angle
$\theta \ge \alpha_\LL$, and for any other geodesic $\gamma$, the
intersection of the nearest point projection intervals of $\alpha$ and $\beta$ to $\gamma$
has diameter at most $\rho_\LL$.
\end{proposition}

\begin{proof}
Suppose $\alpha$ and $\beta$ intersect at the point $p$ with angle
$\theta \ge \alpha_\LL$.  We shall choose
$\rho_\LL = 2 \log \tfrac{1}{\alpha_\LL} + 4 T_0 + 16$.  By
\Cref{prop:projection interval}, the radius of the nearest point
projection interval of $\alpha$ to $\beta$, and also of $\beta$ to
$\alpha$, is at most $\log \tfrac{1}{\alpha_\LL} + T_0$.

\medskip
For any geodesic $\gamma$, let $I_\alpha = p_\gamma(\alpha)$ and $I_\beta = p_\gamma(\beta)$ be the nearest
point projection intervals of $\alpha$ and $\beta$ onto $\gamma$. 
Suppose that their overlap has size at least $\rho_\LL$, i.e. the length of
$I_\alpha \cap I_\beta$ is at least $\rho_\LL$.  If we truncate $I_\alpha$ and $I_\beta$ by length $T_0$ at both ends, then the truncated
intervals have overlap of length at least
$\rho_\LL - 2 T_0$.  Let $I$ be the interval of overlap for the truncated projection intervals, i.e.
$I$ is the closure of
$I_\alpha \cap I_\beta \setminus N_{T_0}(\partial ( I_\alpha \cap
I_\beta ) )$.

\medskip
By \Cref{prop:projection interval}, each endpoint of $I$ is distance
at most $3/2$ from both $\alpha$ and $\beta$.  We denote by $a_1$ and $a_2$ the
points on $\alpha$ closest to each endpoint
of $I$. It follows that the distance between $a_1$ and $a_2$ is at least $|I| - 2 T_0 - 3$, and each point $a_i$ is distance at most $3$ from $\beta$.

\medskip
We pick points $b_i$ in $\beta$ distance at most 3 from the points $a_i$.
Then the nearest point projection of $b_i$ to $\alpha$
is distance at most $6$ from $a_i$.  In particular, the diameter of
the nearest point projection interval of $\beta$ to $\alpha$ is at
least $|I| - 2 T_0 - 15 \leqslant \rho_\LL - 2 T_0 - 15$.  It follows from our choice of $\rho_\LL$ that the diameter of the projection interval of $\beta$ onto
$\alpha$ is at least $2 \log \tfrac{1}{\alpha_\LL} + 2 T_0 + 1$, a
contradiction.
\end{proof}

Finally, we show that if two geodesics $\alpha$ and $\beta$ have
strictly nested projection intervals onto a third geodesic $\gamma$,
i.e. $p_{\gamma}(\alpha) \subset p_\gamma(\beta)$, and $\alpha$
intersects $\gamma$, then $\alpha$ and $\beta$ also intersect.

\begin{proposition}\label{prop:nested implies intersect}
Let $\gamma$ be a geodesic in $\HH^2$ which intersects a geodesic
$\ell_1$, with projection interval $I_1 \subset \gamma$.  Let $\ell_2$
be a geodesic with projection interval $I_2 \subset \gamma$, such that
$I_1 \subset I_2$.  Then $\ell_1$ and $\ell_2$ intersect.
\end{proposition}

\begin{proof}
Consider the nearest point projection map $p \colon \HH^2 \to \gamma$.
Consider the complement of the pre-image of $I_1$, i.e.
$\HH^2 \setminus p^{-1}(I_1)$.  This has two connected components,
which are separated by the geodesic $\ell_1$.  As $I_1$ is a strict
subset of $I_2$, each endpoint of $\ell_2$ is contained in a different
complementary component, and so the endpoints of $\ell_2$ are
separated by $\ell_1$, and so the two geodesics $\ell_1$ and $\ell_2$
intersect.
\end{proof}

\subsection{Lebesgue measures and the geodesic flow}

We review the properties of the geodesic flow we will use, see for
example \cite{einsiedler-ward}.

\medskip
A choice of basepoint $x_0$ in $\HH^n$ determines a measure on the
boundary sphere $\partial \HH^n$, for example by choosing the disc or
ball model for $\HH^n$ with the basepoint as the center point, and
giving $\partial \HH^n$ the measure induced from the standard metric
on the unit sphere.  We will call this measure \emph{Lebesgue
  measure}, and this measure depends on the choice of basepoint,
though we will suppress this from our notation. Different choices of
basepoint give measures which are absolutely continuous with respect to
each other.

\medskip
Let $\Gamma$ be a discrete cocompact subgroup of isometries of
$\HH^n$, and let $M = \HH^n / \Gamma$.  Hopf \cite{hopf} showed that
the geodesic flow $g_t$ on the unit tangent bundle $T^1(M)$ is ergodic
with respect to Liouville measure.  As $M$ is a compact hyperbolic
manifold of constant negative curvature, Liouville measure is
proportional to the Bowen-Margulis-Sullivan measure, the measure of
maximal entropy for the geodesic flow.  Liouville measure on $T^1(M)$
is the product of the measure determined by the hyperbolic metric on
$M$, with Lebesgue measure on the unit tangent spheres.  As the
conditional measures on each unit tangent sphere determined by
Liouville measure are Lebesgue measures, the measure on geodesics in
$M$ induced by Liouville measure is absolutely continuous with respect
to the measure on geodesics in $M$ induced by the product of Lebesgue
measures on $(\partial \HH^n \times \partial \HH^n) \setminus \Delta$.

\medskip
For the case of the Lebesgue surface measure, $n = 2$ and $\Gamma$ is
the fundamental group of the surface, $\pi_1 S$.  For the case of the
Lebesgue $3$-manifold measure, $n = 3$, and $\Gamma$ is the
fundamental group of the mapping torus, $\pi_1 M$.  In both cases, we
will use the fact that a geodesic chosen according to the product of
Lebesgue measures on $\partial \HH^n \times \partial \HH^n$ is
uniformly distributed in the unit tangent bundle of the compact
quotient manifold, almost surely.  We will use the following version
of this result.

\begin{proposition}
Let $\Gamma$ be a discrete cocompact group of isometries of $\HH^n$,
and let $B$ be a Borel set in $M = \HH^n / \Gamma$.  Then for almost
all geodesics $\gamma$ in $T^1(M)$,
\[  \lim_{T \to \infty} \frac{1}{T} \left| \gamma([0, T]) \cap B
\right| = \mathbf{vol}(B).  \]
\end{proposition}

\subsection{Random walks and hitting measures}

We recall some results in the theory of random walks on countable
groups acting on Gromov hyperbolic spaces.  Suppose that a
countable group $G$ acts on a $\delta$-hyperbolic space $X$.  The
results we state do not require $X$ to be locally compact or the
action to be locally finite.  But for our purposes,
$G \curvearrowright X$ will always be either
$\pi_1(S) \curvearrowright \widetilde S_h$ or
$\pi_1(M) \curvearrowright \widetilde S_h \times \RR$, which are
cocompact actions on locally compact spaces.  We say the action of a
group $G$ on a space $X$ is \emph{non-elementary} if it contains two
independent loxodromic isometries of $X$.  We say a probability
measure $\mu$ on $G$ is \emph{geometric} if it has finite exponential moment and
the semigroup generated by its support is equal to $G$.

\medskip A random walk of length $n$ on $G$ is a random product
$w_n = g_1 \cdots g_n$ where each $g_i$ is chosen independently
according to a probability measure $\mu$ on $G$.  We call the
elements $g_i$ the \emph{steps} of the random walk.  Passing to
infinitely many steps, we may consider the sequence of steps $(g_n)$
to be an element of $(G, \mu)^\ZZ$.  We call $(G, \mu)^\ZZ$ the
\emph{step space}.  The \emph{location} $w_n$ at time $n$ is
determined by $w_0 = 1 \in G$, and $w_{n+1} = w_n g_{n+1}$ for all
$n \in \ZZ$.  The \emph{location space} is the probability space
$(G^\ZZ, \mathbb{P})$, where $\mathbb{P}$ is the pushforward of the
product measure under the map that sends $(g_n)$ to $(w_n)$.  A choice
of basepoint $x_0 \in X$ gives a sequence $(w_n x_0)$ which we shall
also call a \emph{sample path} of the random walk.

\begin{theorem}\cite{kaimanovich}\cite{MaherTiozzo}
\label{theorem:boundary_convergence}
Suppose that $G$ is a countable group that has a non-elementary action on a Gromov hyperbolic space $X$.  Suppose that $\mu$ is a non-elementary
probability measure on $G$, that is, the support of $\mu$ generates a non-elementary subgroup of $G$ for its action on $X$. Then almost all (bi-infinite) sample paths $w = (w_n)_{n \in \ZZ}$ for the $\mu$-random walk on $G$, the sequences $(w_n x_0)$ converge as
$n \to \infty$ and $n \to - \infty$ to the Gromov boundary $\partial X$ of $X$.
The convergence defines a pair of non-atomic measures $\nu$ and $\rnu$ on
$\partial X$.
\end{theorem}

We call the measures on $\partial X$ obtained in
\Cref{theorem:boundary_convergence} the \emph{hitting measures} for
the random walk.  If $\mu$ is symmetric, then the forward and backward
hitting measures are equal, i.e. $\nu = \rnu$.

\medskip
For almost every sample path $w = (w_n)_{n \in \ZZ}$, the forward and backward limits $x^+_\infty = \lim_{n \to \infty} w_n x_0$ and $x^-_\infty = \lim_{n \to -\infty} w_n x_0$ are different from each other. Hence, almost every bi-infinite sample path $w$ defines a bi-infinite geodesic $\gamma_w$ in $X$ and all choices for $\gamma$ uniformly fellow travel, in the sense that the fellow traveling constant depends only on the Gromov-hyperbolicity constant. Making a choice for the geodesic, we call it the geodesic \emph{tracked} by the sample path.

\medskip 
The distance $d_X(x_0, w_n x_0)$ grows linearly with
high probability.  This was shown by \cite{bmss} for $\mu$ with finite
exponential moment, and by Gou\"ezel \cite{Gouezel}*{Theorem 1.1} for
general $\mu$.

\begin{lemma}\cite{bmss}\cite{Gouezel}\label{lemma:linear progress}
Suppose that $G$ is a countable group with a non-elementary action on a Gromov hyperbolic space $X$. Suppose that $\mu$ is a non-elementary
probability measure on $G$.  Then there are constants
$\ell > 0, K \ge 0$ and $0 < c < 1$ such that
\[\mathbb{P}( d_X(x_0, w_n x_0) \le \ell n ) \le K c^n. \]
Furthermore, if $\mu$ has finite exponential moment in $X$, then for
any $\epsilon > 0$ there are constants $\ell > 0, K \ge 0$ and $c
< 1$ such that
\[\mathbb{P}( | \tfrac{1}{n} d_X(x_0, w_n x_0) - \ell | \ge \epsilon )
\le K c^n. \]
\end{lemma}

The Gromov product of three points $a, b, c$ in a metric space $X$ is defined to be
\[
\GP{b}{c}_a = \tfrac{1}{2} ( d_X(a, b) + d_X(a, c) - d_X(b, c) )
\]
When $X$ is $\delta$-hyperbolic, the Gromov product $\GP{b}{c}_a$
equals the distance from $a$ to a geodesic from $b$ to $c$, up to
a bounded error depending only on $\delta$.  In this case, the product
can be extended to points which lie in the boundary, i.e. we may
choose $b, c \in \overline{X} = X \cup \partial X$.

\medskip The ``shadow'' will, roughly speaking, be the set of all points $c$ in $\overline{X}$ such that any geodesic from $a$ to $c$ passes close to $b$. Formally:

\begin{definition}\label{def:shadow}
Suppose that $a$ is a point in a Gromov hyperbolic space $X$ and $b$
is a point in $\overline{X}$.  Suppose that $r \ge 0$ is a constant.
The \emph{shadow} $\mho_a(b, r) \subseteq \overline{X}$ consists of
the closure of the set of all points $c$, such that
$\GP{b}{c}_a \ge r$.
\end{definition}

This differs from the usual definition of shadows, which is the
closure of the following set, where again $a\in X$ and $b\in\overline{X}$:
\[ \mho_a(b, r) = \{ c \in \overline{X} \mid \GP{b}{c}_a \ge d_X(a, b)
- r \}. \]
The two definitions are equivalent by setting $r = d_X(a, b) - R$ for
points $b \in X$, but our definition extends more conveniently to
points $b \in \partial X$.

\medskip
We will use the following Gromov product estimates from \cite{bmss}.
We state the version incorporating the linear progress results of
\cite{Gouezel}.  As shadows are defined in terms of the Gromov
product, we also state the results in terms of shadows.

\begin{proposition}\cite{bmss}*{Proposition 2.11}\cite{Gouezel}*{Theorem 1.1}
\label{prop:gp}\label{lemma:exponential decay}
Suppose that $G$ is a countable group with a non-elementary action on a geodesic Gromov-hyperbolic space $X$ with basepoint $x_0$. Suppose that $\mu$ is a non-elementary probability measure on $G$ with a finite exponential
moment. 
Then there are constants $K> 0$ and $c < 1$ such
that for all $0 \le i \le n$ and all $R > 0$ one has
\[ \PP \left( \GP{x_0}{w_n x_0}_{w_i x_0} \ge R \right) \le K c^{R}. \]
%
%
%
Furthermore, using the definition of shadows,
\[ \PP \left( \nu( \mho_{x_0}(w_n x_0, R) ) \right) \le K c^{R}. \]
\end{proposition}

Let $\gamma(t_n)$ be a closest point on $\gamma$ to $w_n x_0$.  Then
combining \Cref{lemma:linear progress} and \Cref{prop:gp} gives the
following linear progress result for the $t_n$.

The above results imply that the distance from a
location $w_n x_0$ to the tracked geodesic $\gamma$ is given by a probability measure
with exponential decay.

\begin{proposition}\label{prop:distance to geodesic}
Suppose that $G$ is a countable group with a non-elementary action on
a geodesic Gromov hyperbolic space $X$ with basepoint $x_0$. Suppose
that $\mu$ is a non-elementary probability measure on $G$ with finite exponential moment with respect
to $X$.
Then there
are constants $K \ge 0$ and $c < 1$, which do not depend on $n$, such that
$\mathbb{P}( d_X( w_n x_0 , \gamma) \ge r ) \le K c^{r}$. \qed
\end{proposition}

The Borel-Cantelli Lemma then gives the following corollary.

\begin{corollary}
\label{cor:deviation}
Suppose that $G$ is a countable group with a non-elementary action on a
geodesic Gromov hyperbolic space $X$ with basepoint $x_0$. Suppose
that $\mu$ is a non-elementary probability measure on $G$ with finite exponential moment with respect
to $X$. Then there is a constant $D> 0$, such that for almost all bi-infinite sample paths $w = (w_n)$, there is a tracked geodesic $\gamma = \gamma_w$, and an integer $N$
such that for all $n \ge N$
\[ d_X(w_n x_0 , \gamma ) \le D \log n. \]
\end{corollary}

\begin{proof}
By \Cref{prop:distance to geodesic}, for any $D > 0$ we have
\[
\PP (d_X(w_n x_0, \gamma) \ge D \log n) \le K \exp(D \log n \log c). 
\]
We choose $D>0$ such that $D  \log c < -1$. Set $\alpha = -D \log c$.
Then 
\[
\sum_n \PP (d_X(w_n x_0, \gamma) \ge D \log n) \le K \sum_n \frac{1}{n^\alpha} < \infty.
\]
Hence, by the Borel-Cantelli Lemma, we deduce that for almost every sample path, there is a natural number $N$ (that depends on the path) such that 
$d (w_n x_0 , \gamma) < D \log n$ for all $n \ge N$.
\end{proof}

\begin{corollary}\label{cor:tlinear}
Suppose that $G$ is a countable group with a non-elementary action on a geodesic Gromov hyperbolic space $X$ with basepoint $x_0$. Suppose that $\mu$ is a non-elementary probability measure on $G$.
Then there are constants
$\ell > 0, K \ge 0$ and $0 < c < 1$ such that
\[\mathbb{P}( t_n \le \ell n ) \le K c^n. \]
Furthermore, if $\mu$ has finite exponential moment with respect to $X$, then for
any $\epsilon > 0$ there are constants $\ell > 0, K \ge 0$ and $c
< 1$ such that
\[\mathbb{P}( | \tfrac{1}{n} t_n - \ell | \ge \epsilon )
\le K c^n. \]
\end{corollary}

We now use the Borel-Cantelli Lemma to show that the gap between
$t_n$ and $t_{n+1}$ is at most $\log n$.

\begin{proposition}\label{prop:segment bound}
Suppose that $G$ is a countable group with a non-elementary action on a geodesic Gromov hyperbolic space $X$ with basepoint $x_0$. Suppose that $\mu$ is a non-elementary probability measure on $G$ with finite exponential moment with respect to $X$. 
Then there is a constant $D> 0$, such that for almost all sample
paths, there is a tracked geodesic $\gamma$, and an integer $N$
such that for all $n \ge N$
\[ |t_{n+1} - t_n | \le D \log n. \]
\end{proposition}

\begin{proof}
By \Cref{prop:gp} there are constants $K_1 > 0$
and $c_1 < 1$ such that
$\PP( d_X(w_n x_0, w_{n+1} x_0) \ge R ) \le K_1 c_1^R$.  By
\Cref{prop:distance to geodesic}, there are constants $K_2> 0$ and $c_2< 1$ such that for all
$n$, we have that $\PP( d_X(w_n x_0, \gamma) \ge R ) \le K_2 c_2^R$.  By the
triangle inequality,
\[ |t_{n+1} - t_n| \le d_X( \gamma(t_n), w_n x_0) + d_X( w_n x_0,
w_{n+1} x_0) + d_X( w_{n+1} x_0, \gamma(t_{n+1}) ). \]
If $|t_{n+1} - t_n| \ge 3R$, then at least one of the terms on the right is at least $R$.  Therefore,
$\PP( |t_{n+1} - t_n| \ge 3R ) \le \max \{ K_1 c_1^R, K_2 c_2^R \}$.
The result then follows from the Borel-Cantelli Lemma applied the same way as in the proof of \Cref{cor:deviation}.
\end{proof}

Finally, we verify that the $(\nu \times \rnu)$-measure, as defined in \Cref{theorem:boundary_convergence}, of the
diagonal
$\Delta \subset \partial X \times \partial X$ is zero.  We may define a neighborhood of the diagonal as follows (using any basepoint $a\in X$).
\[  N_r(\Delta) = \bigcup_{b \in \partial X} \mho_a(b, r) \times
\mho_a(b, r)   \]

\begin{lemma}\cite{Maher}*{Proposition 4.7}\label{lemma:diagonal}
Suppose that $G$ is a countable group with a non-elementary action on
a geodesic Gromov hyperbolic space $X$ with basepoint $x_0$.  Suppose
that $\mu$ is a geometric probability measure on $G$.  Then there are
constants $K \ge 0$ and $c < 1$ such that
$\nu \times \rnu( N_r(\Delta) ) \le K c^{r}$, with $\nu \times \rnu$ again as in \Cref{theorem:boundary_convergence}.
\end{lemma}

\begin{proof}
By the action of $G$ on
$(\partial X \times \partial X, \nu \times \rnu)$, it suffices to show
this for $d_X( x_0, \gamma)$.  
Suppose that $d_X( x_0 , \gamma) \ge r$. 
As the Gromov product of three points $\GP{b}{c}_{a}$ is, up to an error of at most $2 \delta$, equal to the
distance from $a$ to a geodesic from $b$ to $c$, the endpoints of $\gamma$ are contained in the neighborhood
$N_{r+ 2 \delta}(\Delta)$. 
The probability that this occurs is at
most
$\nu \times \rnu( N_{r+ 2\delta}( \Delta) ) \le Kc^{r+ 2 \delta}$, as
required.
\end{proof}

\subsection{The pseudo-metric on the surface and the flat metric}\label{section:flat}

In this section we review some well known results that relate
hyperbolic and flat metrics on surfaces, see for example
\cite{kapovich}*{Chapter 11}.

\medskip
A flat structure $S_q$ on a closed surface $S$ is a Euclidean cone
metric with finitely many cone points, each of whose angles are
integer multiples of $\pi$.  Furthermore, for any sufficiently small
coordinate chart disjoint from the cone points, there is a preferred
choice of orthogonal directions, known as either the horizontal and
vertical directions, or alternatively the real and imaginary
directions.  The integral lines of the horizontal directions give a
(singular) foliation called the horizontal foliation.  Similarly, the
integral lines of the vertical directions give a (singular) foliation
called the vertical foliation.  Flat lengths of orthogonal arcs define
transverse measures for the foliations.

\medskip The Cannon-Thurston metric given by \Cref{eq:ct metric} is $\pi_1(S)$-invariant.  Hence,
the restriction of the infinitesimal pseudometric on
$\widetilde S_h \times \RR$ to $S_0$ gives rise to an infinitesimal
pseudometric on $S_h$,
which we have denoted by $d_{\widetilde S_h}$.  By identifying points on $S_h$ that are
pseudo-metric distance zero apart, i.e. points which lie in the same component of
$S_h \setminus (\Lambda_+ \cup \Lambda_-)$, we get a quotient of $S_h$ which is isometric
to a flat metric $S_q$ on $S$.  The quotient map sends the invariant
laminations $(\Lambda_-, dy)$ and $(\Lambda_+, dx)$ to the real and
imaginary measured foliations ($\mathcal{F}_r$ and $\mathcal{F}_i$
respectively) for the flat metric $S_q$.  As both (pseudo-)metrics are
defined on the closed surface $S$, they give quasi-isometric metrics on the universal
cover.  We record this statement as a proposition to fix notation.

\begin{proposition}\label{prop:qi}\label{c:Q_pa}\label{c:c_pa}
Suppose that $(S_h, \Lambda)$ is a hyperbolic metric on $S$ together
with a full pair of measured laminations.  Then the hyperbolic
metric $d_{\HH^2}$ and the Cannon-Thurston pseudometric
$d_{\widetilde S_h}$ on the universal cover $\widetilde S_h$ are
quasi-isometric, i.e. there are constants $Q_\LL \ge 1$ and
$c_\LL \ge 0$ such that for any points $p$ and $p'$ in the universal
cover,
\[ \frac{1}{Q_\LL} d_{\widetilde S_h}(p, p') - c_\LL \le d_{\HH^2}(p,
p') \le Q_\LL d_{\widetilde S_h}(p, p') + c_\LL. \]
\end{proposition}

By compactness of the circle direction or equivalently $\ZZ$-periodicity by the action of the pseudo-Anosov $\pa$, the constants in \Cref{prop:qi} remain uniform over any choice $S_z$ as fiber. 
More generally, if a doubly degenerate surface group in $\PSL(2,\mathbb{C})$ has bounded geometry then the constants in \Cref{prop:qi} will remain uniform for the quasi-isometry between the pseudo-metric and the flat metric for any $S_z$.

\section{Singularity of measures}\label{section:singularity}

In this section, we prove the singularity of measures, namely
\Cref{theorem:singularity}, by showing that typical geodesics have
different behavior for the pushforwards of the surface measures
(\Cref{theorem:surface measures}) than for the $3$-manifold measures
(\Cref{theorem:3-manifold measures}).  We show that for almost all
geodesics in $\widetilde S_h$ with respect to the surface measures,
the geodesics they determine in $\widetilde S_h \times \RR$ spend a
positive proportion of time close to the base fiber.  On the other
hand, for almost all geodesics in $\widetilde S_h \times \RR$ with
respect to the $3$-manifold measures, the proportion of time spent
close to the base fiber tends to zero.  We prove these facts for
Lebesgue measures in \Cref{section:surface lebesgue} and for
hitting measures for random walks in \Cref{section:surface hitting}.  We start by recording some useful facts about geodesics
which we will use in the subsequent sections.

\medskip
Suppose that $\gamma$ is an oriented geodesic in $\widetilde S_h$. 
We denote its pair of limit points in $\partial \widetilde S_h$ by
$\gamma_+$ and $\gamma_-$.

\begin{definition}\label{def:non-exceptional} 
We say that a bi-infinite geodesic $\gamma$ in $\widetilde S_h$ is
\emph{non-exceptional} if
\begin{itemize}
\item its limit points $\gamma_-$ and $\gamma_+$ are distinct from the
limit points of any boundary leaf of any ideal complementary region of
either of the invariant laminations, and
\item the limit points have distinct images under the Cannon-Thurston
map, that is $\iota(\gamma_-) \not = \iota(\gamma_+)$.
\end{itemize}
\end{definition}

As we see below, for surface measures, almost all
geodesics in $\widetilde S_h$ are non-exceptional. 

\begin{proposition}\label{prop:distinct}
Suppose that $\pa \colon S \to S$ is a pseudo-Anosov map and
$(S_h, \Lambda)$ is a hyperbolic metric on $S$ together with a pair of
invariant measured laminations.  Let $\nu$ be one of the surface
measures from \Cref{def:surface} or \Cref{def:surface general}.  Then
$\nu$-almost all geodesics in $\widetilde S_h$ are non-exceptional.
\end{proposition}

\begin{proof}
Let $\gamma$ be a geodesic in $\widetilde S_h$.  Suppose that its
image $\iota(\gamma)$ has a single limit point at infinity.  Then
$\gamma$ is either a leaf of an invariant lamination, or contained in
an ideal complementary region.  Being ideal polygons with finitely
many sides, there are only countably many ideal complementary regions.
So the collection of geodesics contained in ideal complementary
regions has measure zero as the measures are all non-atomic.

\medskip So we may consider leaves of invariant laminations.  For
Lebesgue measure, Birman and Series \cite{birman-series}*{Theorem II}
showed that the collection of endpoints of all simple geodesics has
measure zero in $\partial \ws_h \times \partial \ws_h$.  For hitting
measure, this result follows from double ergodicity of the action of
$\pi_1(S)$ on the boundary, due to Kaimanovich
\cite{KaimanovichGAFA}*{Theorem 17}.  For completeness, we give the
details for hitting measure below.

\medskip Let $\Delta$ denote the diagonal in
$\partial \ws_h \times \partial \ws_h$.  Suppose $\Lambda$ is a geodesic
lamination and
$\partial \Lambda \subset \partial \ws_h \times \partial \ws_h \setminus
\Delta$ be the endpoints of leaves of $\Lambda$.  The action of the
fundamental group $\pi_1(S)$ on
$\partial \ws_h \times \partial \ws_h \setminus \Delta$ is ergodic for the
hitting measure.  Since $\Lambda$ is $\pi_1(S)$-invariant, the set
$\partial \Lambda$ has measure $0$ or $1$.  Suppose $\ell$ is a leaf
of $\Lambda$.  Suppose that $\gamma$ is a geodesic that crosses
$\ell$.  Then $\gamma$ does not lie in $\Lambda$.  We can then choose
small neighborhoods $U_+$ and $U_-$ of $\gamma_+$ and $\gamma_-$ such
that any geodesic with one endpoint in $U_+$ and the other endpoint in
$U_-$ also crosses $\ell$.  Thus, such a geodesic does not lie in
$\Lambda$.  As open sets have positive measure, $\partial \Lambda$ has
measure strictly less than $1$.  Hence, $\partial \Lambda$ has measure
zero as required.
\end{proof}

\subsection{Quasigeodesics from loxodromics}\label{section:qg lox}

Both $S$ and $M$ are compact, and their fundamental groups are torsion
free.  The hyperbolic metrics on $S$ and $M$ give maps
$\rho_S \colon \pi_1(S) \to \text{Isom}(\HH^2)$ and
$\rho_M \colon \pi_1(M) \to \text{Isom}(\HH^3)$, whose images are
torsion-free cocompact lattices.  In particular, in both cases, every
non-trivial element of the fundamental group maps to a loxodromic
isometry.  We now show that the image of an axis for a non-trivial
element of $\pi_1(S)$ is a quasigeodesic in
$\widetilde{S}_h \times \RR$.

\begin{proposition}\label{prop:alphaQG}
Suppose that $\pa \colon S \to S$ is a pseudo-Anosov map, and
$(S_h, \Lambda)$ is a hyperbolic metric on $S$ together with a pair of
invariant measured laminations, and let $g$ be a non-trivial element
of $\pi_1(S)$ with axis $\alpha$ in $\widetilde S_h$.  Then there are
constants $Q_\alpha \ge 1, c_\alpha$ and $K_\alpha$ such that the
image of the axis $\iota(\alpha)$ in $\widetilde{S}_h \times \RR$ is
$(Q_\alpha, c_\alpha)$-quasigeodesic.  In particular, the geodesic
$\overline{\alpha}$ connecting the endpoints of $\iota(\alpha)$ is
contained in a $K_\alpha$-neighborhood of $S_0$.
\end{proposition}

\begin{proof}
The isometry
$\rho_S(g) \in \text{Isom}(\HH^2)$ is loxodromic.  Let $\alpha$ be the
axis of $\rho_S(g)$ in $\HH^2$.  By abuse of notation, we will also
write $g$ for the image of $g$ in $\pi_1(M)$ under the (injective)
inclusion map $i \colon \pi_1(S) \to \pi_1(M)$.  Then
$\rho_M(g) \in \text{Isom}(\HH^3)$ is also loxodromic.  As
$\widetilde S_h \times \RR$ is quasi-isometric to $\HH^3$, $g$ also
acts loxodromically on $\widetilde S_h \times \RR$.  Recall that
$\iota \colon \widetilde S_h \to \widetilde S_h \times \RR$ is the
inclusion map.  The image $\iota(\alpha)$ is an $\rho_M(g)$-invariant
path in $\widetilde S_h \times \RR$.  Hence, for constants
$(Q_\alpha, c_\alpha)$ that depend on $\alpha$, the path
$\iota(\alpha)$ is a $(Q_\alpha, c_\alpha)$-quasigeodesic.  We shall
write $\overline{\alpha}$ for the geodesic in
$\widetilde S_h \times \RR$ connecting the endpoints of
$\iota(\alpha)$.  It follows that $\overline{\alpha}$ is the axis for
$g$ acting on $\widetilde S_h \times \RR$.  By the Morse lemma,
\Cref{lemma:morse}, there is a constant $K_\alpha > 0$ such that
$\iota(\alpha)$ is contained in an $K_\alpha$-neighborhood of
$\overline{\alpha}$.  As $\iota(\alpha)$ is contained in $S_0$, the
proposition follows.
\end{proof}

\medskip

We next show that there is an upper bound $P_\alpha$ on the length of the nearest
point projection interval $p_\alpha(\ell)$, for any leaf $\ell$ of an
invariant lamination.

\begin{proposition}\label{prop:Palpha}
Suppose that $\pa \colon S \to S$ is a pseudo-Anosov map,
and $(S_h, \Lambda)$ is a hyperbolic metric on $S$ together with a pair of
invariant measured laminations. Let $g$ be a non-trivial element
of $\pi_1(S)$ with axis $\alpha$ in $\widetilde S_h$.  Then there is a
constant $P_\alpha > 0$ such that for any leaf $\ell$ of either of the
invariant laminations, the nearest point projection of $\ell$ to $\alpha$ and $\alpha$ to
$\ell$ has length at most $P_\alpha$.
\end{proposition}

\begin{proof}
Every closed geodesic in $S_0$ intersects both invariant laminations.
Hence, any segment of $\alpha$ with length $L_\alpha$ intersects both
laminations $\Lambda_+$ and $\Lambda_-$.  By compactness, there is
also a minimum angle $\theta_\alpha > 0$ for an intersection of
$\alpha$ with any leaf of an invariant lamination.  The above two
properties together with \Cref{prop:projection interval} imply that there is an upper bound $P_\alpha$ for the
length of the nearest point projection interval $p_\alpha(\ell)$, for
any leaf $\ell$ of an invariant lamination, and similarly for $p_\ell (\alpha)$.
\end{proof}

\medskip

The axis $\alpha$ projects to a closed geodesic in $S_0$.  Let
$L_\alpha$ be the length of this geodesic.  Thus, $L_\alpha$ is the
translation length of $g$ on $\widetilde S_h$. Note that we may replace $P_\alpha$ by a larger
constant to assume that $P_\alpha \ge L_\alpha$, and it will be
convenient to do so.
We now show that the limit points of $\iota(\alpha)$ are disjoint from
the limit sets of any ladder $F(\ell)$ of any leaf $\ell$ of an
invariant lamination.

\begin{proposition}\label{prop:limit disjoint from ladder}
Suppose that $\pa \colon S \to S$ is a pseudo-Anosov map and
$(S_h, \Lambda)$ is a hyperbolic metric on $S$ together with a pair of
invariant measured laminations.  Suppose that $g$ is a non-trivial
element of $\pi_1(S)$.  For any leaf $\ell$ of an invariant
lamination, the fixed points $\overline{\alpha}_+$ and
$\overline{\alpha}_-$ of $\rho_M(g)$ are disjoint from the limit
points of $F(\ell)$ in $\partial ( \widetilde S_h \times \RR )$.
\end{proposition}

\begin{proof}
For an ideal complementary region $C$ of an invariant lamination, we denote by $s(C)$ the number of sides of $C$.
Since there are finitely many ideal complementary regions, the numbers $s(C)$ over all $C$ has a maximum which we denote by $s_{\textrm{max}}$.

\medskip
Let $\ell$ be a leaf of an invariant lamination. Breaking symmetry, suppose that $\ell$ lies in $\Lambda_+$.  
By \Cref{prop:Palpha}, the projection interval $p_\alpha(\ell)$ has length at most $P_\alpha$.  
As $\rho_S(g)$ acts on $\alpha$ by translation, we may choose $n$ large enough such that the projection interval $p_\alpha(\rho_S(g)^n \ell)$ is distance greater than $s_{\textrm{max}} P_\alpha$ from the projection interval $p_\alpha(\ell)$.

\medskip
Suppose that $\ell$ and $\rho_S(g)^n \ell$ intersect a common ideal complementary region of the other lamination $\Lambda_-$. Traversing in a cyclic order, we may choose a cyclic sequence of boundary geodesics $\ell'_1, \cdots, \ell'_j$ such that $\ell'_1$ intersects $\ell$ and $\ell'_j$ intersects $\rho_S(g)^n \ell$. But then the projection intervals $p_\alpha(\ell)$ and $p_\alpha(\rho_S(g)^n \ell)$ are at most $j P_\alpha < s_{\textrm{max}} P_\alpha$ distance apart, a contradiction. We deduce that $\ell$ and $\rho_S(g)^n \ell$ do not intersect a common ideal complementary region of $\Lambda_-$.

\medskip
On the other hand, since $P_\alpha > L_\alpha$, there is a leaf
of $\Lambda_+$ separating $\ell$ and
$\rho_S(g)^n \ell$.   

\medskip
By \Cref{prop:common intersection leaf} and \Cref{prop:separated ladders}, the distance between $F(\ell)$ and
$F(\rho_S(g)^n \ell)$ is at least $\epsilon > 0$. 
By \Cref{prop:disjoint limits}, their limit
sets are disjoint.  
The ladder $F( \rho_S(g)^n \ell)$ equals $\rho_M(g)^n F(\ell)$ and thus the limit set of $F(\ell)$ cannot contain any fixed points of $\rho_M(g)$, as required.
\end{proof}

In a similar vein, we now show that there is an upper
bound $P_{\overline{\alpha}}$ on the nearest point projection in
$\widetilde S_h \times \RR$ of the axis $\overline{\alpha}$ to any
ladder $F(\ell)$ over any leaf $\ell$ of an invariant lamination.

\begin{corollary}\label{cor:ladder projection}
Suppose that $\pa \colon S \to S$ is a pseudo-Anosov map and
$(S_h, \Lambda)$ is a hyperbolic metric on $S$ together with a pair of
invariant measured laminations.  Suppose that $g$ is a non-trivial
element of $\pi_1(S)$ with axis $\overline{\alpha}$ in
$\widetilde S_h \times \RR$.  Then there is a constant
$P_{\overline{\alpha}} >0$ such that for any leaf $\ell$ of an
invariant lamination, the diameter of the projection image
$p_{\overline{\alpha}}(F(\ell))$ is at most $P_{\overline{\alpha}}$.
\end{corollary}

\begin{proof}
Suppose that there is a sequence of leaves $\ell_n$ of the
invariant laminations such that the diameters of the projections $p_{\overline{\alpha}} (F(\ell_n))$
tend to infinity as $n \to \infty$. 
By \Cref{prop:alphaQG}, the image $\iota(\alpha)$ is contained in a bounded neighborhood of
$\overline{\alpha}$. 
Hence, the diameters of the images of the ladders $F(\ell_n)$ under the nearest point projection to $\iota(\alpha)$, also tend to infinity.  As the ladders $F(\ell_n)$ are quasiconvex,
the diameters of the subsets of $\iota(\alpha)$ contained in a bounded
neighborhood of $F(\ell_n)$ tends to infinity. Hence, the
diameters of the nearest point projections of $\ell_n \times \{ 0 \}$
to $\iota(\alpha)$ tends to infinity,  contradicting \Cref{prop:Palpha} which states that the diameters are bounded above by $P_\alpha$.
\end{proof}

We now show that if two points $x$ and $y$ in $\widetilde S_h$ have
nearest point projections to $\alpha$ which are far apart, then their
images $\iota(x)$ and $\iota(y)$ in $\wsr$ have nearest point
projections to $\overline{\alpha}$ which are far apart.  As the
inclusion map distorts distances, \emph{a priori}, the inclusion of
the nearest point projection of $x$ to $\alpha$ need not be close to
the nearest point projection of $\iota(x)$ to $\overline{\alpha}$.


\begin{proposition}\label{prop:axis projection}
Suppose that $\pa \colon S \to S$ is a pseudo-Anosov map and
$(S_h, \Lambda)$ is a hyperbolic metric on $S$ together with a pair of
invariant measured laminations.  Suppose that $g$ is a non-trivial
element of $\pi_1(S)$, with axis $\alpha$ in $\widetilde S_h$, and
axis $\overline{\alpha}$ in $\widetilde S_h \times \RR$.  Then there
are constants $Q \ge 1$ and $c \ge 0$ such that for any two points $x$
and $y$ in $\ws_h \cup \partial \ws_h$,
\[ 
d_{\wsr} \left( p_\oalpha(\iota(x)), p_\oalpha(\iota(y)) \right) \ge
\frac{1}{Q} d_{\ws_h} ( p_\alpha(x), p_\alpha(y) ) - c.
\]
\end{proposition}

\begin{proof}
Let $I = [p, q]$ be the subinterval of $\alpha$ with endpoints
$p := p_\alpha(x)$ and $q := p_\alpha(y)$.  We may assume that
$d_{\ws_h}(p, q) \ge 2L_1$, where $L_1 = L_\alpha + 2 P_\alpha$, where
$L_\alpha$ is the translation length of $g$, and $P_\alpha$ is the
largest size of the projection of any leaf of an invariant lamination
to $\alpha$, from \Cref{prop:Palpha}.  We shall choose $Q = Q_\alpha$
and
$c = 2 L_1 / Q_\alpha + c_\alpha + 2 K_\alpha + 2
P_{\overline{\alpha}}$, where $Q_\alpha$ and $c_\alpha$ are the
quasigeodesic constants for $\iota(\alpha)$ in
$\widetilde{S}_h \times \RR$ from \Cref{prop:alphaQG}, and $K_\alpha$
is the corresponding Morse constant.

\medskip Let $\beta$ be a segment of $I$ of length $L_1$.  The central
segment of $\beta$ of length $L_\alpha$ intersects leaves of both laminations.  Suppose that $\ell$ is such a leaf.  By our choice
of $L_1$, the projection interval $p_\alpha(\ell)$ is contained in the
interior of $\beta$, which in turn is contained in $I$. Let $\gamma$ be the geodesic spanned by $x$ and $y$. By
\Cref{prop:nested implies intersect}, as
$p_\alpha(\ell) \subset p_\alpha(\gamma)$, and as $\alpha$ and $\ell$
intersect, $\ell$ and $\gamma$ also intersect.

\medskip Let $\beta_1$ be an initial segment of $I$ of length $L_1$, and let
$\beta_2$ be a terminal segment of $I$ of length $L_1$. Since $I|I| \geqslant 2L_1$, the segments $\beta_1$ and $\beta_2$ are disjoint. By the conclusions of the previous paragraph, there exists leaves $\ell_1$ and $\ell_2$ of an
invariant lamination such that
\begin{itemize}
    \item $\ell_1$ intersects $\beta_1$ and $\gamma$ and $p_\alpha(\ell_1)$ lies in the interior of $\beta_1$, and
    \item $\ell_2$ intersects $\beta_2$ and $\gamma$ and $p_\alpha(\ell_2)$ lies in the interior of $\beta_2$.
\end{itemize}
It follows that $\ell_1$ and $\ell_2$ are disjoint and divide
$\widetilde S_h$ into three regions $A_1$, $A_2$ and $A_3$ such that
\begin{itemize}
    \item $A_1$ is adjacent to $\ell_1$ and contains $x$,
    \item $A_2$ lies between $\ell_1$ and $\ell_2$, and
    \item $A_3$ is adjacent to $\ell_2$ and contains $y$.
\end{itemize}
The leaves $\ell_1$ and $\ell_2$ divide $\alpha$ into arcs $\alpha_1$,
$\alpha_2$, and $\alpha_3$ such that $\alpha_i$ is contained in $A_i$.
Furthermore,
$p_\alpha(x)$ is contained in $\alpha_1$, and
$p_\alpha(y)$ is contained in $\alpha_3$.
The arc $I \setminus (\beta_1 \cup \beta_2)$ is contained in
$\alpha_2$ and its length is $d_{\ws_h}(p, q) - 2 L_1$.  It follows
that the length of $\alpha_2$ is at least $d_{\ws_h}(p, q) - 2 L_1$.

\begin{figure}[h]
\begin{center}
\begin{tikzpicture}[scale=0.85]

\tikzstyle{point}=[circle, draw, fill=black, inner sep=1pt]

\draw (-3, 0) -- (3, 0) node [label=right:$\alpha$] {};

\draw [thick, arrows=-|] (-2, 0) node [point, label=below:$p$] {} --
              ( -0.5, 0) node [label=below:$\beta_1$] {};

\draw [thick, arrows=|-] (0.5, 0)  node [label=below:$\beta_2$] {} --
              (2, 0) node [point, label=below:$q$] {};

\draw (-2, -2) node [point, label=left:$x$] {};

\draw (2, -2) node [point, label=right:$y$] {};

\draw (-1, -3) node [label=below:$\ell_1$] {} -- (-1, 1);

\draw (1, -3) node [label=below:$\ell_2$] {} -- (1, 1);

\draw (-1.5, 0.5) node {$A_1$};
\draw (0, 0.5) node {$A_2$};
\draw (1.5, 0.5) node {$A_3$};

\draw (0, -4) node {$\widetilde S_h$};

\begin{scope}[xshift=8cm]

\draw (-3, 0) -- (3, 0)
node [label=right:$\overline{\alpha}$] {}
node [pos=0.2, label=below:$\overline{\alpha}_1$] {}
node [midway, label=below:$\overline{\alpha}_2$] {}
node [pos=0.8, label=below:$\overline{\alpha}_3$] {};

\draw (-2, -2) node [point, label=left:$\iota(x)$] {};
\draw (2, -2) node [point, label=right:$\iota(y)$] {};

\draw (-1, -3) node [label=below:$F(\ell_1)$] {} -- (-1, 1);

\draw (1, -3) node [label=below:$F(\ell_2)$] {} -- (1, 1);

\draw (-1.5, 0.5) node {$B_1$};
\draw (0, 0.5) node {$B_2$};
\draw (1.5, 0.5) node {$B_3$};

\draw (0, -4) node {$\widetilde S_h \times \RR$};

\end{scope}

\end{tikzpicture}
\end{center}
\caption{Notation for the complements of the leaves $\ell_i$ and
  ladders $F(\ell_i)$.} \label{fig:division}
\end{figure}
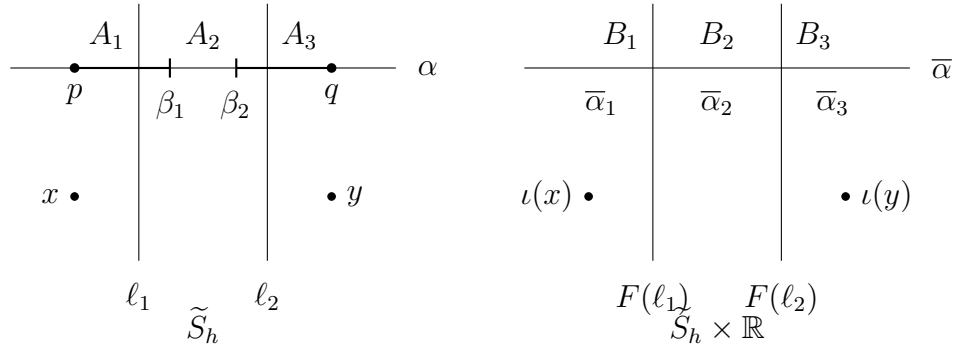

\medskip

The ladders $F(\ell_1)$ and $F(\ell_2)$ are convex in
$\widetilde S_h \times \RR$.  Since $\ell_1$ and $\ell_2$ are
disjoint, $F(\ell_1)$ and $F(\ell_2)$ are also disjoint.  Thus they
similarly divide $\widetilde S_h \times \RR$ into three regions $B_1$,
$B_2$ and $B_3$ where
\begin{itemize}
    \item $B_1$ is adjacent to $F(\ell_1)$,
    \item $B_2$ is between $F(\ell_1)$ and $F(\ell_2)$, and
    \item $B_3$ is adjacent to $F(\ell_2)$.
\end{itemize}
Note that $\iota(A_i) \subseteq B_i$.  This implies that the initial
limit point of $\overline{\alpha}$ is contained in the limit set of
$B_1$, and the terminal limit point of $\overline{\alpha}$ is
contained in the limit set of $B_3$.  By \Cref{prop:limit disjoint
  from ladder}, the limit points of $\overline{\alpha}$ are not
contained in the limit points of either $F(\ell_1)$ or $F(\ell_2)$, so
the ladders $F(\ell_i)$ also divide $\overline{\alpha}$ into three
components, $\overline{\alpha}_1$, $\overline{\alpha}_2$ and
$\overline{\alpha}_3$ so that $\overline{\alpha}_i \subset B_i$.
Since $\iota(x)$ lies in $B_1$ its nearest point projection to
$\overline{\alpha}$ is contained in the nearest point projection of
$B_1$ to $\overline{\alpha}$.  Since the boundary of $B_1$ is the
ladder $F(\ell_1)$, which is convex, we deduce that
\[
p_{\overline{\alpha}}(\iota(x)) \subseteq \overline{\alpha}_1
\cup p_{\overline{\alpha}}(F(\ell_1)).
\]
In particular, the projection of $\iota(x)$ to
$\overline{\alpha}$ is contained in a
$P_{\overline{\alpha}}$-neighborhood of $\alpha_1$.  Similarly, the
projection of $\iota(y)$ to $\overline{\alpha}$ is contained
in $\overline{\alpha}_3 \cup p_{\overline{\alpha}}(F(\ell_2))$, and so
is contained in a $P_{\overline{\alpha}}$-neighborhood of
$\overline{\alpha}_3$.

\medskip

The intersection points of $\ell_1$ and $\ell_2$ with $\alpha$ are
endpoints of $\alpha_2$, whose length is at least
$| I | - 2 L_1$.  Since $\iota(\alpha)$ is a
$(Q_\alpha, c_\alpha)$-quasigeodesic in $\widetilde{S_h} \times \RR$,
the distance between the intersection points of $F(\ell_1)$ and
$F(\ell_2)$ with $\iota(\alpha)$ is at least
\[
\frac{1}{Q_\alpha}( | I | - 2 L_1) - c_\alpha.
\]
Being a quasigeodesic, $\iota(\alpha)$ is contained in an
$K_\alpha$-neighborhood of $\overline{\alpha}$.  Therefore, the
distance between the nearest point projections of $F(\ell_1)$ and
$F(\ell_2)$ to $\overline{\alpha}$ is at least
\[
\frac{1}{Q_\alpha}( \diam( | I | - 2 L_1) - c_\alpha - 2
K_\alpha.
\]
By \Cref{cor:ladder projection}, the diameter of
$p_{\overline{\alpha}}(F(\ell_1))$, and similarly
$p_{\overline{\alpha}}(F(\ell_2))$, is at most
$P_{\overline{\alpha}}$.  It follows that the distance in
$\widetilde{S_h} \times \RR$ between
$p_{\overline{\alpha}}(F(\ell_1))$ and
$p_{\overline{\alpha}}(F(\ell_2))$ is at least
\[
\frac{1}{Q_\alpha}( | I | - 2 L_1) - c_\alpha - 2
K_\alpha - 2 P_{\overline{\alpha}} = \frac{1}{Q} d_{\ws_h}(p, q)  - c,
\]
where $Q = Q_\alpha$ and
$c = 2L_1 / Q_\alpha + c_\alpha + 2 K_\alpha + 2
P_{\overline{\alpha}}$.  Therefore, the distance between the
projections of $\iota(x)$ and $\iota(y)$ to
$\overline{\alpha}$ is at least
$\tfrac{1}{Q} \diam( p_\alpha(\gamma) ) - c$, where the constants $Q$
and $c$ depend only on $\alpha$ and not on $x$ and $y$, as required.
\end{proof}

We will also use the well known fact that for a discrete group of
isometries of $\HH^n$, for any loxodromic element $g$ with axis
$\alpha$, the size of the projections of all of the translates of the
axis $\alpha$ to $\alpha$ is bounded, see for example
\cite{bbf}*{Example 2.1(1)}.

\begin{proposition}
Suppose that $G$ is a countable group acting locally finitely by
loxodromics on $\HH^n$.  Suppose that $g$ is a loxodromic element with
geodesic axis $\alpha$.  Then there is a constant $L$ such that for
any distinct translate $h \alpha \not = \alpha$ of the axis, the size
of the projection interval $p_\alpha(h \alpha)$ is at most $L$.
\end{proposition}

\subsection{Lebesgue measure}\label{section:surface lebesgue}

\medskip

In this section we will consider geodesics chosen using Lebesgue
measure on either $\partial \ws_h$ or $\partial( \wsr )$.  We start
with the surface case, and show that for the pushforward of Lebesgue
measure on $\partial \widetilde S_h$, almost all geodesics in
$\widetilde S_h \times \RR$ spend a positive proportion of time close
to the base fiber $S_0$.  This shows the Lebesgue measure case of
\Cref{theorem:surface measures}.

\medskip

\begin{lemma}
Suppose that $\pa \colon S \to S$ is a pseudo-Anosov map and
$(S_h, \Lambda)$ is a hyperbolic metric on $S$ together with a pair of
invariant measured laminations.  Suppose that $\nu$ is the pushforward
of Lebesgue measure on $\partial \widetilde S_h$ under the
Cannon-Thurston map.  Then there are constants $R \ge 0$ and
$\epsilon > 0$ such that for $(\nu \times \nu)$-almost all geodesics
$\overline{\gamma}$ in $\widetilde S_h \times \RR$, for any unit speed
parametrization $\overline{\gamma}(t)$,
\[ \lim_{T \to \infty} \frac{1}{T} | \{ t \in [0, T] \mid \overline{\gamma}(t) \in
N_R(S_0) \} | \ge \epsilon. \]
\end{lemma}

We give a brief overview of the argument.  Let $\alpha$ be a geodesic
in $\ws_h$ which covers a closed geodesic $\beta$ in $S_h$.  By the
results of the previous section, if a geodesic $\gamma$ in $\widetilde S_h$ has a large fellow travel
with $\alpha$, then $\overline{\gamma}$ has a large fellow travel with
$\overline{\alpha}$.  By ergodicity, a typical geodesic $\gamma$ in
$S_h$ spends a positive proportion of time
fellow-traveling any closed geodesic $\beta$ in $S_h$.  This
implies that $\overline{\gamma}$ spends a positive proportion of time
close to $S_0$.  We now give a precise version of this argument.

\begin{proof}
Let $\alpha$ be a geodesic in $\widetilde S_h$ which projects to a
closed geodesic $\beta$ in $S_h$.  By abuse of notation, we denote the
lift of $\beta$ to $T^1(S_h)$ also by $\beta$.  Let $\delta_1$ be a
sufficiently small constant such that $N_{\delta_1}(\beta)$ is a
regular neighborhood of $\beta$ in $T^1(S_h)$.  In particular, this
implies that $\delta_1$ is less than the injectivity radius of
$T^1(S_h)$.  We fix a constant $L > 0$ such that
\begin{equation}\label{eq:Lbig}
L \ge  Q c + 11 \delta_3 + 1,
\end{equation}
where $Q$ and $c$ are
constants in \Cref{prop:axis projection}, and $\delta_3$ is the
constant of hyperbolicity for $\widetilde{S}_h \times \RR$.

\medskip
Let $\eta$ be an oriented geodesic in $\widetilde{S}_h$ which
intersects $N_{\delta_1}(\alpha)$ in a segment of length $L$, and let
$v$ be the tangent vector to $\eta$ at the first point of intersection
between $\eta$ and $N_{\delta_1}(\alpha)$.  By abuse of notation, we
shall also write $v$ for its image in the quotient $T^1(S_h)$.

\medskip
Let $A(r)$ be the disc of radius $r$ centered at $v$, perpendicular to
the lift of $\eta$ in $T^1(S_h)$.  For any constant $\epsilon > 0$
there is a constant $\delta$ such that the forward image of any point
$w$ in $A(r)$ under the geodesic flow intersects $N_{\delta_1}(\beta)$
in a segment of length $L$, up to additive error at most $\epsilon$,
and furthermore $w$ intersects $N_{\delta_1}(\beta)$ within distance
$\epsilon$ under the geodesic flow.  Let $\delta_2$ be the value of
$\delta$ corresponding to
$\epsilon_2 = \min \{ \delta_1, \delta_3, L/2\}$.  We can replace
$\delta_2$ by any smaller positive number, so in particular we may
assume $\delta_2 \le \delta_3$.

\medskip
Now choose $V$ to be the image of $A(\delta_2)$ under the forward and
backward geodesic flows of distance $\tfrac{1}{2} \delta_2$.  The
interior of $V$ is an open regular neighborhood of $v$.  Every
geodesic flow line that intersects $V$ intersects it in a segment of
length $\delta_2$, and intersects $N_{\delta}(\beta)$ in a segment of
length at least $L - \epsilon_2$, and at most $L + \epsilon_2$,
starting within distance $\epsilon_2$ of the intersection of the flow
line with $V$.  As $N_{\delta_2}(\beta)$ is a regular neighborhood of
$\beta$, for any flow line $\gamma$, segments of
$\gamma \cap N_{\delta_2}(\beta)$ corresponding to distinct
intersections with $V$ do not intersect.

\medskip
As the geodesic flow is ergodic, almost all geodesic flow lines spend
a positive proportion of time in $V$, i.e.
\[ \lim_{T \to \infty} \frac{1}{T} |\gamma(0, T) \cap V| =
\mathbf{vol}_{T^1(S_h)}(V).  \]
As geodesic flow lines intersect $V$ in segments of length $\delta_2$,
the number of times a geodesic flow line intersects $V$ is then
\[ \lim_{T \to \infty} \frac{1}{T} ( \# \text{segments of }\gamma(0,
T) \cap V ) = \frac{1}{\delta_2} \mathbf{vol}_{T^1(S_h)}(V).  \]
Each segment of $\gamma \cap V$ is within distance at most
$\epsilon_2$ of a segment of $\gamma \cap N_{\delta_1}(\beta)$ of
length at least $L - \epsilon_2$.  Therefore, for any interval
$\gamma(I)$ containing the segment, the nearest point projection of
$\gamma(I)$ to $\alpha$ has length at least
$L - \epsilon_2 - 2 \delta_1$.  By \Cref{prop:axis projection}, the
nearest point projection of $\overline{\gamma}(\overline{I})$ to
$\overline{\alpha}$ is therefore at least
$\epsilon_3 = \tfrac{1}{Q}(L - \epsilon_2 - 2 \delta_1) - c$.  As both
$\epsilon_2$ and $\delta_1$ are at most $\delta_3$, our choice of $L$
from \eqref{eq:Lbig} implies $\epsilon_3 \geqslant 8 \delta_3 + 1 > 0$.

\medskip
By \Cref{prop:neighbourhood}, there is a subset of $\overline{\gamma}$
of diameter $\tfrac{1}{Q} L - c $ contained in an
$6 \delta_3$-neighborhood of $\overline{\alpha}$.  By
\Cref{prop:alphaQG}, $\overline{\alpha}$ is contained in a
$K_\alpha$-neighborhood of $S_0$.  The nearest point projection map
from $\iota(\gamma)$ to $\overline{\gamma}$ is distance decreasing, so
$\overline{\gamma}$ spends a positive proportion of its length within
distance $K_\alpha + 6 \delta_3$ of $S_0$, as required.
\end{proof}

We show that for almost all geodesics chosen according to Lebesgue
measure on $\partial(\widetilde S_h \times \RR)$, the proportion of
their length which lies close to $S_0$ tends to zero.  This gives the
Lebesgue measure case of \Cref{theorem:3-manifold measures}.

\begin{lemma}\label{lemma:3-manifold lebesgue}
Suppose that $\pa \colon S \to S$ is a pseudo-Anosov map and
$(S_h, \Lambda)$ is a hyperbolic metric on $S$ together with a pair of
invariant measured laminations.  Suppose that $\nu$ is the Lebesgue
measure on $\partial(\widetilde S_h \times \RR)$.  Then for any
constant $R > 0$, for $(\nu \times \nu)$-almost all geodesics
$\ogamma$ in $\widetilde S_h \times \RR$, for any unit speed
parametrization $\ogamma(t)$,
\[ \lim_{T \to \infty} \frac{1}{T} | \left\{ t \in [0, T] \mid
\ogamma(t) \in N_R(S_0) \right\} | = 0.  \]
\end{lemma}

We deduce \Cref{lemma:3-manifold lebesgue} directly from the work of Oh and Pan \cite{oh-pan},
which we now describe.  We state a special case of their main result
which will suffice here.  The fibering $M \to S^1$
induces a homomorphism $\pi_1(M) \to \ZZ$, whose kernel is isomorphic
to $\pi_1(S)$.  Corresponding to the kernel, we get the $\ZZ$-cover
$M_\ZZ$ homeomorphic to $S \times \RR$.  The hyperbolic metric on $M$
lifts to a $\ZZ$-periodic hyperbolic metric on $M_\ZZ$.

\medskip The hyperbolic metric on $M$ gives a cocompact lattice
$\Gamma_0$ in $G = \text{PSL}(2, \mathbb{C})$, so that
$\Gamma_0 \backslash G$ is the frame bundle for $M$.  The frame flow
is denoted by right multiplication by
$a_t = \begin{bmatrix} e^{t/2} & 0 \\ 0 & e^{-t/2} \end{bmatrix}$,
which projects to the geodesic flow in $\HH^3$.  Haar measure is a
left-invariant measure on $G$, and determines a measure on
$\Gamma_0 \backslash G$.  The frame flow is ergodic with respect to
Haar measure, which projects to the Liouville measure on $T^1(M)$.  Oh
and Pan show the following mixing result for the geodesic flow.

\begin{theorem}\cite{oh-pan}*{Theorem 1.7.}
\label{thm:oh-pan}
Let $\Gamma_0$ be a cocompact lattice in
$G = \text{PSL}(2, \mathbb{C})$, and let $\Gamma \backslash G$ be a
$\ZZ$-cover of $\Gamma_0 \backslash G$.  Let $\psi_1$ and $\psi_2$ be
continuous functions on $\Gamma \backslash G$ with compact support.
Then
\begin{equation}\label{eq:oh-pan}
\lim_{t \to + \infty} t^{1/2} \int_{\Gamma \backslash G}
\psi_1(x a_t) \psi_2(x) \ dx = \frac{1}{(2 \pi \sigma)^{1/2}}
\int_{\Gamma \backslash G} \psi_1 \ dx \int_{\Gamma \backslash G}
\psi_2 \ dx,
\end{equation}
where $\sigma$ is a constant depending on $\Gamma_0$.
\end{theorem}

In particular, the above result implies that the proportion of
geodesics starting at height zero (in $M_\ZZ$) that are close to
height zero at time $t$ decays at rate $1/\sqrt{t}$.  This implies
that a typical Lebesgue geodesic in $M_\ZZ$ is recurrent on the
fibers, but the proportion of time spent near a fixed fiber decays
like $1 / \sqrt{t}$.

\begin{proof}[Proof of \Cref{lemma:3-manifold lebesgue}]
Any choice of homeomorphism $g$ from the mapping torus $M$ to the
hyperbolic manifold $\HH^3 / \Gamma$ lifts to a quasi-isometry
$\widetilde{g}$ from $\wsr$ with the Cannon-Thurston metric to
$\HH^3$ with the standard metric.  We shall write $d_{\HH^3}$ for the
pullback of the hyperbolic metric to $\wsr$ by $\widetilde{g}$.

\medskip
Suppose that $p$ is a basepoint at height zero in
$\widetilde S_h \times \RR$.  Let $\phi$ be a rotationally symmetric
continuous approximation to the indicator function of $p$, normalized
so that it integrates to one, with respect to the hyperbolic metric
$d_{\HH^3}$.  Let $\phi'$ be the pull back of $\phi$ to
$G = \PSL(2, \mathbb{C})$.  Let $\phi'_t$ be the composition of
$\phi'$ with the frame flow, that is $\phi_t'(x) = \phi'(x a_t)$.
Then the forward and backward projections of $\phi_t'$ to
$\widetilde S_h \times \RR$ converge to Lebesgue measure on the sphere
at infinity $\partial ( \wsr )$.

\medskip In applying \Cref{thm:oh-pan}, we set $\psi_1$ to be the push
forward of $\phi'$ to $\Gamma \backslash G$.  We set $\psi_2$ to be a
close approximation of the indicator function for the pre-image of
height-zero fiber in $\Gamma \backslash G$.  With this choice, the
left hand side integral in \eqref{eq:oh-pan} gives the proportion of
geodesics which at time $t$ lie close to the base fiber in $M_\ZZ$.
By \eqref{eq:oh-pan}, this proportion goes to zero at rate
$1 / \sqrt{t}$.  As the hyperbolic metric $d_{\HH^3}$ is
quasi-isometric to the Cannon-Thurston metric $d_{\wsr}$, this also
holds for the Cannon-Thurston metric.
\end{proof}

\subsection{Hitting measure}\label{section:surface hitting}

In this section we consider geodesics chosen according to hitting
measure determined by random walks.

\medskip We start by showing that for the pushforward of a hitting measure on
$\partial \widetilde S_h$ arising from a random walk, almost all
geodesics in $\widetilde S_h \times \RR$ spend a positive proportion
of time close to the base fiber $S_0$.  This completes the proof of
\Cref{theorem:surface measures} by showing the hitting measure case
for surface measures.

\begin{lemma}
Suppose that $\pa \colon S \to S$ is a pseudo-Anosov map and
$(S_h, \Lambda)$ is a hyperbolic metric on $S$ together with a pair of
invariant measured laminations.  Suppose that $\nu$ and $\rnu$ are the
forward and backward hitting measures on $\partial \widetilde S_h$
arising from a non-elementary, full random walk on $\pi_1(S)$.  Let
$\iota_* \nu$ and $\iota_* \rnu$ be their pushforwards under the
Cannon-Thurston map.  Then there are constants $R \ge 0$ and
$\epsilon > 0$ such that for
$(\iota_* \nu \times \iota_* \rnu)$-almost all geodesics
$\overline{\gamma}$ in $\widetilde S_h \times \RR$,
\[ \lim_{T \to \infty} \frac{1}{T} | \{ t \in [0, T] \mid \overline{\gamma}(t) \in
N_R( S_0 ) \} | \ge \epsilon.  \]
\end{lemma}

\begin{proof}
Let $\gamma$ be a bi-infinite geodesic arising from the limit points
of a bi-infinite random walk.  We fix a constant $L > 0$ such that
$\tfrac{L}{Q} - c > 8 \delta_3$, where $Q$ and $c$ are constants in
\Cref{prop:axis projection}. Suppose that $g$ is a non-trivial
element of $\pi_1(S)$, with axis $\alpha$ in $\widetilde S_h$, and
axis $\overline{\alpha}$ in $\widetilde S_h \times \RR$.  As open sets have positive hitting
measure, there is a positive probability that the length of the
projection interval $p_\alpha(\gamma)$ in $\widetilde S_h$ has length
at least $L$.  By \Cref{prop:axis projection}, the diameter of the
projection image $p_{\overline{\alpha}}(\overline{\gamma})$ in
$\widetilde S_h \times \RR$ has diameter at least
$\tfrac{1}{Q_\alpha} L - c$.  As
$\tfrac{1}{Q_\alpha} L_1 - c > 8 \delta_3$, \Cref{prop:neighbourhood}
implies the projection image
$p_{\overline{\alpha}}(\overline{\gamma})$ is contained in a
$6 \delta_3$-neighborhood of $\overline{\gamma}$.  So there are points
on $\overline{\gamma}$ distance
$\tfrac{1}{Q_\alpha} L_1 - c - 8 \delta_3$ apart, such that the
interval between them is contained in a $6 \delta_3$-neighborhood of
$\overline{\alpha}$.  By ergodicity, this happens linearly often for
translates of $\alpha$.  As the projection map $p_{\overline{\gamma}}$
from $\iota(\gamma)$ to $\overline{\gamma}$ is distance decreasing,
$\overline{\gamma}$ spends a positive proportion of time within
distance $6 \delta_3$ of $S_0$.
\end{proof}

We now show that for hitting measure on
$\partial ( \widetilde S_h \times \RR )$ arising from a geometric
random walk on $\pi_1(M)$, for almost all geodesics in
$\widetilde S_h \times \RR$, the proportion of their length which is
close to the base fiber $S_0$ tends to zero.  This completes the proof
of \Cref{theorem:3-manifold measures} by showing the hitting
measure case for $3$-manifold measures.  Theorems \ref{theorem:surface
  measures} and \ref{theorem:3-manifold measures} then imply \Cref{theorem:singularity}.

\begin{proposition}\label{prop:zero limit}
Suppose that $\pa \colon S \to S$ is a pseudo-Anosov map and
$(S_h, \Lambda)$ is a hyperbolic metric on $S$ together with a pair of
invariant measured laminations.  Suppose $\nu$ and $\rnu$ are the
forward and backward hitting measures on
$\partial ( \widetilde S_h \times \RR )$ arising from a geometric
random walk on $\pi_1(M)$.  Then for $(\nu \times \rnu)$-almost all
geodesics $\ogamma$ in $\widetilde S_h \times \RR$, for any unit speed
parametrization $\ogamma(t)$, and any constant $R > 0$,
\[ \lim_{T \to \infty} \frac{1}{T} | \{ t \in [0, T] \mid \ogamma(t) \in N_R( S_0 )
\} | = 0.  \]
\end{proposition}

We shall use the following estimate, which is a consequence of the
Local Central Limit Theorem for random walks on $\ZZ$, see for example
\cite{lawler-limic}*{Section 2}.

\begin{proposition}\cite{lawler-limic}*{Proposition 2.4.4}\label{prop:lclt}
Suppose $\phi_* \mu$ generates an aperiodic random walk on $\ZZ$.
Then there is a constant $C$ such that for all $n$ and $x$,
\[ \PP( \phi(w_n) = x ) \le \frac{C}{\sqrt{n}}.  \]
\end{proposition}

\begin{proof}[Proof of \Cref{prop:zero limit}]
We sketch the key steps before giving the technical details.
\begin{itemize}
    \item By the epimorphism $\phi: \pi_1(M) \to \ZZ$, the $\mu$-random walk projects to an aperiodic random walk on $\ZZ$.
    \item By \Cref{prop:lclt}, the random walk on $\ZZ$ recurs to $O(\log n)$ neighborhoods of zero with negligible probability (as $n \to \infty$).
    \item The epimorphism $\phi$ is distance non-increasing for the Cannon-Thurston metric. 
    So we deduce the same statement for recurrence of $\mu$-random walk to $O(\log n)$ neighborhoods of $S_0$.
    \item At the same time, the $\mu$-random walk recurs to a $O(\log n)$ neighborhoods of the tracked geodesic with probability negligibly smaller than 1.
    \item Choosing the size of the neighborhood of the tracked geodesic a definite proportion smaller than the neighborhood of $S_0$, we derive the fact that the closest points on the tracked geodesic stay $O(\log n)$ distance away from $S_0$ with probability negligibly smaller than 1.
    \item The probability estimates imply a sub-linear upper bound for the expectation of the number of times the tracked geodesic recurs to a $O(\log n)$ neighborhood of $S_0$.  
    \item Using the expectation bound and linear progress with exponential decay, we conclude the proof of \Cref{prop:zero limit}.
\end{itemize}

Let $\phi_* \mu$ denote the pushforward of $\mu$ by
the homomorphism $\phi \colon \pi_1(M) \to \ZZ$.  The $\mu$-random
walk projects to a $\phi_* \mu$ random walk on $\ZZ$.  As the support
of $\mu$ generates $\pi_1(M)$ as a semigroup, the support of
$\phi_* \mu$ generates $\ZZ$ as a semigroup.  In particular, by
\Cref{prop:lclt}, there is a constant $C_1$ such that for all $n$ and
any $A > 0$,
\[ \PP( | \phi(w_n) | \le 2 A \log n ) \le 2 A C_1 \log n /
\sqrt{n}. \]
As the above inequality holds for all $A > 0$, we may choose
\begin{equation}\label{eq:largeA}
A = \frac{1}{ 2\log \tfrac{1}{c} },
\end{equation}
where $c < 1$ is the exponential decay constant
from \Cref{prop:gp}.

\medskip
As the distance between any pair of adjacent fibers $S \times \{ n \}$
and $S \times \{ n + 1\}$ is equal to one in the Cannon-Thurston
metric, $| \phi(w_n) | \ge 2 A \log n$ implies that
$d_{\widetilde S_h \times \RR}( w_n x_0, S_0) \ge 2 A \log n$.  In
particular, it is very likely that $w_n x_0$ is far from $S_0$, i.e.
\begin{equation}\label{eq:dv}
\PP \left( d_{\widetilde S_h \times \RR}(w_n x_0, S_0) \ge 2 A \log
n \right) \ge 1 - \frac{ 2 A C_1 \log n }{ \sqrt{n} }.
\end{equation}
Recall that $\gamma(t_n)$ is a closest point on $\gamma$ to $w_n x_0$.
The Gromov product estimate, \Cref{prop:gp}, implies that it is likely
that $w_n x_0$ is logarithmically close to $\gamma$, i.e.
\begin{equation}\label{eq:dgamma}
\PP \left( d_{\widetilde S_h \times \RR}(w_n x_0, \gamma(t_n) ) \le
A \log n \right) \ge 1 - K c^{A \log n}.
\end{equation}
Combining \eqref{eq:dv} and \eqref{eq:dgamma} above, and using the
triangle inequality, shows that it is likely that $\gamma(t_n)$ is far
from $S_0$, i.e.
\[ \PP \left( d_{\widetilde S_h \times \RR}( \gamma(t_n), S_0 ) \ge A
\log n \right) \ge 1 - \frac{ 2 A C_1 \log n }{ \sqrt{n} } - K n^{A \log
  c}. \]
Taking the complementary event in the line above shows that it is
unlikely that $\gamma(t_n)$ is close to $S_0$, and using the fact that
our choice of $A$ from \eqref{eq:largeA} makes $A \log \tfrac{1}{c} = 1/2$, gives
\[ \PP \left( d_{\widetilde S_h \times \RR }( \gamma(t_n), S_0 ) \le A
\log n \right) \le \frac{ 2 A C_1 \log n + K}{ \sqrt{n} }. \]

Let $X_n$ be the number of times $\gamma(t_k)$ is within distance
$A \log k$ of $S_0$ for $1 \le k \le n$.  Then an elementary integral
comparison bound says that there is a constant $C_2$ such that 
\[ \EE(X_n) \le 2 A C_1 \sqrt{n} \log n + 2 K \sqrt{n} + C_2. \]
For any random variable, the Markov inequality says $\PP(X \ge t) \le
\EE(X)/t$, so choosing $t = \sqrt{n} (\log n)^2$ gives
\begin{equation}\label{eq:close}
\PP( X_n \ge \sqrt{n} (\log n)^2 ) \le \frac{ 2 A C_1 }{ \log n } +
\frac{ 2 K }{ (\log n)^2 } + \frac{C_2}{\sqrt{n}(\log n)^2}.
\end{equation}

Let $\beta_n = \gamma( [t_{n-1}, t_n] )$, and set
\[ B_n = \bigcup_{1 \le k \le n} \beta_k. \]
Recall that by \Cref{cor:tlinear}, $t_n$ makes linear progress with
exponential decay.  In particular, there is a constant $\ell > 0$ such
that
\[ \PP \left( \gamma( [ 0 , \ell n] )\subseteq B_n \subseteq \gamma( [0, 2
\ell n] ) \right) \ge 1 -
K c^n. \]

By \Cref{prop:segment bound}, there is a constant $D > 0$ such that
the probability that $|\beta_k| \le D \log k$ for all $k \ge \log n$
tends to one as $n$ tends to infinity.  Let $D_n$ be the union of all
$\beta_k$, for $1 \le k \le n$, such that any point on $\beta_k$ is
within distance $(A + D) \log k$ of $S_0$.  By \eqref{eq:close}, the
number of such intervals is at most $\sqrt{n} (\log n)^2$, with
probability that tends to one as $n$ tends to infinity.  The
probability that the union of the first $\log n$ segments $B_{\log n}$
has total length at most $2 \ell \log n$ tends to one as $n$ tends to
infinity.  Therefore,
\[ \PP \left( \frac{| D_n |}{| B_n |} \le \frac{2 \ell \log n + D
  \sqrt{n}(\log n)^3}{ \ell n} \right) \to 1 \text{ as } n \to \infty. \]
In particular, the proportion of points in $\gamma([0, t_n])$ which
lie within distance $K$ of $S_0$ tends to zero as $n$ tends to
infinity, as required.
\end{proof}

\section{Effective bounds for surface measures}\label{section:effective}

In this section, we prove \Cref{theorem:effective} using several results that we state and prove in this section, and using our construction of quasigeodesics in $\wsr$ in our companion paper \cite{gmpu}. 


\medskip
Suppose that $\gamma$ is a non-exceptional geodesic with unit speed
parametrization.  The inclusion map $\iota$ embeds $\gamma$ in $\wsr$
at height $z = 0$.  The image $\iota(\gamma)$ is, in general, not a
quasigeodesic.  However, we will show that the height for each point of $\iota(\gamma)$ can be changed in a specified way so that the resulting path is a quasigeodesic, i.e. there is a
function $h_\gamma(t)$ such that $(\gamma(t), h_\gamma(t))$ is an
(unparametrized) quasigeodesic.

\medskip
In order to define the function $h_\gamma$, we need the following mild
generalization of a measured lamination.  A
measured lamination is \emph{maximal} if every complementary region is
an ideal triangle.  By adding finitely many leaves to divide every ideal complementary region of a measured lamination $\Lambda$ into ideal triangles we may extend $\Lambda$ to a maximal lamination. There are thus only finitely many
such maximal laminations containing $\Lambda$.  We will call the union
of these maximal laminations the \emph{extended lamination}
$\overline{\Lambda}$, which in general is not itself a measured
lamination. See \cite{gmpu}*{Section 2.2} for more details.


\medskip
Let $\oLambda$ be the union of the two extended laminations obtained from the invariant laminations
for $\pa$, and let $\oLambda^1$ be its lift in $T^1(S_h)$.  Let
$h \colon T^1(S_h) \setminus \oLambda^1 \to \RR$ be a continuous
function.  Then $h$ determines an embedding in $\wsr$ of any non-exceptional
unit-speed geodesic $\gamma$ by the map $t \mapsto (\gamma(t), h(\gamma^1(t)))$,
where $\gamma(t)$ is the oriented geodesic and
$\gamma^1(t)$ is the unit tangent vector to $\gamma$ at $\gamma(t)$.  We
shall call $h$ the \emph{height function}, and the embedding
$\tau_\gamma(t) = ( \gamma(t), h(\gamma^1(t)) )$ the \emph{test path}
for $\gamma$ determined by $h$.

\medskip
We now specify the height function we will use, see
\cite{gmpu}*{Section 3}
for background and motivation.

\medskip

We shall write $\log_k$ for the logarithm function with base $k$,
where $k = k_f > 1$ is the stretch factor of $\pa$.
\begin{definition}\label{def:height function2}
Suppose that $\pa \colon S \to S$ is a pseudo-Anosov map, let
$(S_h, \Lambda)$ is a hyperbolic metric on $S$ together with a pair of
invariant measured laminations, and $\oLambda^1_-$ and
$\overline{\Lambda}^1_+$ the lifts in $T^1(S_h)$ of the extended
laminations given by the invariant laminations of $\pa$, and let
$\theta > 0$ be a positive constant.
We define the \emph{height function}
$h_\theta \colon T^1(S_h) \setminus (\overline{\Lambda}^1_+ \cup
\overline{\Lambda}^1_-) \to \RR$ to be
\[ h_\theta(v) = \log_k \left\lfloor \log \frac{1}{d_{\PSL(2, \RR)}(v,
  \overline{\Lambda}^1_+ )} - \log \frac{1}{\theta} \right\rfloor_1
- \log_k \left\lfloor \log \frac{1}{d_{\PSL(2, \RR)}(v,
  \overline{\Lambda}^1_- )} - \log \frac{1}{\theta}
\right\rfloor_1 \]
\end{definition}

Here $\lfloor x \rfloor_c = \max \{ x, c \}$ is the standard floor
function.  As the two extended laminations are a positive distance
apart in $T^1(S_h)$, for sufficiently small $\theta$, at most one
of the terms on the right hand side above will be non-zero.

\medskip
We prove that for a choice of $\theta$ sufficiently small, the test path determined
by the corresponding height function is a quasigeodesic in $\wsr$.

\begin{theorem}\label{theorem:quasigeodesic-h2}
Suppose that $\pa \colon S \to S$ is a pseudo-Anosov map and
$(S_h, \Lambda)$ is a hyperbolic metric on $S$ together with a pair of
invariant measured laminations.  Then there are constants
$\theta > 0$, $Q \ge 1$ and $c \ge 0$, such that for any
non-exceptional geodesic $\gamma$ in $S_h$, with a unit speed
parametrization $\gamma(t)$, the test path
$\tau_\gamma(t) = (\gamma(t), h_\theta(\gamma^1(t)) )$ is an
unparametrized $(Q, c)$-quasigeodesic in $\widetilde S_h \times \RR$
with the same limit points as $\iota(\gamma)$, where $h_\theta$ is the
height function from \Cref{def:height function2}.
\end{theorem}

\medskip
We shall now fix a sufficiently small constant $\theta$ in \Cref{theorem:quasigeodesic-h2} and simplify notation to just write $h$ for $h_\theta$.  See
\cite{gmpu}*{Section 3.2}
for the exact choice of $\theta$ that
we use.  Furthermore, we will write $h_\gamma(t)$ for
$h_\theta(\gamma^1(t))$.

\medskip
We will use one further property of these quasigeodesics.  The test
path $\tau_\gamma$ lies in the ladder $F(\gamma)$ determined by
$\gamma$, so vertical projection gives a map
$(\gamma(t), 0) \mapsto (\gamma(t), h_\gamma(t))$ from $\iota(\gamma)$
to the test path $\tau_\gamma$.  We will prove that this map is coarsely distance non-increasing.

\begin{proposition}\label{prop:vertical flow distance decreasing}
Suppose that $\pa \colon S \to S$ is a pseudo-Anosov map and
$(S_h, \Lambda)$ is a hyperbolic metric on $S$ together with a pair of
invariant measured laminations.  There are constants $K > 0$ and
$c \ge 0$ such that for any non-exceptional geodesic $\gamma$ with
unit speed parametrization, and any real numbers $s$ and $t$,
\[ d_{\wsr}( \tau_\gamma(s), \tau_\gamma(t) ) \le K d_{\wsr} \left(
\iota(\gamma(s)), \iota(\gamma(t)) \right) + c , \]
where here the test path has the parametrization inherited from the
unit speed parametrization on $\gamma$.
\end{proposition}

\subsection{Lebesgue measure}

Assuming \Cref{theorem:quasigeodesic-h2} and \Cref{prop:vertical flow distance decreasing}, we now derive \Cref{theorem:surface lebesgue} giving effective bounds for the amount of time that a geodesic
chosen according to Lebesgue measure on the boundary circle of
$\widetilde S_h$ spends close to the base fiber $S_0$.  This implies the
first half of \Cref{theorem:effective}.

\begin{theorem}\label{theorem:surface lebesgue}
Suppose that $\pa$ is a pseudo-Anosov map with stretch factor $k > 1$,
and $(S_h, \Lambda)$ is a hyperbolic metric on $S$ together with a
pair of invariant measured laminations.  Let $\nu$ be the pushforward
of Lebesgue measure on $\partial \ws_h$ under the Cannon-Thurston map.
Then there are constants $K > 0$ and $\alpha > 0$ such that for
$(\nu \times \nu)$-almost all geodesics $\ogamma$ in $\wsr$, for any
unit speed parametrization $\overline{\gamma}(t)$, for any $R \ge 0$,
\[ \lim_{T \to \infty} \frac{1}{T} \left| \overline{\gamma}([0, T])
\cap N_R(S_0) \right| \ge 1 - K e^{- \alpha k^R}. \]
\end{theorem}

Here is an overview of the argument.

\begin{enumerate}

\item Let $\overline{\Lambda}^1$ be the union of the lifts of
both extended laminations in $T^1(S_h)$.  By 
\cite{gmpu}*{Definition 26}
of the height
function $h_\gamma$, if $| h_\gamma(t) | \ge R$, then
the tangent vector at $\gamma(t)$ lies in $N_r(\overline{\Lambda}^1)$,
for $r = K e^{-k^R}$.

\item  Work of Birman--Series \cite{birman-series} shows that the
volume of $N_r(\overline{\Lambda}^1)$ goes to zero as $r \to 0$ at the rate
$r^2 (\log \tfrac{1}{r})^{6g-6}$.

\item Suppose that $\gamma$ is a geodesic chosen according to Lebesgue measure
on $\partial \ws_h$.  The geodesic flow on $T^1(S_h)$ is ergodic with
respect to Liouville measure, which is the product of the hyperbolic
metric on $S_h$ with Lebesgue measure on the unit tangent circles.
Therefore almost all geodesics $\gamma$ are uniformly distributed in
$T^1(S_h)$. In particular, we deduce that 
\begin{itemize}
    \item we may fix $R_0$ large enough and hence $r_0 = K e^{-kR_0}$ small enough such that $\gamma$ recur to $T^1(S_h) \setminus N_{r_0}(\overline{\Lambda}^1)$, and
    \item the proportion of time along $\gamma$  for which $| h_\gamma(t) |$ is at least $R$ is at most  $O(e^{-k^R})$, where the time is being measured in terms of a unit speed parametrization of $\gamma$ in $\ws_h$.
\end{itemize}

\item By the work of Gadre--Hensel \cite{Gad-Hen} the distance in $\wsr$ from $\iota(\gamma)(0)$ to $\iota (\gamma)(T)$  grows linearly in $T$. From the linear growth and recurrence to $T^1(S_h) \setminus N_{r_0}(\overline{\Lambda}^1)$, we deduce that the distance $d(t)$ between $\tau_\gamma(0)$ and $\tau_\gamma(t)$ also grows linearly in $T$. 
From this, we deduce that, as a proportion of $d(T)$, the time along $\tau_\gamma$ for which $| h_\gamma |$ is at least $R$ is again at most $O(e^{-k^R})$.

\item Let $\ogamma$ be the geodesic in $\wsr$ determined by $\gamma$.
As the test path $\tau_\gamma$ is quasigeodesic, there is a constant
$L$ such that the test path is contained in an $L$-neighborhood of
$\ogamma$. Therefore, parametrizing $\ogamma$ with unit speed, the proportion of times on $\ogamma$ outside
$N_R(S_0)$ goes to zero at rate $e^{-k^{R - L}}$, as desired.

\end{enumerate}

Step 1 requires no further elaboration.  We now justify Step 2.

\medskip
Let $\Lambda$ be a geodesic lamination in the surface $S_h$, and let
$N_r(\Lambda)$ be the set of all points in $S$ distance at most $r$
from $\Lambda$.  Birman and Series \cite{birman-series} give the
following bounds for the area of $N_r(\Lambda)$.

\begin{theorem}\cite{birman-series}\label{theorem:birman-series}
Suppose that $S$ is a surface with a complete finite-area hyperbolic metric.
Then there are constants $A> 0$ and $r_0>0$, such that for any geodesic lamination 
$\Lambda$ on $S$, and for all $r \le r_0$,
\[ \frac{1}{A} r \le \mathbf{area}(N_r(\Lambda)) \le A r (\log \tfrac{1}{r})^{6g-6}.  \]
\end{theorem}

The upper bound follows from \cite{birman-series}*{Proposition 4.1}
with the degree of the exponent given by \cite{birman-series}*{Remark
  7.2}.  Birman--Series state only the upper bound. The lower bound is
immediate from the observation that any geodesic lamination
contains an embedded simple arc, and the area of an $r$-neighborhood
of a simple arc is proportional to $r$ for $r$ sufficiently small.
The theorem above gives the following immediate bounds on the volumes of an
$r$-neighborhood of $\Lambda$ in the unit tangent bundle.  Let
$\Lambda^1$ be the pre-image of $\Lambda$ in the unit tangent bundle
$T^1(S)$, and let $N_r(\Lambda^1)$ be the set of all points in
$T^1(S)$ distance at most $r$ from $\Lambda^1$.

\begin{corollary}\label{cor:volume}
Suppose that $S$ is a surface with a complete finite-area hyperbolic metric.  
Then there are constants $A>0$ and $r_0>0$, such that
for any extended geodesic lamination $\overline{\Lambda}$ on $S$, and
all $r \le r_0$,
\[ \frac{1}{A} r^2 \le  \mathbf{vol}(N_r(\overline{\Lambda}^1)) \le A r^2 (\log \tfrac{1}{r})^{6g-6}.  \]
\end{corollary}

\begin{proof}
A geodesic lamination $\Lambda$ has finitely many complementary
regions and can be extended in finitely many ways to a maximal lamination $\Lambda_i$ by dividing each complementary region into ideal
triangles.  The extended
lamination $\overline{\Lambda}$ is thus contained in the union of
finitely many laminations $\Lambda_i$.
The result follows immediately from \Cref{theorem:birman-series},
by replacing $A$ by $An$, where $n$ is the number of geodesic
laminations in the collection $\Lambda_i$.
\end{proof}

Step 3 follows from \Cref{cor:volume} and ergodicity of the geodesic flow on $T^1(S_h)$.
Using 
\cite{gmpu}*{Definition 26},
we may write a version of \Cref{theorem:surface lebesgue}
for the test path $\tau_\gamma$, using the parametrization
of $\gamma$ in $\widetilde{S}_h$, instead of a unit speed
parametrization of the test path.

\begin{proposition}\label{prop:surface speed}
Suppose $\pa$ is a pseudo-Anosov map with stretch factor $k > 1$, and
$(S_h, \Lambda)$ is a hyperbolic metric on $S$ together with a pair of
invariant measured laminations.  Then there is a constant $K > 0$ such
that for Lebesgue-almost all geodesics $\gamma$ in $\widetilde S_h$,
for any unit speed parametrization of the geodesic $\gamma(t)$, for
any $R \ge 0$,
\[ \lim_{T \to \infty} \frac{1}{T} \left| \tau_\gamma([0, T])
\setminus N_R(S_0) \right| \le K e^{ - k^R}, \]
where $\tau_\gamma(t)$ is the parametrization induced by $\gamma(t)$,
not the arc length parametrization.
\end{proposition}

\begin{proof}
By ergodicity of the geodesic flow, the amount of time almost every
geodesic spends within distance $r$ of the union of the extended
laminations $\overline{\Lambda}^1$ is equal to the volume of
$N_r(\overline{\Lambda}^1)$.
\begin{align}
\lim_{T \to \infty} \frac{1}{T} \left| \gamma^1([0, T]) \cap
N_r(\overline{\Lambda}^1) \right| & =
\mathbf{vol}(N_r(\overline{\Lambda}^1)) \nonumber \\
\intertext{By \Cref{cor:volume},}
\lim_{T \to \infty} \frac{1}{T} \left| \gamma^1([0, T]) \cap
N_r(\oLambda^1) \right| & \le  A r^2 (\log \tfrac{1}{r})^{6g-6}. \label{eq:rbound}  \\
\intertext{By the definition of the test path, 
\cite{gmpu}*{Theorem 27},
if
$\gamma(t)$ lies in $N_r(\overline{\Lambda}^1)$ then the corresponding
point on the test path $\tau_\gamma(t)$ lies outside $N_R(S_0)$, where
$r = Ke^{-k^R}$.  Rewriting the upper bound in \eqref{eq:rbound} in terms of $R$ gives}
\lim_{T \to \infty} \frac{1}{T} \left| \tau_\gamma([0, T])
\setminus N_R(S_0) \right| & \le A K^2 e^{- 2 k^R} k^{(6g-6)R}. \nonumber \\
\intertext{Furthermore, for sufficiently large $R$, $e^{k^R} \ge k^{(6g - 6)R}$,
so for an appropriate choice of $K_1 = A K^2$,}
\lim_{T \to \infty} \frac{1}{T} \left| \tau_\gamma([0, T])
\setminus N_R(S_0) \right| & \le K_1 e^{-k^R}, \nonumber \\
\end{align}
as required.
\end{proof}

To deduce Step 4, we first state the following result of Gadre--Hensel \cite{Gad-Hen} showing that the distance in
$\widetilde{S}_h \times \RR$ along $\iota(\gamma)$ grows linearly.

\begin{theorem}\cite{Gad-Hen}*{Theorem 2.2}\label{theorem:gh-linear}
Suppose that $\pa \colon S \to S$ is a pseudo-Anosov map and
$(S_h, \Lambda)$ is a hyperbolic metric on $S$ together with a pair of
invariant measured laminations.  Then there is a constant
$\epsilon > 0$ such that for Lebesgue-almost all geodesics $\gamma$ in
$\widetilde{S}_h$, with unit speed parametrization, we have
\[
\frac{1}{T} d_{\widetilde{S}_h \times
  \RR}(\iota(\gamma(0)), \iota(\gamma(T))) > \epsilon
\]
for all $T$ large enough depending on $\gamma$.
\end{theorem}

We deduce from \Cref{theorem:gh-linear}, the recurrence of $\gamma$ to $T^1(S_h) \setminus N_{r_0}(\overline{\Lambda}^1)$, and the quasi-geodesicity of $\tau_\gamma$, that the distance in $\wsr$ along the test path $\tau_\gamma$ also grows linearly
with respect to the parametrization coming from $\gamma$.

\begin{corollary}\label{cor:gh-linear}
Suppose that $\pa \colon S \to S$ is a pseudo-Anosov map and
$(S_h, \Lambda)$ is a hyperbolic metric on $S$ together with a pair of
invariant measured laminations.  Then there is a constant
$\epsilon > 0$ such that for Lebesgue-almost all geodesics $\gamma$ in
$\widetilde{S}_h$, with unit speed parametrization, we have
\[
\frac{1}{T} d_{\wsr} \left( \tau_\gamma(0),
\tau_\gamma(T) \right) > \epsilon. 
\]
for all $T$ large enough depending on $\gamma$.

\end{corollary}

\begin{proof}
Suppose that $\gamma$ is a non-exceptional geodesic in $\ws_h$ sampled with respect to the Lebesgue measure and parametrized by unit speed. 
Suppose that the test path $\tau_\gamma$ is given the parametrization from $\gamma$.

\medskip
We may fix $R_0$ large enough and hence $r_0 = K e^{-kR_0}$ small enough such that for Lebesgue-almost all $\gamma$ the time that $\gamma([0,T])$ spends in $T^1(S_h) \setminus N_{r_0}(\overline{\Lambda}^1)$ is strictly greater than $2/3$ for all $T$ large enough.
Suppose that $t_0 > 0$ is the smallest time for which $\gamma(t_0)$ lies in $T^1(S_h) \setminus N_{r_0}(\overline{\Lambda}^1)$, and suppose that $t_1< T$ is the largest time for which $\gamma(t_1 )$ lies in $T^1(S_h) \setminus N_{r_0}(\overline{\Lambda}^1)$. It follows that $t_0 < T/3$ and $t_1 > 2T/3$, and so $t_1 - t_0 \ge T/3$.
\medskip
By the triangle inequality,
\begin{align*}
  d_{\wsr}( \tau_\gamma(t_0), \tau_\gamma(t_1) ) & \ge d_{\wsr}(
\iota(\gamma(t_0)), \iota(\gamma(t_1)) ) - d_{\wsr}( \iota(\gamma(t_0)) ,
\tau_\gamma(t_0)) - d_{\wsr}( \iota(\gamma(t_1)) , \tau_\gamma(t_1)). \\
  \intertext{By definition of the test path, the final two terms on
  the right hand side are equal to the absolute value of the
  height function, and so}
d_{\wsr}( \tau_\gamma(t_0), \tau_\gamma(t_1) ) & \ge d_{\wsr}(
\iota(\gamma(t_0)), \iota(\gamma(t_1)) ) - | h_\gamma(t_0) | - | h_\gamma(t_1) |  \\
&\ge d_{\wsr}(
\iota(\gamma(t_0)), \iota(\gamma(t_1)) ) - R_0 - R_0.
\intertext{By \Cref{theorem:gh-linear}, $d_{\wsr}(\iota(\gamma(t_0)), \iota(\gamma(t_1)) ) \ge \epsilon (t_1 - t_0)$, and so
}
d_{\wsr}( \tau_\gamma(t_0), \tau_\gamma(t_1) ) & \ge \epsilon (t_1 - t_0) - 2R_0 \ge \frac{1}{3} \epsilon T - 2R_0.
\end{align*}
We conclude the proof by noting that the test path $\tau_\gamma$ is a quasi-geodesic in $\wsr$ and hence the distance $d_{\wsr}( \tau_\gamma(t_0), \tau_\gamma(t_1) )$ is a coarse lower bound for $d_{\wsr}(\tau_\gamma (0), \tau_\gamma(T))$.
\end{proof}

\medskip
Finally, we prove Step 5, completing the proof of \Cref{theorem:surface
  lebesgue}.

\begin{proof}[Proof (of \Cref{theorem:surface lebesgue})]
We obtain from $\gamma$ a parametrization $\mathbb{R} \to \overline{\gamma}$ by letting $\overline{\gamma}_t$ be a point of $\overline{\gamma}$ closest to $\tau_\gamma(t)$.
Thus $\overline{\gamma}_0$ and $\overline{\gamma}_T$ are points of $\overline{\gamma}$ closest to $\tau_\gamma(0)$ and $\tau_\gamma(T)$ respectively. 
Since $\tau_\gamma$ is a quasi-geodesic there is a constant $L$ such that $d_{\wsr} (\overline{\gamma}_0, \tau_\gamma(0)) < L $ and $d_{\wsr} (\overline{\gamma}_T, \tau_\gamma(T)) < L$.
By triangle inequality,
\[
d_{\wsr}(\overline{\gamma}_0, \overline{\gamma}_T) \ge d_{\wsr} (\tau_\gamma(0),\tau_\gamma(T)) - 2L.
\]
By \Cref{cor:gh-linear}, there is then an $\epsilon > 0$ such that 
\begin{equation}\label{eq:ogamma lower bound}
d_{\wsr}(\overline{\gamma}_0, \overline{\gamma}_T) \ge \epsilon T 
\end{equation}
On the other hand, the projection $\iota(\gamma)$ to $\tau_\gamma$ along flow lines is distance decreasing. Hence, it follows that
\begin{equation}\label{eq:ogamma upper bound}
    d_{\wsr}(\overline{\gamma}_0, \overline{\gamma}_T) \le T+ 2L.
\end{equation}

\medskip
By \Cref{prop:surface speed}, for Lebesgue almost all geodesics $\gamma$, there
is a $T_1(\gamma)$, such that for all $T \ge T_1(\gamma)$,
\[ \left| \tau_\gamma([0, T]) \setminus N_R(S_0) \right| \le
T K e^{- k^R}.  \]
As $\overline{\gamma}$ is contained in an $L$-neighborhood of the
test path $\tau_\gamma$, we deduce
\[
\left| \overline{\gamma}([0, T]) \setminus N_{R+L} (S_0) \right| \le T K e^{- k^R}
\]
where $\overline{\gamma}$ is parametrized from $\gamma$.
Finally, suppose that $\overline{\gamma}(D) = \overline{\gamma}_T$ in the arc length parametrization of $\overline{\gamma}$ in which $\overline{\gamma}(0) = \overline{\gamma}_0$. Then, by \Cref{eq:ogamma lower bound}, 
\[
\left| \overline{\gamma}([0, T]) \setminus N_{R+L} (S_0) \right| \le \frac{1}{\epsilon} D K e^{- k^R}
\]
from which, by tweaking constants, we may deduce \Cref{theorem:surface lebesgue}, as required.
\end{proof}

\subsection{Hitting measure}

Again assuming \Cref{theorem:quasigeodesic-h2} and \Cref{prop:vertical flow distance decreasing}, we now derive \Cref{theorem:effective hitting} giving effective bounds for the proportion of time that a geodesic
chosen by a hitting measure on the boundary circle of
$\widetilde S_h$ spends close to the base fiber $S_0$.  This completes
the second half of \Cref{theorem:effective}.

\begin{theorem}\label{theorem:effective hitting}
Suppose that $\pa$ is a pseudo-Anosov map with stretch factor $k > 1$,
and $(S_h, \Lambda)$ is a hyperbolic metric on $S$ together with a
pair of invariant measured laminations.  Suppose that $\mu$ is a
finitely supported probability measure on $\pi_1(S)$ whose support
generates $\pi_1(S)$ as a semigroup. Let $\nu$ and $\rnu$ be the
forward and backwards hitting measures on $\partial \ws_h$. Then there
are constants $K > 0$ and $\alpha > 0$ such that for
$(\iota_* \nu \times \iota_* \rnu)$-almost all geodesics $\ogamma$ in
$\widetilde S_h \times \RR$, for any unit speed parametrization
$\ogamma(t)$, for any $R \ge 0$,
\[ \lim_{T \to \infty} \frac{1}{T} | \ogamma([0, T]) \setminus
N_R(S_0) \} | \le K e^{- \alpha k^R}. \]
\end{theorem}

Here is an overview of the argument.

\begin{enumerate}
    
\item Suppose that $\gamma$ is a geodesic in $\widetilde{S}_h$ chosen according
to the hitting measure $\nu \times \rnu$ on
$\partial \widetilde{S}_h \times \partial \widetilde{S}_h$.  Suppose that the tangent vector at $\gamma(t)$ is within distance $r>0$ of one of the
extended laminations $\Lambda$.
If $r$ is very small then the endpoints of $\gamma$ are
within distance $r + o(r)$ of the limit set of $\Lambda$, when $\partial \widetilde{S}_h$ is equipped with the angular metric viewed from $\gamma(t)$. Work of Birman--Series
\cite{birman-series} shows that the hitting measure of
$N_r(\partial \Lambda)$ goes to zero at rate $r^\alpha$.

\item Let $x_0$ be a choice of basepoint in $\widetilde{S}_h$, and let
$\gamma(t_n)$ be the closest point on $\gamma$ to $w_n x_0$.  The
previous argument shows that the probability that the tangent vector
at $\gamma(t_n)$ is within distance $r$ of one of the extended
laminations goes to zero at rate $r^\alpha$.  
Let $R$ be such that $r = e^{-k^R}$. By the definition of the
height function, the probability that $|h_\gamma(t_n)| \ge R$ goes to
zero at rate $e^{-\alpha k^R}$.

\item Since $\mu$ is assumed to have finite support, the distance
in $\ws_h$ between any two successive locations $w_n x_0$ and
$w_{n+1}$ of the random walk is bounded.  As nearest point projection
to geodesics is distance reducing, the distance between any two
successive nearest point projections $\gamma(t_n)$ and
$\gamma(t_{n+1})$ is also bounded.  We use the fact that the height
function is Lipschitz to show that that the distance in
$\ws \times \RR$ between the corresponding test path locations
$\tau_\gamma(t_n)$ and $\tau_\gamma(t_n+1)$ is also bounded.

\item We use the linear progress/ drift of the random walk in
$\wsr$ to show that the test path locations $\tau_\gamma(t_n)$ also
make linear progress in $\wsr$.  Let $\ogamma$ be the geodesic in
$\wsr$ determined by $\gamma$.  We parametrize $\ogamma$ by unit speed so that $\ogamma(0)$ is the point closest to the first test path
location $\tau_\gamma(t_0)$.  As the test path $\tau_\gamma$ is
contained in a bounded neighborhood of the geodesic $\ogamma$, the
number of test path locations close to the segment $\ogamma([0, T])$
grows linearly in $T$.

\item As the proportion of test path locations outside of $N_R(S_0)$
goes to zero at rate $e^{-\alpha k^R}$, the proportion of time
$\ogamma([0, T])$ spends outside $N_R(S_0)$ goes to zero at the same
rate.

\end{enumerate}

We start by estimating hitting measure for a regular neighborhood of
the limit set of the extended laminations.  We will use the following
result from the proof of \cite{birman-series}*{Theorem II, page 224},
with the degree of the polynomial from \cite{birman-series}*{Remark
  7.2}.

\begin{theorem}\cite{birman-series}\label{theorem:bs}
Suppose that $S_h$ is a closed hyperbolic surface and $\Lambda$ is a
geodesic lamination.  Let $x_0$ be a basepoint in $\widetilde S_h$.
Then there are constants $A, c$ and $\alpha > 0$ such that for any
integer $n > 0$, there are $A n^{6g-g}$ squares of side length
$c e^{-\alpha n}$ which cover $\partial \Lambda$ in
$(\partial S_h \times \partial S_h) \setminus \Delta$.
\end{theorem}

We use \Cref{theorem:bs} to estimate the hitting measure of a regular
neighborhood of the limit set of the extended laminations.

\begin{corollary}\label{cor:bscor}
Suppose that $\pa$ is a pseudo-Anosov map and $S_h$ a
hyperbolic metric.  Suppose that $\mu$ is a finitely supported
probability measure on $\pi_1(S)$ whose support generates $\pi_1(S)$
as a semigroup. Let $\nu \times \rnu$ be the resulting stationary
measure on $\partial \widetilde{S}_h \times \partial \widetilde{S}_h$.
Let $\partial \overline{\Lambda}$ be the limit set in
$\partial \widetilde{S}_h \times \partial \widetilde{S}_h$ of the union of the
extended invariant laminations.  Choose a
basepoint $x_0$ in $\widetilde{S}_h$, and give $\widetilde{S}_h$ the
visual metric based at $x_0$, and give
$\partial \widetilde{S}_h \times \partial \widetilde{S}_h$ the product
metric. Then there are constants $K > 0$ and $\alpha > 0$ such that
\[ \nu \times \rnu ( N_r( \partial \overline{\Lambda} ) ) \le K
r^{\alpha} \]
\end{corollary}

\begin{proof}
As noted before, the union of the extended laminations $\overline{\Lambda}$ is
contained in the union of a finite number of geodesic laminations
$\Lambda_1, \ldots \Lambda_k$. Let $A_1, c_1$ and $\alpha_1$ be the
constants from \Cref{theorem:bs}.  Setting $r = c_1 e^{-\alpha_1 n}$
in \Cref{theorem:bs}, each $N_r( \Lambda_i)$ is contained in the union
of at most $A_1 (\tfrac{1}{\alpha_1} \log \tfrac{c_1}{r})^{6g - 6}$
squares of side length $r$.  As an $r$-neighborhood of a square of
side length $r$ is contained in a union of $9$ squares of side length
$r$, this shows that $N_r(\overline{\Lambda})$ is contained in the
union of at most
$9 A_1 k (\tfrac{1}{\alpha} \log \tfrac{c_1}{r})^{6g - 6}$ squares of
side length $r$.

\medskip
The Gromov product of two points in the boundary is equal to
$\log (1 / \sin(\theta/2))$, where $\theta$ is the angle between them
viewed from the basepoint, see for example \cite{roe}*{page 114}.
Using the elementary bounds $x/2 \le \sin x \le x$ for $0 \le x \le 1$, 
we know that an interval of length $r$ in $\partial \widetilde{S}_h$ is
contained in a shadow of distance $\log \tfrac{1}{r}$ from the
basepoint.  By exponential decay of shadows, \Cref{lemma:exponential
  decay}, there is a constant $c_2 < 1$ such that the hitting measure
of an interval of length $r$ is at most
$K c_2^{\log \tfrac{1}{r}} = K r^\beta$ for some
$\beta = \log \tfrac{1}{c_2} > 0$.  So the hitting measure of a square
of side length $r$ is at most $K^2 r^{2 \beta}$.  This gives
\begin{align*}
  \nu \times \rnu( N_r( \partial \overline{\Lambda} ) ) & \le K^2 r^{2
\beta} \ 9 A_1 (\tfrac{1}{\alpha} \log \tfrac{c_1}{r})^{6g - 6}. \
\intertext{As $(\log \tfrac{c_1}{r})^{6g-6}$ is a polynomial in
$\log \tfrac{1}{r}$, there is a constant $A_2$ such that}
\nu \times \rnu( N_r( \partial \overline{\Lambda} ) ) & \le A_2 r^{2
  \beta} (\log \tfrac{1}{r})^{6g-6}. \\
\intertext{As $r^{\beta}( \log \tfrac{1}{r} )^{6g-6}$ tends to zero as $r$ tends
to zero, there is a constant $A_3$ such that}
\nu \times \rnu( N_r( \partial \overline{\Lambda} ) ) & \le A_3 r^{
  \beta},
\end{align*}
and so the result follows by setting $K = A_3$ and $\alpha = \beta$.
\end{proof}

Let $x_0$ be a basepoint for $\ws_h$.  As the distribution
$\mu$ generating the random walk has finite support, the distance
between any two successive locations $w_n x_0$ and $w_{n+1} x_0$ of
the random walk is bounded.  As closest point projection to a geodesic
is distance reducing, this implies that the distance between the
corresponding closest points $\gamma(t_n)$ and $\gamma(t_{n+1})$ on
$\gamma$ is also bounded.  We now show that the distance between the
corresponding test path locations $\tau_\gamma(t_n)$ and
$\tau_\gamma(t_{n+1})$ is bounded.

\begin{proposition}\label{prop:test path steps bounded}
Suppose that $\pa$ is a pseudo-Anosov map and $S_h$ a
hyperbolic metric.  Suppose that $\mu$ is a finitely supported
probability measure on $\pi_1(S)$ whose support generates $\pi_1(S)$
as a semigroup.  Let $x_0$ be a basepoint for $\widetilde{S}_h$.  Let
$\gamma$ be the geodesic with the same limit points as
$( w_n x_0 )_{n \in \ZZ}$, and let $\gamma(t_n)$ be the nearest point
projection of $w_n x_0$ to $\gamma$.  Then there is a constant $B$
such that for all $n$,
\[ d_{\widetilde{S} \times \RR} \left( \tau_\gamma(t_n),
\tau_\gamma(t_{n+1}) \right) \le B. \]
\end{proposition}

\begin{proof}
Let $x_0$ be a basepoint for $\ws_h$, and let
$( w_n x_0 )_{n \in \ZZ}$ be a bi-infinite sample path of the random
walk.  Let $\gamma$ be the geodesic determined by its limit points of the sample path.
By a unit speed parametrization of $\gamma$, we get a parametrization of the test path
$\tau_\gamma(t)$.

\medskip
As $\mu$ has finite support, there is an upper bound $B_1$ on the
distance between $w_n x_0$ and $w_{n+1} x_0$.  Let $\gamma(t_n)$ be
the nearest point projection of $w_n x_0$ to the geodesic $\gamma$ in
$\ws_h$.  As nearest point projection to a geodesic is distance
decreasing, the distance between $\gamma(t_n)$ and $\gamma(t_{n+1})$
is at most $B_1$.  As the inclusion map $\iota$ is distance
decreasing, the distance between $\iota(\gamma(t_n))$ and
$\iota(\gamma(t_{n+1}))$ is also at most $B_1$.

\medskip
By \Cref{prop:vertical flow distance decreasing}, there are constants
$K$ and $c$ such that the distance between the corresponding test path
locations $\tau_\gamma(t_1)$ and $\tau_\gamma(t_2)$ is at most
$B = K B_1 + c$, as required.
\end{proof}

We now show that the distance in $\wsr$ between the test path
locations $\tau_\gamma(t_0)$ and $\tau_\gamma(t_n)$ grows linearly in
$n$.

\begin{proposition}\label{prop:test path linear progress}
Suppose that $\pa$ is a pseudo-Anosov map and $S_h$ a
hyperbolic metric.  Suppose that $\mu$ is a finitely supported
probability measure on $\pi_1(S)$ whose support generates $\pi_1(S)$
as a semigroup.  Let $x_0$ be a basepoint for $\widetilde{S}_h$.  Let
$\gamma$ be the geodesic with the same limit points as $\{ w_n x_0
\}$, and let $\gamma(t_n)$ be the nearest point projection of $w_n
x_0$ to $\gamma$.  Then there is a constant $\ell > 0$ such that
\[ \lim_{N \to \infty} \frac{1}{N} d_{\widetilde{S} \times \RR} \left(
\tau_\gamma(t_0), \tau_\gamma(t_{N}) \right) = \ell > 0. \]
\end{proposition}

\begin{proof}
Let $\overline{x}_0 = \iota(x_0)$. By the triangle inequality applied to the path consisting of the
three geodesic segments
$[\overline{x}_0, \tau_\gamma(t_0)] \cup [\tau_\gamma(t_0),
\tau_\gamma(t_n)] \cup [\tau_\gamma(t_n), w_n \overline{x}_n] $, we get 
\begin{align*}
  d_{\widetilde{S}_h \times \RR}( \tau_\gamma(t_0),
  \tau_\gamma(t_n) )
  & \ge d_{\widetilde{S}_h \times \RR}(                        
  \overline{x}_0, w_n \overline{x}_0) ) - d_{\widetilde{S}_h \times
  \RR}( \overline{x}_0, \tau_\gamma(t_0) ) - d_{\widetilde{S}_h
  \times \RR}( w_n \overline{x}_0, \tau_\gamma(t_n) ) . \\
\intertext{
Set $A = d_{\widetilde{S}_h \times \RR}( \overline{x}_0,
  \tau_\gamma(t_0) )$, which is independent of $n$.  As the inclusion
  map $\iota$ is distance decreasing, the distance in $\widetilde{S}_h
  \times \RR$ from $w_n \overline{x}_0$ to $\tau_\gamma(t_n)$ is at
  most the distance in $\widetilde{S}_h$ from $w_n x_0$ to
  $\gamma(t_n)$, plus the distance from $\iota(\gamma(t_n))$ to
  $\tau_\gamma(t_n)$, which is given by the value of the height
  function $h_\gamma(t_n)$.  This gives
}  
d_{\widetilde{S}_h \times \RR}( \tau_\gamma(t_0),
  \tau_\gamma(t_n) ) & \ge d_{\widetilde{S}_h \times \RR}(
\overline{x}_0, w_n \overline{x}_0) ) - A - d_{\widetilde{S}_h}(
                                     w_n x_0, \gamma(t_n) ) - h_\gamma(t_n). 
\end{align*}
  
The random walk makes linear progress in $\widetilde{S}_h \times \RR$
at rate $\ell_3$, with exponential decay.  The random variable $d_{\widetilde{S}_h}( x_0, \gamma(t_0) )$ has a distribution with exponential tails. By \Cref{cor:bscor}, the random variable $h_\gamma(t_0)$ also has 
exponential tails. Hence, as $N \to \infty$, $\tfrac{d_{\widetilde{S}_h}( w_N x_0, \gamma(t_N) )}{N} \to 0$ and $\tfrac{h_\gamma(t_N)}{N} \to 0$.
Therefore
\[ \lim_{t \to \infty} \frac{1}{N} d_{\widetilde{S}_h \times \RR}
\left( \overline\gamma(\overline{t}_0),
\overline\gamma(\overline{t}_N) \right) \ge \ell_3 > 0, \]
as required.
\end{proof}

We now record the following elementary properties of visual measure in
$\ws_h$.

\begin{proposition}\label{prop:small angle}
Suppose that $x_0$ is a basepoint for $\HH^2$, and $\ell$ a geodesic. Suppose that $p$ is the closest point of $\ell$ to $x_0$.  For $r \le \tfrac{1}{2}$, let $\gamma$ be a geodesic which intersects an $r$-neighborhood in $T^1(\HH^2)$ of the tangent vector to $\ell$ at $p$. Then the endpoints of
$\ell$ and $\gamma$ in $\partial_\infty \HH^2$ are distance at most $8 r$ apart.
\end{proposition}

\begin{proof}
Suppose $\ell$ passes through $x_0$ so that $p = x_0$. Suppose that $\gamma$ also passes through $x_0$ and makes an angle at most $r$ with $\ell$. Then the distance between their
endpoints in the visual metric on $\partial_\infty \HH^2$ from $x_0$, is at most $r$.

\medskip
Now suppose that 
\begin{itemize}
    \item $\ell$ and $\gamma$ are disjoint, and
    \item $x_0$ is the midpoint of the shortest geodesic segment $\alpha$ between $\ell$ and $\gamma$.
\end{itemize}
Let $\beta$ be the perpendicular bisector to
$\alpha$ at $x_0$.  
Choose an endpoint in $\partial_\infty \HH^2$ of $\gamma$ and let $q$ be the nearest point projection of that endpoint to $\beta$. In fact, the geodesic between the endpoint and $q$ extends to a bi-infinite geodesic perpendicular to $\beta$ at $q$. 
Consider the right angled triangle with vertices $x_0$, $q$ and the chosen endpoint of $\gamma$ as an ideal vertex. 
Let $t$ be the distance from $x_0$ to $q$.  By
\Cref{prop:projection interval}, $t \ge \log \tfrac{2}{r}$.  The angle
at $x_0$ is half the angle between the endpoints of $\ell$ and
$\gamma$ at $x_0$.  By the tangent formula for right angled
hyperbolic triangles, $\tan(\theta / 2) = 1 / \sinh(t)$.  Using
the elementary bounds that $x \le \tan(x) \le 2x$ for $0 \le x \le 1$,
and the bound on $t$ in terms of $r$, we get $x \le 8r/(4 - r^2)$.  As we
have assumed that $r \le \tfrac{1}{2}$, this implies that
$\theta/2 \le 4r$, as required.

\medskip
Finally, suppose that in the case that the geodesics intersect, $x_0$ is not the point of intersection, and if the geodesics do not intersect, $x_0$ is not the midpoint of the shortest geodesic segment between them. In the former case, we may fix some isometry $g$ that moves $x_0$ to the intersection point. Restricted to small intervals about the endpoints of $g^{-1}\gamma$ and their images by $g$, the isometry $g$ is a contraction for the visual metric from $x_0$. 
In the case that $\gamma$ and $\ell$ do not intersect, we consider an isometry that moves $x_0$ to the midpoint of the shortest geodesic segment between $\gamma$ and $\ell$. Again, the isometry restricts to a contraction near the endpoints of $g^{-1}(\gamma)$, and so the result follows.
\end{proof}

We now prove \Cref{theorem:effective hitting}.

\begin{proof}[Proof (of \Cref{theorem:effective hitting})]
We fix a basepoint $x_0$ in $\widetilde{S}_h$ and give
$\partial \widetilde{S}_h$ the angular metric from $x_0$. 
Suppose that $( w_n x_0 )_{n \in \ZZ}$ is a typical bi-infinite sample
path of the random walk on $\pi_1(S)$ generated by $\mu$.  Let
$\gamma$ be the bi-infinite geodesic determined by the limit points of the sample path.  Let
$\gamma(t_n)$ be the nearest point projection of $w_n x_0$ to $\gamma$
in $\widetilde{S}_h$.  The value of the height function
$h_\gamma(t_n)$ is determined by the distance from the tangent vector
at $\gamma(t_n)$ to the invariant laminations in $T^1(S_h)$.

\medskip
Let $\theta_\LL$ be the constant from the definition of the height
function, 
\cite{gmpu}*{Section 3.2.2}
By \Cref{prop:small angle}, if $\gamma(t)$ is distance at most
$r \le \theta_\LL \le 1$ in $T^1(S_h)$ from a leaf $\ell$ of one of the
invariant laminations, then the endpoints of $\gamma$ are distance at
most $8r$ from the endpoints of $\ell$ in $\partial S_h$.  We shall
write $\overline{\Lambda}$ for the union of the extended invariant
laminations, and $\partial \overline{\Lambda}$ for the boundary in
$\partial \widetilde{S}_h \times \partial \widetilde{S}_h$.  Then
\begin{align*}
  \PP \left( d_{T^1(S_h)}( \gamma^1(t_0), \overline{\Lambda}^1 ) \le r
  \right) & \le \nu \times \rnu \left( N_{8r}(\partial
            \overline{\Lambda} )
            \right),  \\
\intertext{and by \Cref{cor:bscor} there are constants
  $A_1$ and $\beta > 0$ such that}
  \PP \left( d_{T^1(S_h)}( \gamma^1(t_0), \overline{\Lambda}^1 ) \le r
  \right) & \le A_1 r^{\beta}. 
\end{align*}
Action by the shift map then gives us the same result for
all $n$.

\medskip
Using the relation 
between distance $r$ in
$T^1(S_h)$ and the value of the height function $R$, there is a
constant $A_2$ such that 
\[ \PP( | h_\gamma(t_n) | \ge R ) \le A_2 e^{- \beta k^R} .  \]

We shall choose the basepoint in $\widetilde{S}_h \times \RR$ to be
$\overline{x}_0 = \iota(x_0)$.  Let $\overline{\gamma}$ be the
geodesic in $\widetilde{S}_h \times \RR$ determined by $\gamma$.  Let
$\tau_\gamma(t_n)$ be the corresponding point on the test path, and
let $\overline{\gamma}(p_n)$ be the nearest point projection of
$\tau_\gamma(t_n)$ to $\overline{\gamma}$ in
$\widetilde{S}_h \times \RR$.

\medskip
By \Cref{prop:test path linear progress}, the distance between
$\tau_\gamma(t_0)$ and $\tau_\gamma(t_n)$ grows linearly in $n$.  As
$\tau_\gamma$ is contained in an $L$-neighborhood of 
$\overline{\gamma}$, the distance between
$\overline{\gamma}(p_0)$ and $\overline{\gamma}(p_n)$ also grows linearly
in $n$.

\medskip
Every point of $\overline{\gamma}([p_0, p_{N}]) \setminus N_R(S_0)$ is
within distance $B_2 = B + 2L$ of a point $\overline{\gamma}(p_k)$,
and the proportion of such points is at most
$\PP( | h_\gamma(t_0) | \ge R )$.  Therefore
\[ \lim_{N \to \infty} \frac{1}{|[p_0, p_N]|} \left| \overline{\gamma}([p_0,
p_{N}]) \setminus N_R(S_0) \right| \le \frac{B_2}{\ell_3} A_2 e^{-
  \beta k^R} . \]
As $p_N$ tends to infinity as $N$ tends to infinity, the result
follows.
\end{proof}


\bibliography{references}


\medskip

School of Mathematics and Statistics, University of Glasgow \href{mailto:Vaibhav.Gadre@glasgow.ac.uk}{Vaibhav.Gadre@glasgow.ac.uk}

\medskip

CUNY College of Staten Island and CUNY Graduate Center
\href{mailto:joseph.maher@csi.cuny.edu}{joseph.maher@csi.cuny.edu}

\medskip 
Queen's University at Kingston \href{mailto:catherine.pfaff@gmail.com}{catherine.pfaff@gmail.com}

\medskip

Department of Mathematics,
University of Wisconsin--Madison
\href{mailto:caglar@math.wisc.edu}{caglar@math.wisc.edu}

\end{document}